\documentclass[12pt]{article}

\usepackage{amsmath,amsthm, amssymb, latexsym}
\usepackage{amsfonts}
\usepackage{graphicx}
\usepackage{multicol}
\usepackage[T1]{fontenc}
\usepackage{lmodern}
\usepackage{caption}
\usepackage{color}
\usepackage{graphicx}
\usepackage{amsmath}
\usepackage{multicol}
\usepackage{mathrsfs} 
\usepackage{float}
\usepackage{bbold}
\usepackage{color}
\usepackage{subfigure}
\usepackage{url}
\usepackage{tabularx}
\usepackage{xcolor}
\usepackage{grffile}
\usepackage[english]{babel}
\usepackage{enumerate}\numberwithin{equation}{section}
\usepackage{fullpage}
 \normalbaselines
\newtheorem{Theorem}{Theorem}[section]
\newtheorem{Definition}[Theorem]{Definition}
\newtheorem{Proposition}[Theorem]{Proposition}
\newtheorem{Lemma}[Theorem]{Lemma}
\newtheorem{Corollary}[Theorem]{Corollary}

\newtheorem{Remark}[Theorem]{Remark}

\newcommand{\N}{\mathbb{N}}
\newcommand{\Z}{\mathbb{Z}}
\newcommand{\R}{\mathbb{R}}

\newcommand{\norm}[1]{\lVert #1 \rVert}	
\newcommand{\ra}{\rightarrow}
\newcommand\restr[2]{{
		\left.\kern-\nulldelimiterspace 
		#1 
		\vphantom{\big|} 
		\right|_{#2} 
}}
\allowdisplaybreaks[4]
\begin{document}
	
	\title{Central limit theorems for stationary random fields under weak dependence with application to ambit and mixed moving average fields \\
	}
	\author{Imma Valentina Curato, Robert Stelzer and Bennet Str{\"o}h\footnote{Ulm University, Institute of Mathematical Finance, Helmholtzstra\ss e 18, 89069 Ulm, Germany. Emails: imma.curato@uni-ulm.de, robert.stelzer@uni-ulm.de, bennet.stroeh@uni-ulm.de. }}

	\maketitle
	
	\textwidth=160mm \textheight=225mm \parindent=8mm \frenchspacing
	\vspace{3mm}

	\begin{abstract}
We obtain central limit theorems for stationary random fields employing a novel measure of dependence called $\theta$-lex weak dependence. We show that this dependence notion is more general than strong mixing, i.e., it applies to a broader class of models. Moreover, we discuss hereditary properties for $\theta$-lex and $\eta$-weak dependence and illustrate the possible applications of the weak dependence notions to the study of the asymptotic properties of stationary random fields.
Our general results apply to mixed moving average fields (MMAF in short) and ambit fields. We show general conditions such that MMAF and ambit fields, with the volatility field being an MMAF or a $p$-dependent random field, are weakly dependent. For all the models mentioned above, we give a complete characterization of their weak dependence coefficients and sufficient conditions to obtain the asymptotic normality of their sample moments. Finally, we give explicit computations of the weak dependence coefficients of MSTOU processes and analyze under which conditions the developed asymptotic theory applies to CARMA fields.

	\end{abstract}
	
	\noindent
	{\it MSC 2020: primary 60G10, 60G57, 60G60,62M40; 
		secondary 62F10, 62M30.}  
	\\
	\\
	{\it Keywords: stationary random fields, weak dependence, central limit theorems, mixed moving average fields, CARMA fields, ambit fields.}
	
\section{Introduction}
Many modern statistical applications consider the modeling of phenomena evolving in time and/or space with either a countable or uncountable index set.
To this end, we can employ random fields on $\Z^m$ or $\R^m$ which are defined, for example, as solutions of recurrence equations, e.g., in \cite{DT2007}, or stochastic partial differential equations \cite{B2019b, CK2015, P2018}. Noticeable examples of the latter come from the class of ambit and mixed moving average fields. 

The mixed moving average fields, MMAF in short, are defined as 
\begin{equation}
	\label{begin}
	X_t=\int_{S} \int_{\R^m} f(A,t-s) \, \Lambda(dA,ds),\,\,\, t \in \R^m,
\end{equation}
where $S$ is a Polish space, $f$ a deterministic function called kernel and $\Lambda$ a L\'evy basis.
The above model encompasses Gaussian and non-Gaussian random fields by choosing the L\'evy basis $\Lambda$.
Ambit fields are defined by considering an additional multiplicative random function in the integrand (\ref{begin}) called volatility or intermittency field. However, an ambit field is typically defined without the variable $A$ in its kernel function. We refer the reader to \cite{BBV2018} for a comprehensive introduction to ambit fields which provide a rich class of spatio-temporal models on $\R \times \R^m$. Overall, MMAF and ambit fields are used in many applications throughout different disciplines, like geophysics \cite{H2002}, brain imaging \cite{JN2013}, physics \cite{BS2004}, biology \cite{BJJS2007,BNS2007}, economics and finance \cite{BBV2010,BBV2015,BM2017,NV2015,NV2017}. 
The generality and flexibility of these models motivate an in-depth analysis of their asymptotic properties. 

Central limit theorems for moving average fields, a sub-class of MMAF, are discussed in \cite{B2019} where the author shows the asymptotic normality of its sample mean and autocovariance. However, we do not pursue this approach because it is not directly applicable to the study of higher-order sample moments. Under strong mixing conditions, several central limit theorems for stationary random fields are available in the literature, see \cite{B1982, C1991, D1998, IL89, M1999, N1988}. In \cite{IL89}, we also find some of the first results related to the analysis of central limit theorems for autocovariance functions. Note that, in general, caution must be used when applying some of the classical strong mixing dependence notions to random fields. We refer to \cite{B1989} and \cite[Chapter 29]{B2007} for a thorough investigation on this point. The above said, for an MMAF on $\R$, i.e., a mixed moving average process, several difficulties already arise in showing that it is strongly mixing, see \cite{CS2018}. Usually, strong mixing is established by using a Markovian representation and showing geometric ergodicity of it. In turn, this often requires smoothness conditions on the driving random noise, and it is well-known that even autoregressive processes of order one are not strongly mixing when the distribution of the noise is not sufficiently regular, see \cite{A1984}. For $m \geq 2$, a Gaussian MMAF on $\R^m$ satisfying the conditions of \cite[Theorem $7$, pg. 73]{R1985} is strong mixing. However, for general driving L\'evy bases, no results in the literature can be found regarding the strong mixing of MMAF. Sharp central limit results for stationary random fields can also be obtained under the dependence notion of association (see \cite{BS2007, N1980} for a comprehensive introduction on this topic). However, in this case, central limit theorems for MMAF hold just under restrictive conditions on the kernel function $f$ in (\ref{begin}), see, e.g., \cite[Theorem $3.27$]{BS2007}. Moreover, association is inherited only under monotone functions, restricting the possible extension of its related asymptotic theory.

Concerning purely temporal ambit fields, i.e., L\'evy semistationary processes, in \cite{BCP2011,BPS2014,BHP2018} the authors obtain infill asymptotic results for this class of processes, that is, under the assumption that the number of observations in a given interval approaches to infinity. For ambit fields on $\R \times \R^m$ with $m \geq 1$ where $\Lambda$ is of Gaussian type and the volatility field is independent of $\Lambda$, the asymptotic behavior of the lattice power variation of the field is studied in \cite{P2014}. We notice that in the literature, there are no asymptotic results for partial sums of ambit fields when the number of observations approaches infinity without infill asymptotics.

We are interested in studying the asymptotic behavior of the partial sums (and of higher-order sample moments) of MMAF and ambit fields in general, i.e., without imposing regularity conditions on the driving L\'evy basis $\Lambda$ apart from moment conditions. To do so, we apply the $\eta$-weak dependence as defined in \cite{DMT2008} and a new notion of dependence called $\theta$-lex weak dependence. Although all the examples of our theory come from the model classes mentioned above, we want to emphasize that we present general central limit theorem results that apply to different stationary random fields. 

To introduce the notion of $\theta$-lex weak dependence, let us start with a brief digression into the notions of  $\eta$ and $\theta$-weak dependence for stochastic processes defined in \cite{DL1999} and \cite{DD2003}, respectively. $\eta$-weak dependence is typically associated with the study of non-causal processes, whereas $\theta$-weak dependence is related to the analysis of the causal ones. Central limit theorems for $\theta$-weakly dependent processes hold under weaker conditions compared to results for $\eta$-weakly dependent processes (different demands on the decay rate of the $\eta$ and $\theta$-coefficients as determined in \cite[Theorem 2.2]{DW2007} and \cite[Theorem 2]{DD2003}). We have that the definition of $\eta$ and $\theta$-weak dependence can be easily extended to the random field case by following \cite[Remark 2.1]{DDLLLP2008}. However, just for $\eta$-weakly dependent random fields, asymptotics of the partial sums of stationary random fields have been so far analyzed in \cite{DMT2008}. We aim to determine a central limit theorem that improves the results obtained in \cite{DMT2008}. We achieve this by defining the notion of $\theta$-lex-weak dependence, which is a modification of the original definition of $\theta$-weak dependence. We show that for $\theta$-lex-weakly dependent random fields, the sufficient conditions of a very powerful central limit theorem from Dedecker \cite{D1998} hold. Moreover, we obtain hereditary properties for $\theta$-lex and $\eta$-weakly dependent random fields, which allow us to easily extend the asymptotic results under weak dependence to the study of higher-order sample moments. 
We then investigate the relationship between $\theta$-lex weak dependence and strong mixing. We prove that for random fields defined on $\Z^m$, $\theta$-lex weak dependence is a more general notion of dependence than $\alpha_{\infty,1}$-mixing as defined in Dedecker \cite{D1998}, i.e., it applies to a broader class of models. In the case of processes, we also show that $\theta$-lex weak dependence is a more general notion of dependence than $\alpha$-mixing as defined in \cite{B2007}, see for more details Section \ref{sec2-2}.

Let us now look at the class of MMAF. We distinguish in our theory between influenced and non-influenced MMAF, see Definition \ref{definition:influenced}. Influenced MMAF represent a possible extension of causal mixed moving average processes, see \cite[Section 3.2]{CS2018}, to random fields.
Hence, we show that influenced MMAF are $\theta$-lex-weakly dependent and that non-influenced MMAF are $\eta$-weakly dependent with coefficients computable in terms of the kernel function $f$ and the characteristic quadruplet of the L\'evy basis $\Lambda$. From this, we notice that in the case of influenced MMAF, the conditions ensuring asymptotic normality of the partial sums of $X$ are weaker-- in terms of the decay rate of the weak dependence coefficients-- in comparison with the one obtained for non-influenced MMAF. We then observe a parallel between our results and the one obtained for causal and non-causal mixed moving average processes \cite{CS2018}. Moreover, we exploit the hereditary properties of $\eta$ as well as $\theta$-lex-weak dependence and obtain conditions for the sample moments of order $p$ with $p \geq 1$ to be asymptotic normally distributed. Finally, we give explicit computations for mixed spatio-temporal Ornstein-Uhlenbeck processes \cite{NV2017}, also called MSTOU processes, and L\'evy-driven CARMA fields \cite{BM2017,P2018}. In particular, our calculations in the case of the MSTOU processes show that it is possible to determine the asymptotic normality of the generalized method of moments estimator, GMM in short, proposed in \cite{NV2017}.

At last, we apply our theory to ambit fields. We assume that the volatility field is an MMAF or a $p$-dependent random field which is independent of the L\'evy basis $\Lambda$. Under these assumptions, we show that homogeneous and stationary ambit fields are $\theta$-lex-weakly dependent and give sufficient conditions on the $\theta$-lex-coefficients to ensure asymptotic normality of the sample moments. 

The paper is structured as follows. In Section \ref{sec2}, we introduce $\eta$-weak dependence and the novel $\theta$-lex-weak dependence. In Section \ref{sec2-2} we state central limit theorems for $\theta$-lex weakly dependent random fields in an ergodic, non-ergodic, and multivariate setting. Additionally, we analyze the relationship between $\theta$-lex weak dependence and strong mixing and provide some insight into possible functional extensions of the central limit theorem. 
In Section \ref{sec3}, we discuss the weak dependence properties of MMAF. We first give a comprehensive introduction to L\'evy bases and its related integration theory, leading to the formal definition of an MMAF.
We discuss conditions on MMAF to be $\theta$-lex or $\eta$-weakly dependent and their related sample moment asymptotics.
In Section \ref{sec3-7}, we apply the developed theory to MSTOU processes and give explicit conditions assuring their sample moments' asymptotic normality under a Gamma distributed mean reversion parameter. We conclude Section \ref{sec3} by giving conditions under which the developed asymptotic theory can be applied to L\'evy-driven CARMA fields. In Section \ref{sec4}, we discuss weak dependence properties and related limit theorems for ambit fields. Section \ref{sec5} contains the detailed proofs of most of the results presented in the paper.

\section{Weak dependence and central limit theorems}
\label{sec2}

\subsection{Notations}
$\N_0$ denotes the set of non-negative integers, $\N$  the set of positive integers, and $\R^+$ the set of the non-negative real numbers.  For $x\in \R^d$, $x^(j)$ denotes the $j$-th coordinate of $x$, $\norm{x}$ its Euclidean norm and we define $|x|=\norm{x}_\infty=\max_{j=1,\ldots,d}|x^{(j)}|$. For $d, k \geq 1$ and $F:\R^d\ra \R^k$, we define $\norm{F}_\infty=\sup_{t\in\R^d}\norm{F(t)}$. Let $A \in M_{n \times d}(\R)$,  $A^{\prime}$ denote the transpose of the matrix $A$.

In the following Lipschitz continuous is understood to mean globally Lipschitz. For $u,n\in\N$, let $\mathcal{G}_u^*$ be the class of bounded functions from $(\R^n)^u$ to $\R$ and $\mathcal{G}_u$ be the class of bounded, Lipschitz continuous functions from $(\R^n)^u$ to $\R$ with respect to the distance $\sum_{i=1}^{u}\norm{x_i-y_i}$, where $x,y\in(\R^n)^u$. For $G\in\mathcal{G}_u$ we define
\begin{gather*}
	Lip(G)=\sup_{x\neq y}\frac{|G(x)-G(y)|}{\norm{x_1-y_1}+\ldots+\norm{x_u-y_u}}.
\end{gather*} 
We assume that all random elements in this paper are defined on a given complete probability space $(\Omega, \mathcal{F},P)$.  $\norm{\cdot}_p$ for $p>0$ denotes throughout the $L^p$-norm of a random element. For a random field $X=(X_t)_{t\in\R^m}$ and a finite set $\Gamma\subset\R^m$ with $\Gamma=(i_1,\ldots,i_u)$, we define the vector $X_\Gamma=(X_{i_1},\ldots,X_{i_u})$. $A\subset B$ denotes a not necessarily proper subset $A$ of a set $B$, $|B|$ denotes the cardinality of $B$ and $dist(A,B)=\inf_{i\in A, j\in B} \norm{i-j}_\infty$ indicates the distance of two sets $A,B\subset\R^m$.

Hereafter, we often use the lexicographic order on $\R^m$. For distinct elements $y=(y_1,\ldots,y_m)\in\R^m$ and $z=(z_1,\ldots,z_m)\in\R^m$ we say $y<_{lex}z$ if and only if $y_1<z_1$ or $y_p<z_p$ for some $p\in\{2,\ldots,m\}$ and $y_q=z_q$ for $q=1,\ldots,p-1$. Moreover, $y\leq_{lex}z$ if $y<_{lex}z$ or $y=z$ holds. Finally, let us define the sets $V_t=\{s\in\R^m:s<_{lex}t\}\cup\{t\}$ and $V_t^h=V_t\cap \{s\in\R^m: \norm{t-s}_\infty\geq h \}$ for $h>0$. The definitions of the sets $V_t$ and $V_t^h$ are also used when referring to the lexicographic order on $\Z^m$.

\subsection{Weak dependence properties}
\label{sec2-1}

\begin{Definition}\label{thetaweaklydependent}
	Let $X=(X_t)_{t\in\R^m}$ be an $\R^n$-valued random field. Then, $X$ is called $\theta$-lex-weakly dependent if
	\begin{gather*}
		\theta(h)=\sup_{u\in\N}\theta_{u}(h) \underset{h\ra\infty}{\longrightarrow} 0,
	\end{gather*} 
	where
	\begin{align*}
		\theta_{u}(h)=\sup\bigg\{&\frac{|Cov(F(X_{\Gamma}),G(X_j))|}{\norm{F}_{\infty}Lip(G)},
		F\in\mathcal{G}^*_u,G\in\mathcal{G}_1,j\in\R^m, \Gamma \subset V_j^h, |\Gamma|= u  \bigg\}.
	\end{align*}
	We call $(\theta(h))_{h\in\R^+}$ the $\theta$-lex-coefficients. 
\end{Definition}

\begin{Remark}
	Our definition of $\theta$-lex-weak dependence differs from the $\theta$-weak dependence definition for random fields given in \cite[Remark 2.1]{DDLLLP2008}. In fact, instead of considering the covariance of two arbitrary finite-dimensional samples $X_\Gamma$ and $X_{\tilde{\Gamma}}$, for $\Gamma, \tilde{\Gamma} \subset \R^m$, we control the covariance of a finite-dimensional sample $X_\Gamma$ and an arbitrary one point sample $X_j$. Secondly, assuming that all points in the sampling set $\Gamma$ are lexicographically smaller than $j$, we provide order in the sampling scheme.
	
	For $m=1$, i.e., in the process case, our definition of $\theta$-lex-weak dependence coincides with the definition of $\theta$-weak dependence given in \cite{DD2003}. 
\end{Remark}

\begin{Definition}[{\cite[Definition 2.2 and Remark 2.1]{DDLLLP2008}}]\label{etaweaklydependent}\
	Let $X=(X_t)_{t\in\R^m}$ be an $\R^n$-valued random field. Then, $X$ is called $\eta$-weakly dependent if
	\begin{gather*}
		\eta(h)=\sup_{u,v\in\N}\eta_{ u,v}(h) \underset{h\ra\infty}{\longrightarrow} 0,
	\end{gather*}
	where
	\begin{align*}
		\eta_{ u,v}(h) =\sup\bigg\{&\frac{|Cov(F(X_{\Gamma}),G(X_{\tilde\Gamma}))|}{u\norm{G}_{\infty}Lip(F)+v\norm{F}_{\infty}Lip(G)},
		\\ &F\in\mathcal{G}_u,G\in\mathcal{G}_v, \Gamma,\tilde\Gamma\subset\R^m,|\Gamma|= u,  |\tilde\Gamma|= v,dist(\Gamma,\tilde\Gamma)\geq h \bigg\}.
	\end{align*}
	We call $(\eta(h))_{h\in\R^+}$ the $\eta$-coefficients.
\end{Definition}

Let $(X_t)_{t\in\R^m}$ be $\theta$-lex- or $\eta$-weakly dependent and $h:\R^n\ra\R^k$ be an arbitrary Lipschitz function, then the field $(h(X_t))_{t\in\R^m}$ is also $\theta$-lex- or $\eta$-weakly dependent. The latter can be readily checked based on Definition \ref{thetaweaklydependent} and \ref{etaweaklydependent}.
In the next proposition, we give conditions for hereditary properties of functions that are only locally Lipschitz continuous. The proof of the result below is analogous to Proposition 3.2 \cite{CS2018}.

\begin{Proposition}\label{proposition:mmathetahereditary}
	Let $X=(X_t)_{t\in\R^m}$ be an $\R^n$-valued stationary random field and assume that there exists a constant $C>0$ such that $E[\norm{X_0}^p]\leq C$, for $p>1$. Let $h:\R^n\ra\R^k$ be a function such that $h(0)=0, h(x)=(h_1(x),\ldots,h_k(x))$ and 
	\begin{gather*}
		\norm{h(x)-h(y)}\leq c\norm{x-y}(1+\norm{x}^{a-1}+\norm{y}^{a-1}),
	\end{gather*} 
	for $x,y\in\R^n$, $c>0$ and $1\leq a<p$. Define $Y=(Y_t)_{t\in\R^m}$ by $Y_t=h(X_t)$. If $X$ is $\theta$-lex or $\eta$-weakly dependent, then $Y$ is $\theta$-lex or $\eta$-weakly dependent respectively with coefficients
	\begin{gather*}
		\theta_Y(h)\leq\mathcal{C}\theta_X(h)^{\frac{p-a}{p-1}} \,\,\text{ or } \,\,
		\eta_Y(h)\leq\mathcal{C}\eta_X(h)^{\frac{p-a}{p-1}}  
	\end{gather*}
	for all $h>0$ and a constant $\mathcal{C}$ independent of $h$.
\end{Proposition}

\subsection{Mixing properties}
\label{sec2-2}

Let $\mathcal{M}$ and $\mathcal{V}$ be two sub-$\sigma$-algebras of $\mathcal{F}$. We define the strong mixing coefficient of Rosenblatt \cite{R1956}
\begin{align*}
	\alpha(\mathcal{M},\mathcal{V})=\sup\{|P(M)P(V)-P(M\cap V)|,M\in\mathcal{M},V\in\mathcal{V}\}.
\end{align*}
A random field $X=(X_t)_{t\in\Z^m}$ is said to be $\alpha_{u,v}$-mixing for $u,v\in\N\cup\{\infty\}$ if 
\begin{align*}
	\alpha_{u,v}(h)=\sup\left\{ \alpha(\sigma(X_\Gamma),\sigma(X_{\tilde{\Gamma}})),\, \Gamma,\tilde{\Gamma}\subset\R^m,|\Gamma|\leq u, |\tilde{\Gamma}|\leq v, dist(\Gamma,\tilde{\Gamma})\geq h \right\}
\end{align*}
converges to zero as $h\rightarrow \infty$. Moreover, for $m=1$, the stochastic process $X$ is said to be $\alpha$-mixing if 
\begin{align*}
	\alpha(h)=\alpha\left( \sigma(\{X_s, s\leq0\}),\sigma(\{X_s, s\geq h\})\right)
\end{align*}
converges to zero as $h\rightarrow \infty$. Clearly, we have that $\alpha(h) \leq \alpha_{\infty,\infty}(h)$ for $m=1$. For a comprehensive discussion on the coefficients $\alpha_{u,v}(h)$, $\alpha(h)$ and their relation to other strong mixing coefficients we refer to \cite{B2007, BCS2020, D1998}.\\
The following proposition establishes a relationship between the $\theta$-lex-coefficients and the mixing coefficients $\alpha(h)$ and $\alpha_{\infty,1}(h)$.

\begin{Proposition}\label{proposition:thetavsalpha}
	Let $X=(X_t)_{t\in\Z^m}$ be a stationary real-valued random field such that $E[\norm{X_0}^q]<\infty$ for some $q>1$. Then, for all $h\in\R^+$ and $m=1$, we have that
	$$
	\theta(h)\leq 2^{\frac{2q-1}{q}} (\alpha(h))^{\frac{q-1}{q}} \norm{X_0}_q.
	$$
	Moreover, for all $h\in\R^+$ and $m \geq 1$
	$$\theta(h)\leq 2^{\frac{2q-1}{q}} (\alpha_{\infty,1}(h))^{\frac{q-1}{q}} \norm{X_0}_q.$$
\end{Proposition}
\begin{proof}
	See Section \ref{sec5-1}.
\end{proof}

\begin{Remark}
	If a stationary real-valued random field admits all finite moments, then Proposition \ref{proposition:thetavsalpha} ensures that $\theta(h)\leq C \alpha(h)$ for $m=1$ and $\theta(h)\leq C \alpha_{\infty,1}(h)$ for all $m\geq1$, where $C>0$ is a constant independent of $h$.
\end{Remark}

\begin{Proposition}\label{proposition:AR1notmixingbuttheta}
	Let $(\xi_k)_{k\in\Z}$ be a sequence of independent random variables such that $\xi_k\sim Ber(\frac{1}{2})$ for all $k\in\Z$. Then, the stationary process $X_t=\sum_{j=0}^\infty 2^{-j-1}\xi_{t-j}$ for $t\in\Z$ is
	$\theta$-lex-weakly dependent but neither $\alpha$-mixing nor $\alpha_{\infty,1}$-mixing.
	
\end{Proposition}
\begin{proof}
	See Section \ref{sec5-1}.
\end{proof}

\begin{Remark}
	Proposition \ref{proposition:thetavsalpha} shows that every stationary $\alpha$-mixing stochastic process and every stationary $\alpha_{\infty,1}$-mixing random field with finite $q$-th moments are $\theta$-lex-weakly dependent. On the other hand, Proposition \ref{proposition:AR1notmixingbuttheta} shows that there exists a stationary $\theta$-lex-weakly dependent process with finite variance, that is neither $\alpha$-mixing nor $\alpha_{\infty,1}$-mixing. Therefore, $\theta$-lex-weak dependence is a more general notion of dependence than $\alpha$- and $\alpha_{\infty,1}$-mixing.
\end{Remark}

We cite for completeness the results available in the literature regarding the relationship between $\eta$-weak dependence and $\alpha$-mixing. For integer-valued processes, the authors show in \cite[Proposition 1]{DFL2012}  that $\eta$-weak dependence implies $\alpha$-mixing. Moreover, Andrews \cite{A1984} gives an example of an $\eta$-weakly dependent process that is not $\alpha$-mixing.

\subsection{Central limit theorems for $\theta$-lex-weakly dependent random fields}
\label{sec2-3}

In the theory of stochastic processes, one of the typical ways to prove central limit type results is to approximate the process of interest by a sequence of martingale differences. This approach was first introduced by Gordin \cite{G1969}. However, the latter does not apply to high-dimensional random fields as successfully as to processes. This unpleasant circumstance has been known among researchers for almost 40 years, as Bolthausen \cite{B1982} noted that martingale approximation appears a difficult concept to generalize to dimensions greater or equal than two.

For stationary random fields $X=(X_t)_{t \in \Z^m}$, Dedecker derived a central limit  result in \cite{D1998} under the projective criterion
\begin{gather}\label{equation:l1criteria}
	\sum_{k\in V_0^1}|X_kE[X_0|\mathcal{F}_{\Gamma(k)}]|\in L^1, \,\,\, \textrm{for $\mathcal{F}_{\Gamma(k)}=\sigma(X_k :k\in V_0^{|k|})$}.
\end{gather}
This condition is weaker than a martingale-type assumption and provides optimal results for mixing random fields.
Early use of such a projective criterion can be found in the central limit theorems for stationary processes derived in \cite{DM2002, PU2006}.

We show in this section that (\ref{equation:l1criteria}) is also fulfilled by appropriate $\theta$-lex-weakly dependent random fields. 

In the following, by stationarity we mean stationarity in the strict sense. Let $\Gamma$ be a subset of $\Z^m$. We define $\partial\Gamma=\{i\in\Gamma: \exists j \notin \Gamma: \norm{i-j}_\infty=1\}$. Let $(D_n)_{n\in\N}$ be a sequence of finite subsets of $\Z^m$ such that 
\begin{gather*}
	\lim_{n\rightarrow\infty} |D_n|=\infty \text{ and }\lim_{n\rightarrow\infty} \frac{|\partial D_n|}{|D_n|}=0.
\end{gather*}

\begin{Theorem}\label{theorem:clt}
	Let $X=(X_t)_{t\in\Z^m}$ be a stationary centered real-valued random field such that $E[|X_t|^{2+\delta}]<\infty$ for some $\delta>0$. 
	Additionally, assume that $\theta(h)\in \mathcal{O}(h^{-\alpha})$ with $\alpha>m(1+\frac{1}{\delta})$. Define
	\begin{gather*}
		\sigma^2=\sum_{k\in\Z^m}E[X_0X_k|\mathcal{I}],
	\end{gather*}
	where $\mathcal{I}$ is the $\sigma$-algebra of shift invariant sets as defined in \cite[Section 2]{D1998} (see \cite[Chapter 1]{K1985} for an introduction to ergodic theory). Then, $\sigma^2$ is finite, non-negative and
	\begin{gather}\label{eq:clt}
		\frac{1}{|D_n|^{\frac{1}{2}}}\sum_{j\in D_n}X_j\underset{n\ra\infty}{\xrightarrow{\makebox[2em][c]{d}}}\varepsilon \sigma,
	\end{gather}
	where $\varepsilon$ is a standard normally distributed random variable which is independent of $\sigma^2$.
\end{Theorem}
\begin{proof}
	See Section \ref{sec5-1}.
\end{proof}	

The multivariate extension of Theorem \ref{theorem:clt}, appearing below, is obtained by applying the Cram\'er-Wold device and noting that linear functions are Lipschitz.

\begin{Corollary}\label{corollary:ergodicclt}
	Let $X=(X_t)_{t\in\Z^m}$ be a stationary ergodic centered $\R^n$-valued random field such that $E[\norm{X_t}^{2+\delta}]<\infty$ for some $\delta>0$. 
	Additionally, let us assume that $\theta(h)\in \mathcal{O}(h^{-\alpha})$ with $\alpha>m(1+\frac{1}{\delta})$. Then,
	\begin{gather*}
		\Sigma=\sum_{k\in\Z^m}E[X_0X_k']
	\end{gather*}
	is finite, positive definite and
	\begin{gather*}
		\frac{1}{|D_n|^{\frac{1}{2}}}\sum_{j\in D_n}X_j\underset{n\ra\infty}{\xrightarrow{\makebox[2em][c]{d}}}N(0,\Sigma),
	\end{gather*}
	where $N(0,\Sigma)$ denotes the multivariate normal distribution with mean $0$ and covariance matrix $\Sigma$.
\end{Corollary}

\begin{Remark}
	Using Proposition \ref{proposition:thetavsalpha} the condition $\theta(h)\in \mathcal{O}(h^{-\alpha})$ with $\alpha>m(1+\frac{1}{\delta})$ in Theorem \ref{theorem:clt} can be replaced by $\alpha_{\infty,1}(h)\in \mathcal{O}(h^{-\beta})$ or $\alpha(h) \in \mathcal{O}(h^{-\beta})$ with $\beta>m(1+\frac{2}{\delta})$.\\ 
	For real valued stochastic processes, the latter condition represents the sharpest one available for $\alpha$-mixing coefficients (see \cite[Theorem 25.70]{B2007}).
\end{Remark}

\begin{Remark}
	It is natural to ask for conditions ensuring a functional extension of Theorem \ref{theorem:clt}. As a matter of fact, results of this kind are strongly related to the following $L^p$-projective criterion
	\begin{gather}\label{equation:lpcriteria}
		\sum_{k\in V_0}E[|X_kE[X_0|\mathcal{F}_{V_0^{|k|}}]|^p]<\infty,~ p\in[1,\infty],
	\end{gather}
	where $\mathcal{F}_{\Gamma}=\sigma(X_k,k\in\Gamma)$.\\
	For $m=1$, Dedecker and Rio show in \cite[Theorem]{DR2000} that if (\ref{equation:lpcriteria}) holds for $p=1$, then a functional central limit theorem holds.\\
	In the general case $m>1$, Dedecker proved in \cite[Theorem 1]{D2001} a functional central limit theorem if (\ref{equation:lpcriteria}) holds for $p>1$.\\
	Since we can establish the connection between the $L^p$-projective criterion (\ref{equation:lpcriteria}) and the summability condition of the $\theta$-lex-coefficients of $X$ just for $p=1$, there is no functional extension of Theorem \ref{theorem:clt} readily obtainable, except for $m=1$ (see \cite[Remark 4.2]{CS2018}).
\end{Remark}	

\section{Mixed moving average fields}
\label{sec3}
In this section we first introduce MMAF driven by a L\'evy basis. Then, we discuss weak dependence properties of such MMAF and derive sufficient conditions such that the asymptotic results of Section \ref{sec2-3} apply.

\subsection{Preliminaries}
\label{sec3-1}
Let $S$ denote a non-empty Polish space, $\mathcal{B}(S)$ the Borel $\sigma$-algebra on $S$, $\pi$ some probability measure on $(S,\mathcal{B}(S))$ and $\mathcal{B}_b(S \times \R^m)$ the bounded Borel sets of $S \times \R^m$. 

\begin{Definition}
	\label{basis}Consider a family $\Lambda=\{\Lambda(B), B \in \mathcal{B}_b(S\times\R^m) \}$ of $\R^d$-valued random variables. Then $\Lambda$ is called an $\R^d$-valued L\'evy basis or infinitely divisible independently scattered random measure on $S\times\R^m$ if
	\begin{enumerate}[(i)]
		\item the distribution of $\Lambda(B)$ is infinitely divisible (ID) for all $B\in \mathcal{B}_b(S\times\R^m)$,
		\item for arbitrary $n\in \N$ and pairwise disjoint sets $B_1,\dots,B_n \in \mathcal{B}_b(S\times\R^m)$ the random variables $\Lambda(B_1),\ldots,\Lambda(B_n)$ are independent and
		\item for any pairwise disjoint sets $B_1,B_2,\ldots \in \mathcal{B}_b(S\times\R^m)$ with $\bigcup_{n\in\N} B_n \in \mathcal{B}_b(S\times\R^m)$ we have, almost surely, $\Lambda (\bigcup_{n\in\N} B_n)= \sum_{n\in\N}\Lambda(B_n)$.
	\end{enumerate}
\end{Definition}

In the following we will restrict ourselves to L\'evy bases which are homogeneous in space and time and factorisable, i.e. L\'evy bases with characteristic function
\begin{equation}\label{equation:fact}
	\varphi_{\Lambda(B)}(u)=E\left[e^{\text{i}\langle u,\Lambda(B) \rangle}\right]=e^{\Phi(u)\Pi(B)}
\end{equation}
for all $u\in \R^d$ and $B\in\mathcal{B}_b(S\times\R^m)$, where $\Pi=\pi\times\lambda$ is the product measure of the probability measure $\pi$ on $S$ and the Lebesgue measure $\lambda$ on $\R^m$. Furthermore, 
\begin{align}\label{equation:phi}
	\Phi(u)=\text{i}\langle\gamma,u\rangle -\frac{1}{2} \langle u,\Sigma u\rangle +\int_{\R^d} \left(e^{\text{i}\langle u,x \rangle}-1-\text{i}\langle u,x\rangle \mathbb{1}_{[0,1]}(\norm{x})\right)\nu(dx)
\end{align}
is the cumulant transform of an ID distribution with characteristic triplet $(\gamma,\Sigma,\nu)$, where $\gamma \in \R^d$, $\Sigma\in M_{d\times d}(\R)$ is a symmetric positive-semidefinite matrix and  $\nu$ is a L\'evy-measure on $\R^d$, i.e.
\begin{align*}	
	\nu(\{0\})=0 \quad \text{and} 
	\int_{\R^d}\left(1\wedge\norm{x}^2\right)\nu(dx)<\infty.
\end{align*}
The quadruplet $(\gamma, \Sigma,\nu,\pi)$ determines the distribution of the L\'evy basis completely and therefore it is called the characteristic quadruplet.
Following \cite{P2003}, it can be shown that a L\'evy basis has a L\'evy-It\^o decomposition.

\begin{Theorem}\label{theorem:levyitobasis}
	Let $\{\Lambda(B), B\in \mathcal{B}_b(S\times\R^m)\}$ be an $\R^d$-valued L\'evy basis on $S\times\R^m$ with characteristic quadruplet $(\gamma,\Sigma,\nu,\pi)$. Then, there exists a modification $\tilde{\Lambda}$ of $\Lambda$ which is also a L\'evy basis with characteristic quadruplet $(\gamma,\Sigma,\nu,\pi)$ such that there exists an $\R^d$-valued L\'evy basis $\tilde{\Lambda}^G$ on $S\times \R^m$ with characteristic quadruplet $(0,\Sigma,0,\pi)$ and an independent Poisson random measure $\mu$ on $(\R^d\times S\times \R^m, \mathcal{B}(\R^d\times S\times \R^m))$ with intensity measure $\nu\times\pi\times\lambda$ such that
	\begin{align}
		\begin{split}\label{equation:levyito}
			\tilde{\Lambda}(B)=\gamma(\pi\times\lambda)(B) + &\tilde{\Lambda}^G(B)+\int_{\norm{x}\leq1}\int_Bx(\mu(dx,dA,ds)-ds\pi(dA)\nu(dx))\\ &+\int_{\norm{x}>1}\int_B x \mu(dx,dA,ds)
		\end{split}
	\end{align}
	for all $B\in \mathcal{B}_b(S\times\R^m)$.\\
	If the L\'evy measure additionally fulfills $\int_{\norm{x}\leq1}\norm{x}\nu(dx)<\infty$, it holds that
	\begin{align}\label{equation:levyitofinvar}
		\tilde{\Lambda}(B)=\gamma_0(\pi\times\lambda)(B) +\tilde{\Lambda}^G(B)+\int_{\R^d}\int_B x \mu(dx,dA,ds)
	\end{align}
	for all $B\in \mathcal{B}_b(S\times\R^m)$ with 
	\begin{align}\label{equation:gammazero}
		\gamma_0:=\gamma-\int_{\norm{x}\leq1}x\nu(dx).
	\end{align} 
	Note that the integral with respect to $\mu$ exists $\omega$-wise as a Lebesgue integral.
\end{Theorem}
\begin{proof}
	Analogous to \cite[Theorem 2.2]{BS2011}.
\end{proof}
We refer the reader to \cite[Section 2.1]{JS2003} for further details on the integration with respect to Poisson random measures. From now on we assume that any L\'evy basis has a decomposition (\ref{equation:levyito}).

Let us recall the following multivariate extension of \cite[Theorem 2.7]{RR1989}. 

\begin{Theorem}\label{theorem:2}
	Let $\Lambda=\{\Lambda(B), B\in \mathcal{B}_b(S\times\R^m)\}$ be an $\R^d$-valued L\'evy basis with characteristic quadruplet $(\gamma,\Sigma,\nu,\pi)$, $f:S\times\R^m\rightarrow M_{n\times d}(\R)$ be a $\mathcal{B}(S\times\R^m)$-measurable function. Then $f$ is $\Lambda$-integrable in the sense of \cite{RR1989}, if and only if
	\begin{gather}
		\int_S\!\int_{\R^m}\!\! \Big\| f(A,s)\gamma\!+\!\! \!\int_{\R^d}\!\!\, f(A,s)x\left(\mathbb{1}_{[0,1]}(\|f(A,s)x\|)\!-\!\!\mathbb{1}_{[0,1]}(\|x\|)\right)\!\nu(dx)\Big\| ds \pi(dA)\!\!<\!\!\infty, \label{equation:intcond1} \\
		\int_S\int_{\R^m}\|f(A,s)\Sigma f(A,s)'\|\ ds\pi(dA)<\infty \text{ and} \label{equation:intcond2}\\
		\int_S\int_{\R^m}\int_{\R^d}  \Big(1\wedge \|f(A,s)x\|^2 \Big) \nu(dx)ds\pi(dA)<\infty. \label{equation:intcond3}
	\end{gather}
	If $f$ is $\Lambda$-integrable, the distribution of the stochastic integral $\int_S\int_{\R^m}f(A,s)\Lambda(dA,ds)$ is ID with the characteristic triplet $(\gamma_{int},\Sigma_{int},\nu_{int})$ given by
	\begin{gather*}
		\gamma_{int}=\int_S\int_{\R^m} \Big(f(A,s)\gamma+\int_{\R^d}f(A,s)x\left(\mathbb{1}_{[0,1]}(\|f(A,s)x\|)-\mathbb{1}_{[0,1]}(\|x\|)\Big)\nu(dx)\right)ds \pi(dA), \\[-5pt]
		\Sigma_{int}=\int_S\int_{\R^m}f(A,s)\Sigma f(A,s)'ds\pi(dA) \text{ and} \\
		\nu_{int}(B)= \int_S\int_{\R^m}\int_{\R^d}\mathbb{1}_B(f(A,s)x)\nu(dx) ds\pi(dA)
	\end{gather*}
	for all Borel sets $B\subset \R^n\backslash \{0\}$.
\end{Theorem}
\begin{proof}
	Analogous to \cite[Proposition 2.3]{BS2011}.
\end{proof}		

Implicitly, we always assume that $\Sigma_{int}$ or $\nu_{int}$ are different from zero throughout the paper to rule out the deterministic case.

For $m=1$ it is known that the L\'evy-It\^o decomposition simplifies if the underlying L\'evy process $L_t=\Lambda(S\times (0,t])$ is of finite variation (if and only if $\Sigma=0$ and $\int_{|x|\leq1}|x|\nu(dx)<\infty$). 
Extending this one-dimensional notion, we speak of the finite variation case whenever $\Sigma=0$ and $\int_{\norm{x}\leq1}\norm{x}\nu(dx)<\infty$.

\begin{Corollary}\label{corollary:1}
	Let $\Lambda=\{\Lambda(B), B\in \mathcal{B}_b(S\times\R^m)\}$ be an $\R^d$-valued L\'evy basis with characteristic quadruplet $(\gamma,0,\nu,\pi)$ satisfying $\int_{\|x\|\leq1}\|x\|\nu(dx)<\infty$, and define $\gamma_0$ as in (\ref{equation:gammazero}), such that for $\Phi(u)$ in (\ref{equation:fact}) we have $\Phi(u)=\text{i}\langle\gamma_0,u\rangle +\int_{\R^d} \left(e^{\text{i}\langle u,x \rangle}-1\right) \nu(dx)$. Furthermore, let $f:S\times\R^m\rightarrow M_{n\times d}(\R)$ be a $\mathcal{B}(S\times\R^m)$-measurable function satisfying
	\begin{gather}
		\int_S\int_{\R^m}\|f(A,s) \gamma_0\|\ ds\pi(dA)<\infty \text{ and} \label{equation:intcondfinvar1}\\
		\int_S\int_{\R^m}\int_{\R^d}  \Big(1\wedge \|f(A,s)x\| \Big)\nu(dx)ds\pi(dA)<\infty.\label{equation:intcondfinvar2}
	\end{gather}
	Then,
	\begin{gather*}
		\int_S\int_{\R^m}f(A,s)\Lambda(dA,ds)=\int_S\int_{\R^m}f(A,s)\gamma_0 \ ds \pi(dA)+ \int_{\R^d}\int_{S}\int_{\R^m} f(A,s) x \mu(dx,dA,ds), 
	\end{gather*}
	where the right hand side denotes an $\omega$-wise Lebesgue integral. Additionally, the distribution of the stochastic integral $\int_S\int_{\R^m}f(A,s)\Lambda(dA,ds)$ is ID with characteristic function
	\begin{gather*}
		E\left[e^{\text{i}\langle u,\int_S\int_{\R^m}f(A,s)\Lambda(dA,ds) \rangle}\right]=e^{\text{i}\langle u,\gamma_{int,0}\rangle+\int_{\R^d} \left(e^{\text{i}\langle u,x \rangle}-1\right) \nu_{int}(dx)}, \ u\in\R^d,
	\end{gather*}
	where
	\begin{gather*}
		\gamma_{int,0}=\int_S\int_{\R^m}f(A,s)\gamma_0\ ds \pi(dA), \, \, \textrm{and} \,\,
		\nu_{int}(B)= \int_S\int_{\R^m}\int_{\R^d} \mathbb{1}_B(f(A,s)x)\nu(dx)ds\pi(dA).
	\end{gather*}
\end{Corollary}

\subsection{The MMAF framework}
\label{sec3-2}
\begin{Definition}
	Let $\Lambda=\{\Lambda(B), B\in \mathcal{B}_b(S\times\R^m)\}$ be an $\R^d$-valued L\'evy basis and let $f:S\times\R^m\rightarrow M_{n\times d}(\R)$ be a $\mathcal{B}(S\times\R^m)$-measurable function satisfying the conditions (\ref{equation:intcond1}), (\ref{equation:intcond2}) and (\ref{equation:intcond3}). Then, the stochastic integral 
	\begin{gather}\label{equation:MMAfield}
		X_t:=\int_S\int_{\R^m}f(A,t-s)\Lambda(dA,ds)
	\end{gather}
	is stationary, well-defined for all $t\in\R^m$ and its distribution is ID. The random field $X$ is called an $\R^n$-valued mixed moving average field (MMAF) and $f$ its kernel function. 
\end{Definition}	

In the following result we give conditions ensuring finite moments of an MMAF and explicit formulas for the first- and second-order moments.

\begin{Proposition}\label{proposition:MMAexistencemoments}
	Let $X$ be an $\R^n$-valued MMAF driven by an $\R^d$-valued L\'evy basis with characteristic quadruplet $(\gamma,\Sigma,\nu,\pi)$ and with $\Lambda$-integrable kernel function $f:S\times\R^m\rightarrow M_{n\times d}(\R)$.
	\begin{enumerate}[(i)]
		\item If $\int_{\|x\|>1} \| x\|^r \,\nu(dx) < \infty \, \textrm{and} \, f\! \in\! L^r(S \times \R^m, \pi \otimes \lambda)$
		for $r \in [2, \infty)$, then $E[\|X_t\|^r]<\infty$ for all $t\in\R^m$.
		\item If $\int_{\|x\|>1} \| x\|^r \,\nu(dx) < \infty \, \textrm{and} \, f \!\in\! L^r(S \times \R^m, \pi \otimes \lambda)\cap L^2(S \times \R^m, \pi \otimes \lambda) $ for $r \in (0, 2)$, then $E[\|X_t\|^r]<\infty$ for all $t\in\R^m$.
	\end{enumerate}
	Consider the finite variation case, i.e. $\Sigma=0$ and $\int_{\norm{x}\leq1}\norm{x}\nu(dx)<\infty$, then the following holds:
	\begin{enumerate}[(i)]
		\item If $\int_{\|x\|>1} \| x\|^r \,\nu(dx) < \infty \, \textrm{and} \, f \!\in\! L^r(S \times \R^m, \pi \otimes \lambda)$
		for $r \in [1, \infty)$, then $E[\|X_t\|^r]<\infty$.
		\item If $\int_{\|x\|>1} \| x\|^r \,\nu(dx) < \infty \, \textrm{and} \, f \!\in\! L^r(S \times \R^m, \pi \otimes \lambda)\cap L^1(S \times \R^m, \pi \otimes \lambda) $ for $r \in (0, 1)$, then $E[\|X_t\|^r]<\infty$.
	\end{enumerate}
\end{Proposition}
\begin{proof}
	Analogous to \cite[Proposition 2.6]{CS2018}.
\end{proof}

\begin{Proposition}\label{proposition:MMAmoments}
	Let $X$ be an $\R^n$-valued MMAF driven by an $\R^d$-valued L\'evy basis with characteristic quadruplet $(\gamma,\Sigma,\nu,\pi)$ and with $\Lambda$-integrable kernel function $f:S\times\R^m\rightarrow M_{n\times d}(\R)$.
	\begin{enumerate}[(i)]
		\item If $\int_{\norm{x}>1}\norm{x}\nu(dx)<\infty$ and $f\in L^1(S\times\R^m,\pi\times\lambda)\cap L^2(S\times\R^m,\pi\times\lambda)$ the first moment of $X$ is given by
		\begin{align*}
			E[X_t]= \int_{S}\int_{\R^m}f(A,-s) \mu_\Lambda ds \pi(dA),
		\end{align*}
		where $\mu_\Lambda =\gamma+\int_{\norm{x}\geq1}x\nu(dx)$.
		\item If $\int_{\R^d}\norm{x}^2\nu(dx)<\infty$ and $f\in L^2(S\times\R^m,\pi\times\lambda)$, then $X_t\in L^2$ and
		\begin{align*}
			Var(X_t)= \int_{S}\int_{\R^m}f(A,-s)\Sigma_\Lambda f(A,-s)'ds \pi(dA)\text{ and}\\
			Cov(X_0,X_t)=\int_{S}\int_{\R^m}f(A,-s)\Sigma_\Lambda f(A,t-s)'ds \pi(dA),
		\end{align*}
		where $\Sigma_\Lambda =\Sigma+\int_{\R^d}xx'\nu(dx)$. 
		\item Consider the finite variation case, i.e. it holds that $\Sigma=0$ and $\int_{\norm{x}\leq1}\norm{x}\nu(dx)<\infty$. If $\int_{\norm{x}>1}\norm{x}\nu(dx)<\infty$ and $f\in L^1(S\times\R^m,\pi\times\lambda)$, the first moment of $X$ is given by
		\begin{align*}
			E[X_t]= \int_{S}\int_{\R^m}f(A,-s)\Big(\gamma_0+\int_{\R^d}x\nu(dx)\Big) ds \pi(dA),
		\end{align*}
		where $\gamma_0$ as defined in (\ref{equation:gammazero}).
	\end{enumerate}
\end{Proposition}
\begin{proof}
	Immediate from \cite[Section 25]{S2013} and Theorem \ref{theorem:2}.
\end{proof}

\subsection{Weak dependence properties of $(A,\Lambda)$-influenced MMAF}
\label{sec3-3}

Since there is no natural order on $\R^m$ for $m>1$, we cannot extend the definition of a natural filtration and, therefore, causal processes in a natural way to random fields. In the following, we will propose such an extension and prove $\theta$-lex-weak dependence for MMAF falling within this framework. Examples will be presented in Section \ref{sec3-7}.

\begin{Definition}\label{definition:influenced}
	Let $X=(X_t)_{t\in\R^m}$ be a random field, $A=(A_t)_{t\in\R^m}\subset\R^m$ a family of Borel sets with Lebesgue measure strictly greater than zero and $M=\{M(B),B\in\mathcal{B}_b(S\times\R^m)\}$ an independently scattered random measure. Assume that $X_t$ is measurable with respect to $\sigma(M(B),B\in \mathcal{B}_b(S\times A_t))$. We then call $A$ the \emph{sphere of influence}, $M$ the \emph{influencer}, $(\sigma(M(B),B\in \mathcal{B}_b(S\times A_t)))_{t\in\R^m}$ the \emph{filtration of influence} and $X$ an \emph{$(A,M)$-influenced random field}. If $A$ is translation invariant, i.e. $A_t=t+A_0$, the sphere of influence is fully described by the set $A_0$ and we call $A_0$ the \emph{initial sphere of influence}.
\end{Definition}

Note that for $m=1$, the class of causal mixed moving average processes driven by a L\'evy basis $\Lambda$ equals the class of $(A,\Lambda)$-influenced mixed moving average processes driven by $\Lambda$ with $A_t=V_t$.

Let $A=(A_t)_{t\in\R^m}$ be a translation invariant sphere of influence with initial sphere of influence $A_0$. In this section we consider the filtration $(\mathcal{F}_{t})_{t\in\R^{m}}$ generated by $\Lambda$, i.e. the $\sigma$-algebra generated by the set of random variables $\{\Lambda(B): B\in\mathcal{B}(S\times A_t)\}$ with $t\in\R^{m}$.\\
Consider an MMAF $X$ that is adapted to $(\mathcal{F}_{t})_{t \in\R^{m}}$. Then, $X$ is $(A,\Lambda)$-influenced and can be written as 
\begin{gather}\label{equation:influencedMMAfield}
	\begin{gathered}
		X_t=\int_S\int_{\R^m}f(A,t-s)
		\Lambda(dA,ds)=\int_S\int_{A_t}f(A,t-s)\Lambda(dA,ds).
	\end{gathered}
\end{gather}
Note that the translation invariance of $A$ is required to ensure stationarity of $X$.

In the following, we discuss under which assumptions an $(A,\Lambda)$-influenced MMAF is $\theta$-lex-weakly dependent. We start with a preliminary definition.

\begin{Definition}[{\cite[Definition 2.4.1]{BV2004}}]\label{definition:propercone}
	$K\subset\R^m$ is called a closed convex proper cone if it satisfies the following properties
	\begin{enumerate}[(i)]
		\item $K+K\subset K$ (ensures convexity)
		\item $\alpha K\subset K$ for all $\alpha\geq0$ (ensures that $K$ is a cone)
		\item $K$ is closed
		\item $K$ is pointed (i.e., if $x\in K$ and $-x\in K$, then $x=0$). 
	\end{enumerate}
\end{Definition}

We then apply a truncation technique to show that $X$ is $\theta$-lex-weakly dependent. Define $X_j$, $X_{\Gamma}$ as in Definition \ref{thetaweaklydependent} such that $j\in\R^m$ and $\Gamma\subset V_j^h$ (see Figure \ref{Plot1}). We truncate $X_j$ such that the truncation $\tilde{X}_j$ and $X_{\Gamma}$ become independent. From our construction, it will become clear that it is enough to find a truncation such that $\tilde{X}_j$ and $X_i$ are independent for the lexicographic greatest point $i\in V_j^h$. \\
For a given point $j$, we determine the truncation of $X_j$ by intersecting the integration set with $V_j^\psi$ for $\psi>0$ such that it does not intersect with $A_i$ (see Figure \ref{Plot2} and \ref{Plot3}). In the following, we will describe the choice of $\psi$. The figures illustrate the case $m=2$.

Let $i\in V_j^h$ be the lexicographic greatest point in $V_j^h$, i.e. $k\leq_{lex}i$ for all $k\in V_j^h$. In the following $dist(A,B)=\inf_{a\in A, b\in B}\norm{a-b}$ denotes the Euclidean distance of the sets $A$ and $B$. To ensure the existence of the above truncation, we assume that there exists an $\alpha\in\R^m\backslash\{0\}$ such that
\begin{gather}
	\sup_{x\in A_0, x\neq0}\frac{\alpha'x}{\norm{x}}<0.\label{condition:scalarproduct}
\end{gather}
Intuitively, (\ref{condition:scalarproduct}) ensures that the initial sphere of influence $A_0$  can be covered by a closed convex proper cone.
Moreover, w.l.o.g. by applying a rotation to $A_0$, we can always assume to work with $A_0 \subset V_0$. The following remark discusses such transformation.

\begin{Remark}\label{remark:halfspaces}
	Let $A_0$ be a subset of a half-space with Lebesgue measure strictly greater than zero such that $A_0 \nsubseteq V_0$.
	Define the translation invariant sphere of influence $A=(A_t)_{t\in\R^m}$ by $A_t=(A_0+ t)_{t\in\R^m}$ and consider the $(A,\Lambda)$-influenced MMAF $X=(X_{t})_{t\in\R^m}$ of the form $X_{t}=\int_S\int_{A_0+t}f(A,t-s)\Lambda(dA,ds)$.
	Note that if $A_0$ had Lebesgue measure zero, $X$ would be $0$ since the Lebesgue measure of $A_0$ is zero. 
	Define the hyperplane $D=\{x\in\R^m: \alpha'x=0\}$. Using the principal axis theorem we find an orthogonal matrix $O$ such that the axis of the first coordinate is orthogonal to the rotated hyperplane $OD$. Since $O$ is orthogonal, it holds that $|Det(D\varphi)(u)|=|Det(O)|=1$, where $D\varphi$ denotes the Jacobian matrix of the function $\varphi:u\mapsto Ou$. Additionally, for the rotated initial set $OA_0$ it holds that $OA_0\backslash V_0\subset \{0\}\times[0,\infty)^{m-1}$, such that $\lambda(\{0\}\times[0,\infty)^{m-1})=0$. By substitution for multiple variables we obtain for $\tilde{t}=O t$ 
	\begin{align}\label{equation:MMArotateinitialsphere}
		\begin{aligned}
			X_{t}&=\!\!\int_S\!\int_{\R^m}f(\!A,t-s)\mathbb{1}_{A_0+t}(s)\Lambda(dA,ds)\\
			&=\!\!\int_S\!\int_{\R^m}f(\!A,O^{-1}(Ot-Os))\mathbb{1}_{OA_0+Ot}(Os)\Lambda(dA,ds)\\
			&=\!\!\int_S\!\int_{O\!A_0+\tilde{t}}\!f(\!A,O^{-1}\!(\tilde{t}-\tilde{s}))\Lambda(dA,d\tilde{s})\!
			=\!\!\int_S\!\int_{O\!A_0\cap V_0+\tilde{t}}\!f(\!A,O^{-1}\!(\tilde{t}-\tilde{s}))\Lambda(dA,d\tilde{s})\\
			&=\!\!\int_S\!\int_{V_{\tilde{t}}}\tilde{f}^{O}(\!A,\tilde{t}-\tilde{s})\Lambda(dA,d\tilde{s})=\tilde{X}_{\tilde{t}}
		\end{aligned}
	\end{align}
	with $\tilde{f^O}(A,t-s)= f(A,O^{-1}(t-s))\mathbb{1}_{\{s\in OA_0+t \}}$. 
\end{Remark} 

Figure \ref{Plot4} shows the smallest closed convex proper cone covering $A_i$, which is called $K$. Note that all conditions can be formulated in terms of $A_0$ since the sphere of influence $A$ is translation invariant.

\begin{figure}[H]
	\begin{minipage}[H]{3.5cm}
		\center{
			\includegraphics[width=4.5cm]{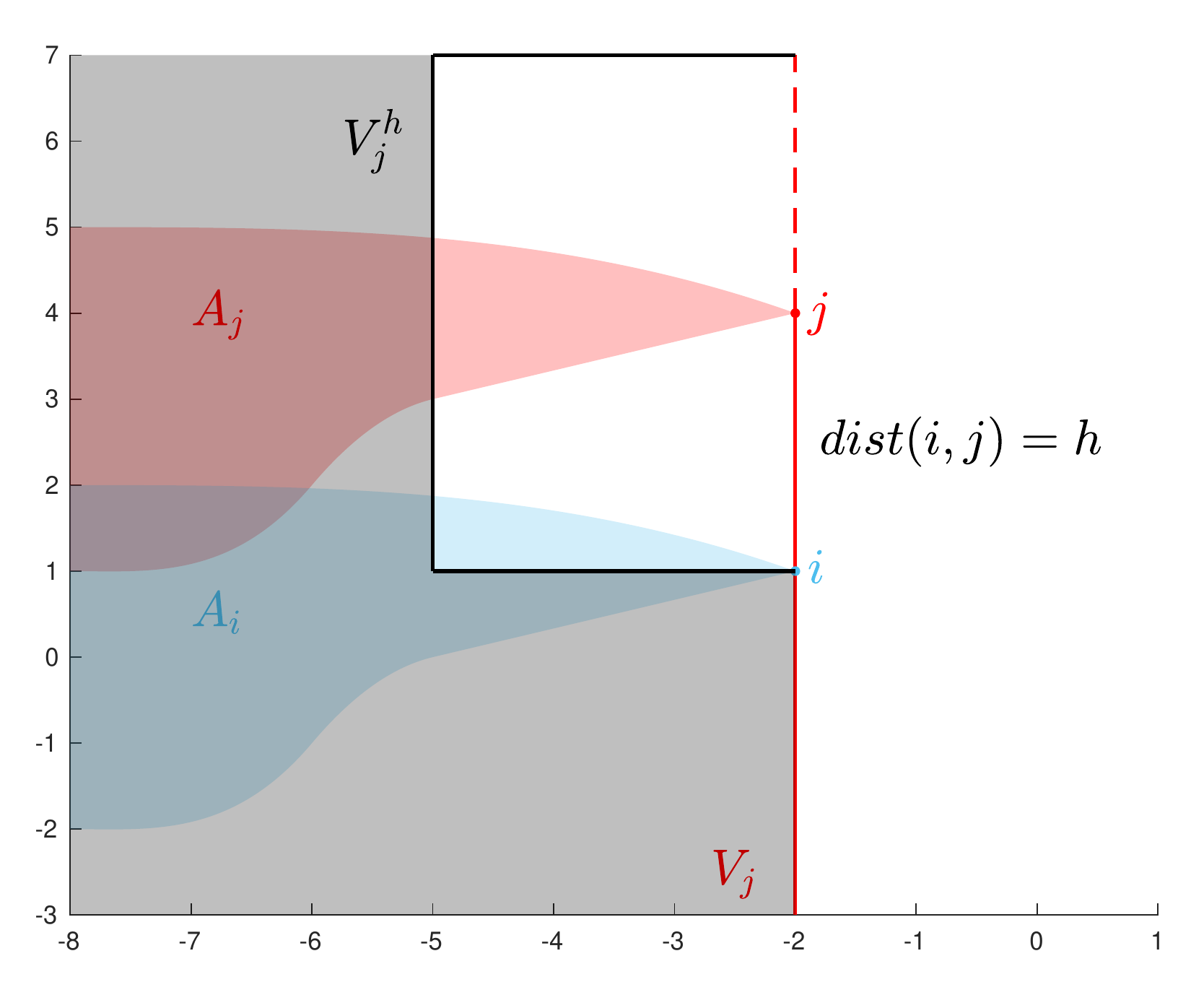}\label{figurePlot1}\\
			\caption{\small{Integration sets $A_j$ and $A_i$} of $X_j$ and $X_i$}\label{Plot1}}
	\end{minipage}
	\hspace{0.9cm}
	\begin{minipage}[H]{3.5cm}
		\vspace{0.19cm}
		\centering
		\includegraphics[width=4.5cm]{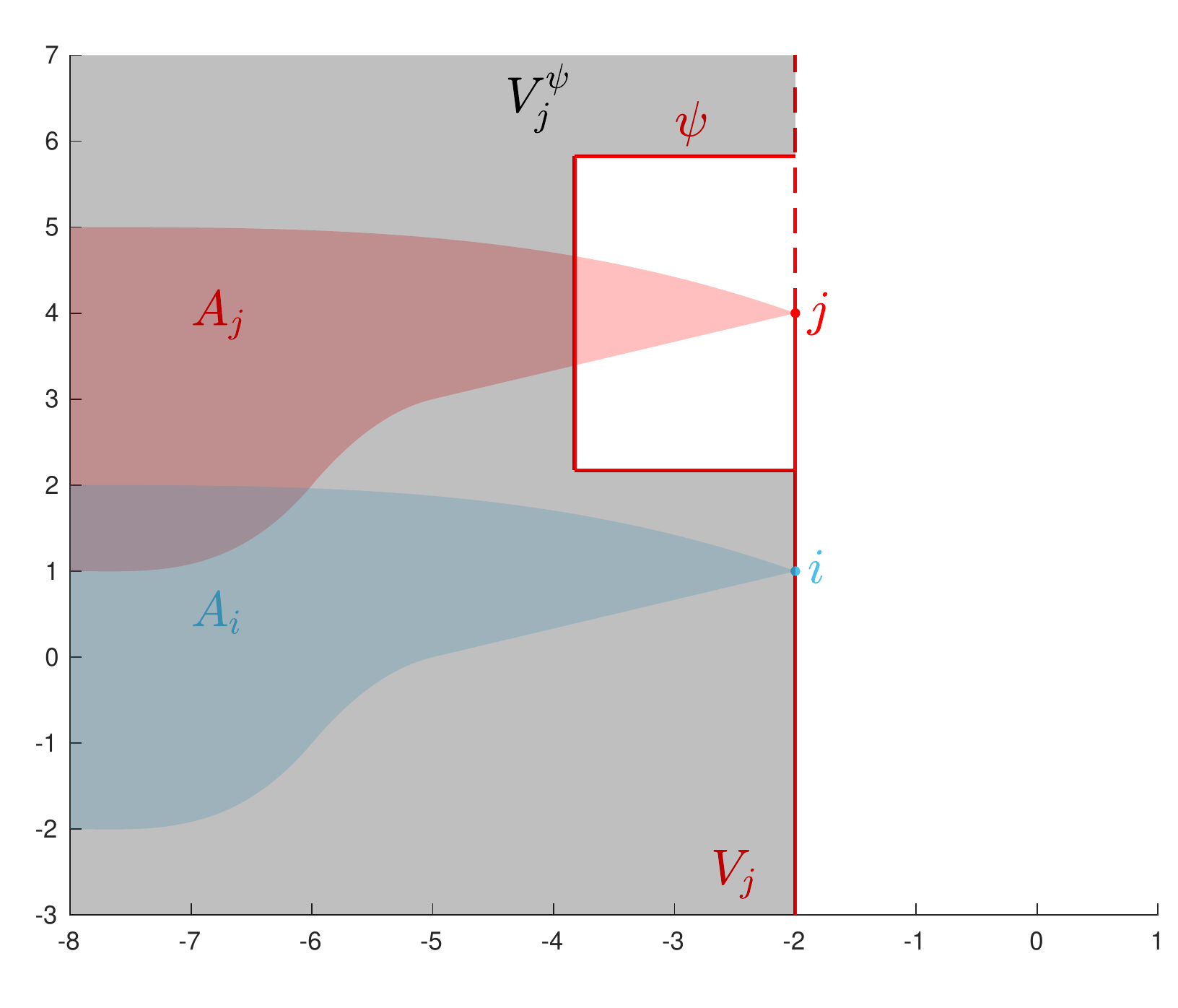}\label{figurePlot2}\\
		\caption{\small{$A_j$ and $A_i$ together with $V_j^\psi$}}
		\label{Plot2}
	\end{minipage}
	\hspace{0.9cm}
	\begin{minipage}[H]{3.5cm}
		\vspace{0.6cm}
		\centering
		\includegraphics[width=4.5cm]{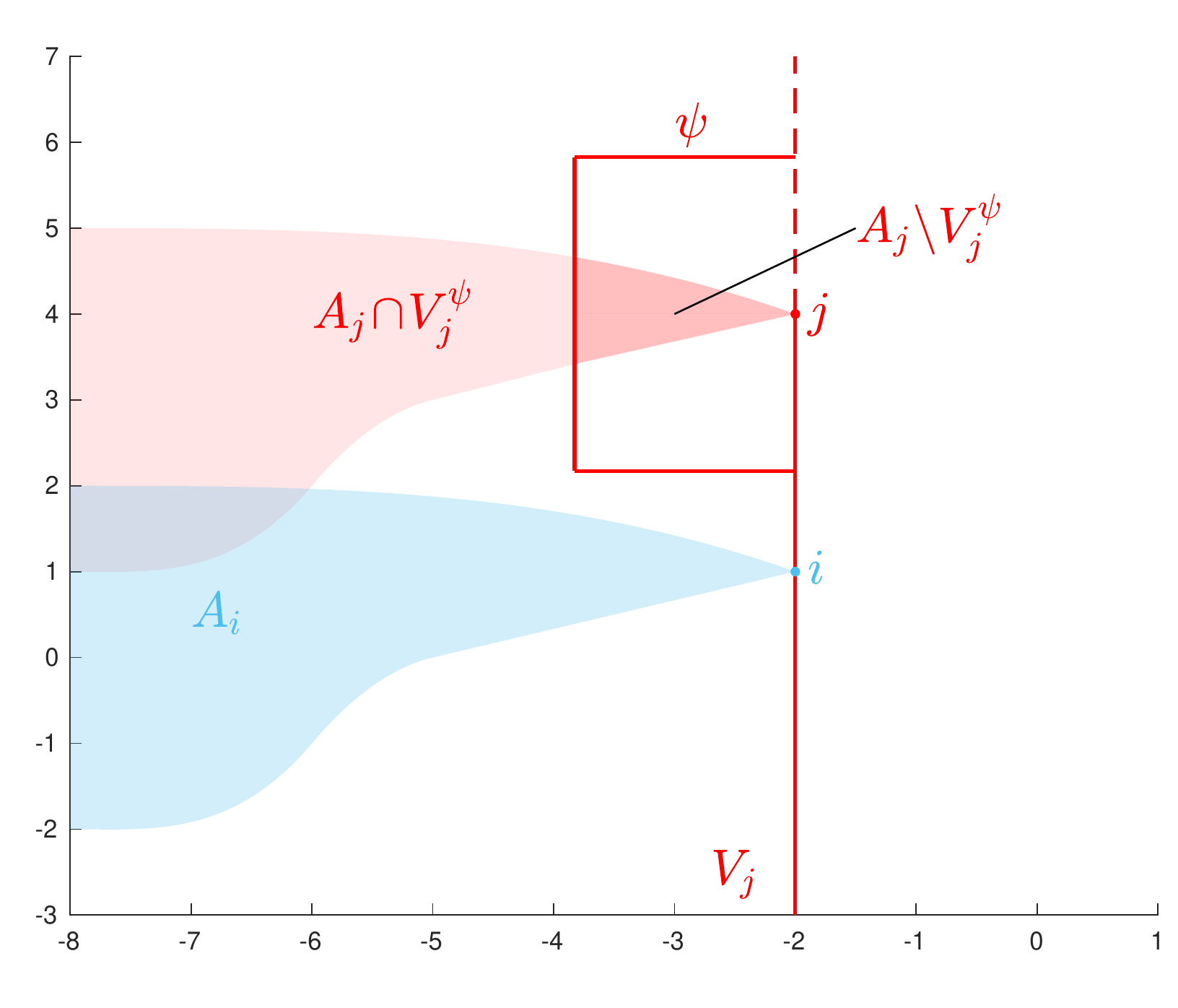}\label{figurePlot3}\\
		\caption{\small{Integration sets $A_i$ and $A_j \backslash V_j^{\psi}$ of $X_i$ and $\tilde{X}_j$}}
		\label{Plot3}
	\end{minipage}
\end{figure}

In order to choose $\psi$ we first define 
\begin{gather}\label{equation:xi}
	b=\sup_{\substack{x\in A_0\\ \norm{x}=1}}\frac{\alpha'x}{\norm{\alpha}} \text{ and }\tilde{K}=\left\{x\in\R^m: \frac{\alpha'x}{\norm{x}}\leq b \right\}.
\end{gather}
Due to (\ref{condition:scalarproduct}) (see Figure \ref{Plot4}), it holds $-1 \leq b<0$. For $x_1,x_2\in\tilde{K}$ it holds 
\begin{gather*}
	\frac{\alpha'(x_1+x_2)}{\norm{x_1+x_2}}= \frac{\alpha'x_1}{\norm{x_1+x_2}}+\frac{\alpha'x_2}{\norm{x_1+x_2}}\leq b \frac{\norm{x_1}+\norm{x_2}}{\norm{x_1+x_2}}\leq b
\end{gather*}
such that $\tilde{K}$ is a closed convex proper cone. It can be interpreted as the smallest equiangular closed convex proper cone that contains $A_0$.
Then, $\cos(\beta+\frac{\pi}{2})=b$ such that $\beta=\arcsin(-b)\in [0,\frac{\pi}{2})$ (see Figure \ref{Plot5}) and $dist(j,\tilde{K})\geq \sin(\beta)h=-bh$ (see Figure \ref{Plot6}). We choose $\psi$ as 
\begin{gather}\label{equation:psi}
	\psi(h)=\frac{-bh}{\sqrt{m}}.
\end{gather}
In particular we have $\psi(h)= \mathcal{O}(h)$. 

Let $l\in V_j^h$ be an arbitrary point. From the given choice of $\psi$ and $i$ it holds $dist(l,j)\geq dist(i,j)$, $A_i\cap (A_j\backslash V_j^\psi) =\emptyset$, $A_i=i+A_0\subset i+ \tilde{K}$ and $A_l=l+A_0\subset l+ \tilde{K}$. Since $\tilde{K}$ is an equiangular closed convex proper cone we get $A_l\cap (A_j\backslash V_j^\psi) =\emptyset$.

\begin{figure}[H]
	\begin{minipage}[H]{3.5cm}
		\centering
		\includegraphics[width=4.5cm]{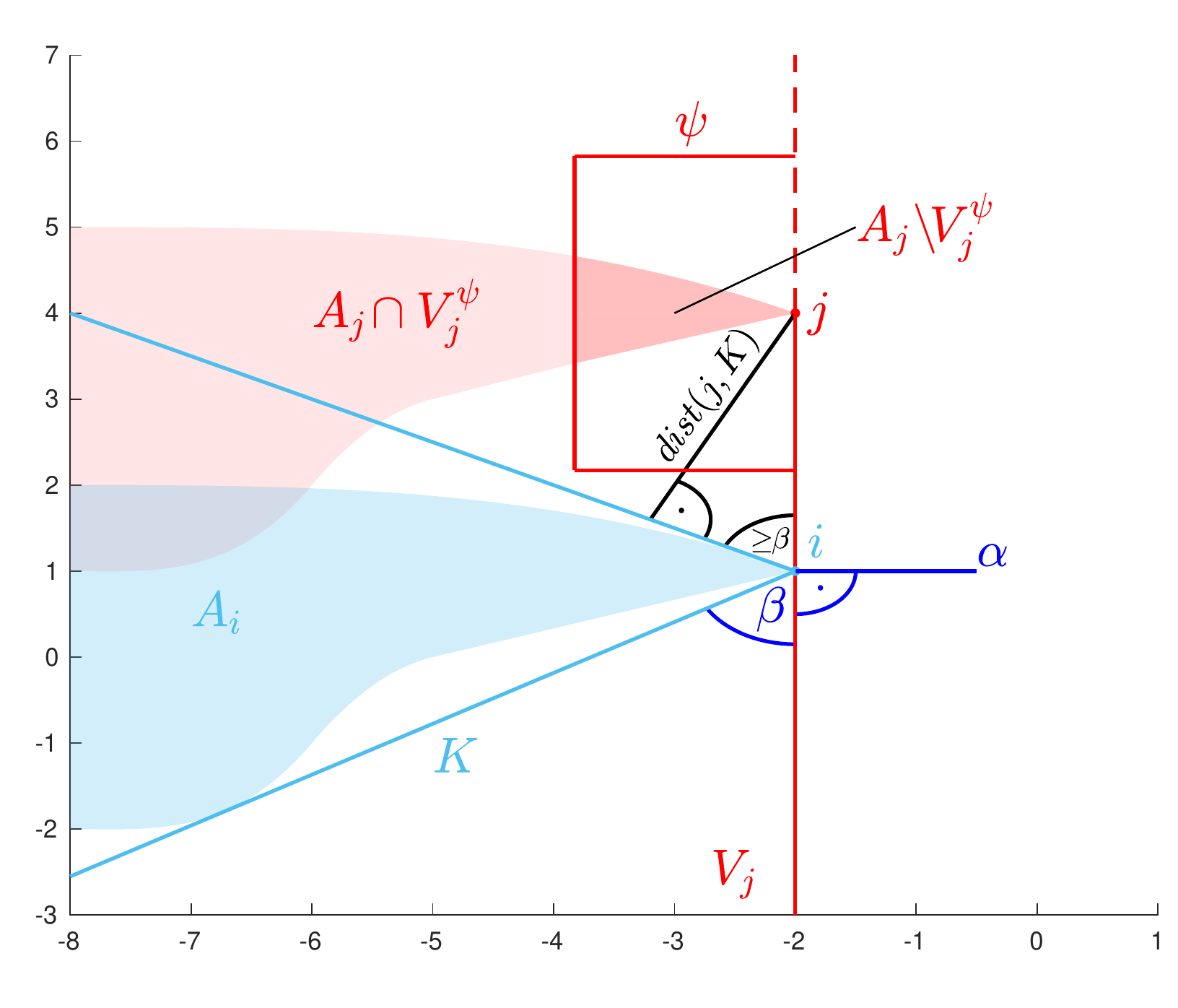}\label{figurePlot4}\\
		\caption{\small{ Choice of $\alpha$ and $\beta$ }}\label{Plot4}
	\end{minipage}
	\hspace{0.9cm}
	\begin{minipage}[H]{3.5cm}
		\vspace{-0.04cm}
		\centering
		\includegraphics[width=4.5cm]{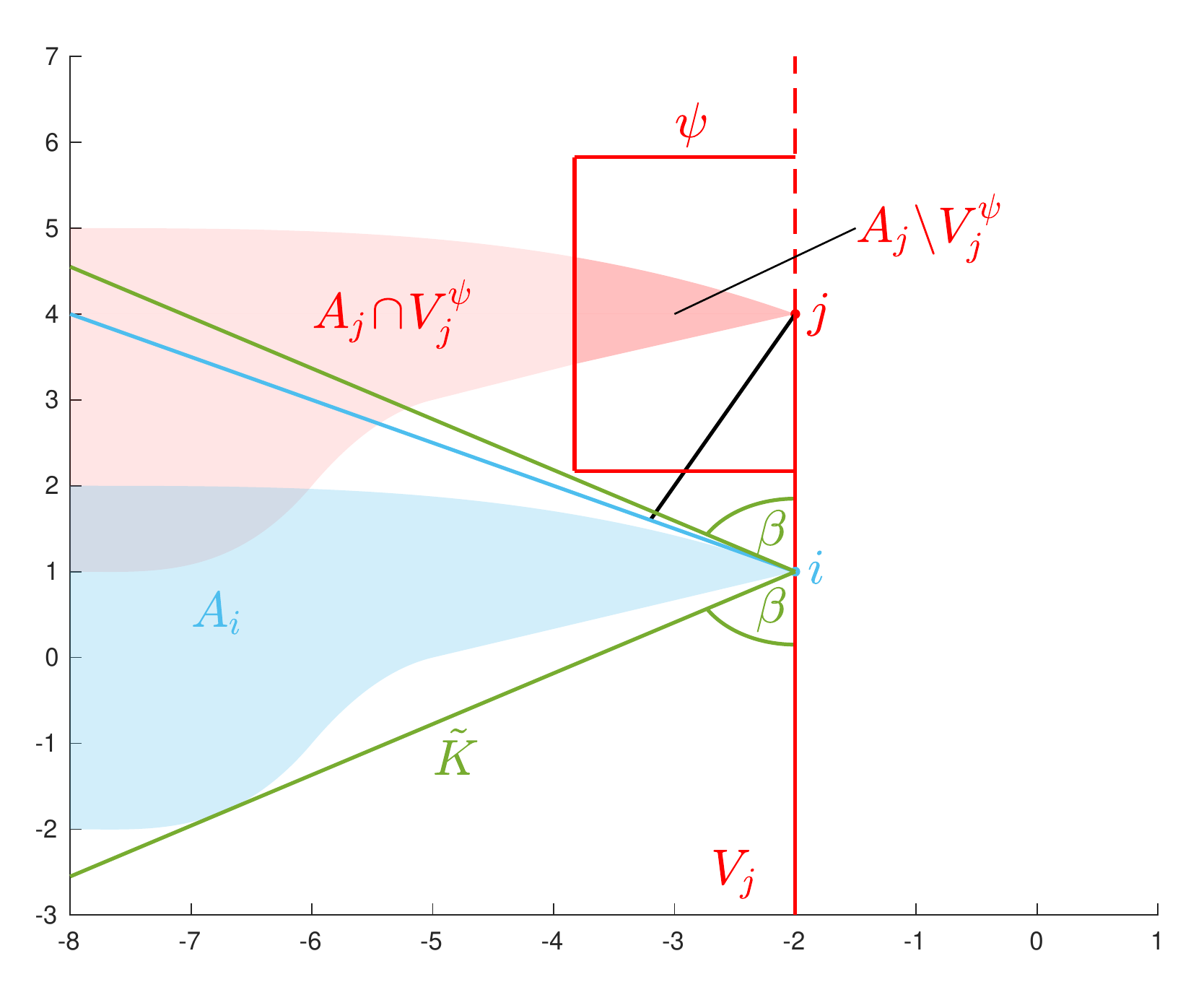}\label{figurePlot5}\\
		\caption{\small{ Choice of $\tilde{K}$ }}	\label{Plot5}
	\end{minipage}
	\hspace{0.9cm}
	\begin{minipage}[H]{3.5cm}
		\centering
		\includegraphics[width=4.5cm]{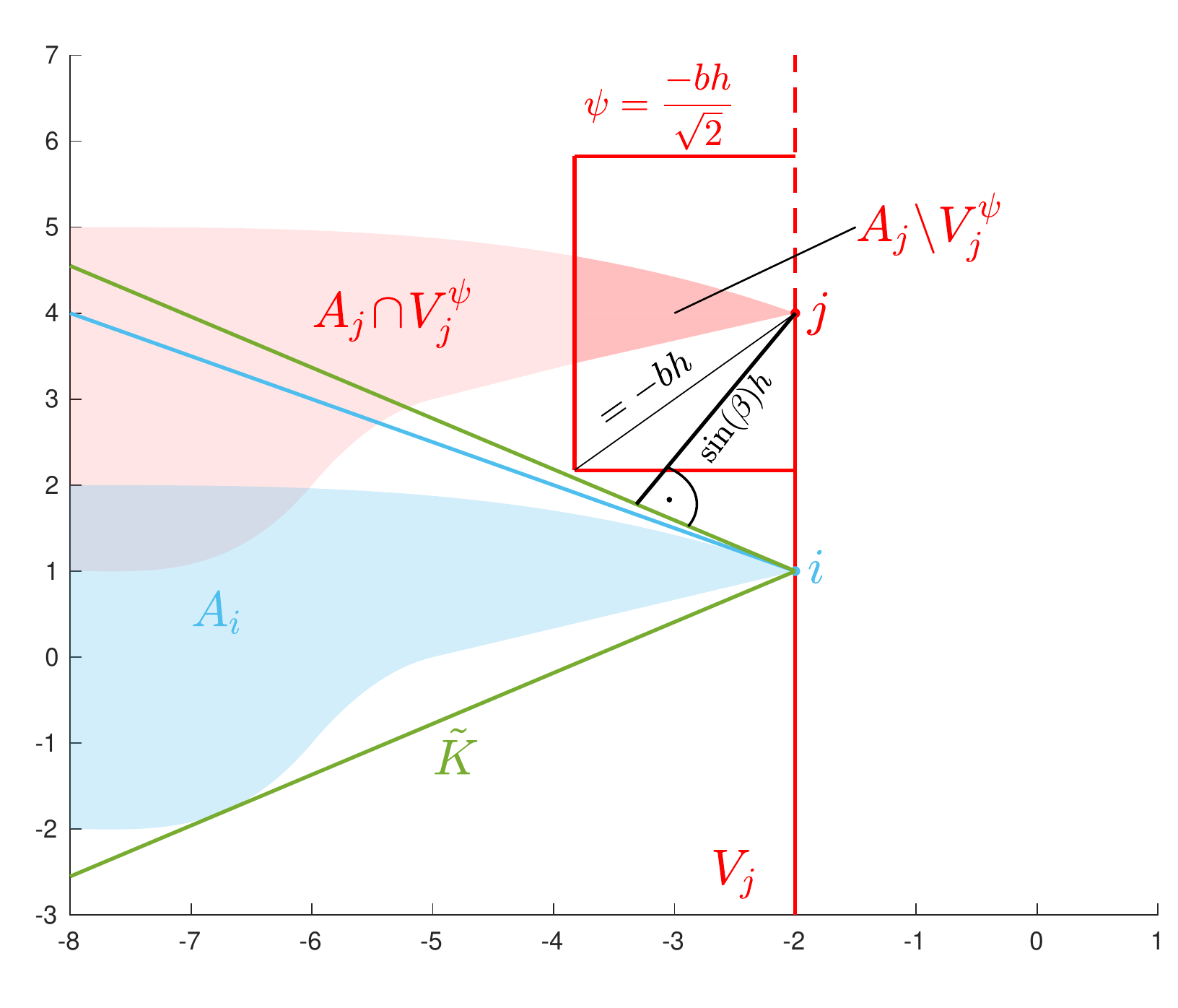}\label{figurePlot6}\\
		\caption{\small{ Construction of $\psi$ }}
		\label{Plot6}
	\end{minipage}
\end{figure}

The conditions below, which are expressed in terms of the kernel function $f$ and the characteristic quadruplet of the driving L\'evy basis, are sufficient to show that an $(A,\Lambda)$-influenced MMAF is $\theta$-lex-weakly dependent.

\begin{Proposition}\label{proposition:mmathetaweaklydep}
	Let $\Lambda$ be an $\R^d$-valued L\'evy basis with characteristic quadruplet $(\gamma,\Sigma,\nu,\pi)$ and $f:S\times\R^{m}\ra M_{n\times d}(\R)$ a $\mathcal{B}(S\times\R^m)$-measurable function. Consider the $(A,\Lambda)$-influenced MMAF
	\begin{gather*}
		X_{t}=\int_S\int_{A_{t}}f(A,t-s)\Lambda(dA,ds), \ t\in\R^{m}
	\end{gather*}
	with translation invariant sphere of influence $A$ such that (\ref{condition:scalarproduct}) holds.
	\begin{enumerate}[(i)]
		\item If $\int_{\norm{x}>1}\norm{x}^2\nu(dx)<\infty$, $\gamma+\int_{\norm{x}>1}x\nu(dx)=0$ and $f\in L^2(S\times\R^{m},\pi\otimes\lambda)$, then $X$ is $\theta$-lex-weakly dependent with $\theta$-lex-coefficients satisfying 
		\begin{gather}\label{equation:thetalexcoefficientsmmafield}
			\theta_X(h)\leq 2 \Big(\int_S\int_{A_0\cap V_{0}^{\psi(h)}}\textup{tr}(f(A,-s)\Sigma_\Lambda f(A,-s)')ds\pi(dA)\Big)^{\frac{1}{2}}=:\hat{\theta}_X^{(i)}(h).
		\end{gather}
		\item If $\int_{\norm{x}>1}\norm{x}^2\nu(dx)<\infty$ and $f\in L^2(S\times\R^{m},\pi\otimes\lambda)\cap L^1(S\times\R^{m},\pi\otimes\lambda)$, then $X$ is $\theta$-lex-weakly dependent with $\theta$-lex-coefficients satisfying 
		\begin{gather}
			\begin{aligned}
				\theta_X(h)\leq 2& \bigg(\int_S\int_{A_0\cap V_{0}^{\psi(h)}}\textup{tr}(f(A,-s)\Sigma_\Lambda f(A,-s)')ds\pi(dA)\\
				&+\Big\lVert\int_S\int_{A_0\cap V_{0}^{\psi(h)}} f(A,-s)\mu_\Lambda  ds\pi(dA)\Big\rVert^2\bigg)^{\frac{1}{2}}=:\hat{\theta}_X^{(ii)}(h).\label{equation:thetalexcoefficientsmeanmmafield}
			\end{aligned}
		\end{gather}
		\item If $\int_{\R^d}\norm{x}\nu(dx)<\infty$, $\Sigma=0$ and $f\in L^1(S\times\R^m,\pi\otimes\lambda)$ with $\gamma_0$ as in (\ref{equation:gammazero}), then $X$ is $\theta$-lex-weakly dependent with $\theta$-lex-coefficients satisfying 
		\begin{gather}
			\begin{aligned}
				\theta_X(h)\leq 2& \bigg(\int_S\int_{A_0\cap V_{0}^{\psi(h)}}\norm{f(A,-s)\gamma_0}ds\pi(dA)\\
				+&\int_S\int_{A_0\cap V_{0}^{\psi(h)}} \int_{\R^d}\norm{f(A,-s)x} \nu(dx)ds\pi(dA)\bigg)=:\hat{\theta}_X^{(iii)}(h).\label{equation:thetalexcoefficientsfinitevariationmmafield}
			\end{aligned}                
		\end{gather}
		\item If $\int_{\norm{x}>1}\norm{x}\nu(dx)<\infty$ and $f\in L^1(S\times\R^m,\pi\otimes\lambda)\cap L^2(S\times\R^m,\pi\otimes\lambda)$, then $X$ is $\theta$-lex-weakly dependent with $\theta$-lex-coefficients satisfying 
		\begin{gather*}
			\theta_X(h)\leq 2\bigg(\int_S\int_{A_0\cap V_0^{\psi(h)}}\textup{tr}(f(A,-s)\Sigma_{\Lambda_1} f(A,-s)')ds\pi(dA)\\
			+\Big\lVert\int_S\int_{A_0\cap V_0^{\psi(h)}} f(A,-s)\gamma  ds\pi(dA)\Big\rVert^2\bigg)^{\frac{1}{2}}\\
			+2\int_S\int_{A_0\cap V_0^{\psi(h)}} \int_{\norm{x}>1}\norm{f(A,-s)x} \nu(dx)ds\pi(dA)=:\hat{\theta}_X^{(iv)}(h).
		\end{gather*}
	\end{enumerate}
	The results above hold for all $h>0$ with $\psi$ as defined in $(\ref{equation:psi})$, $\Sigma_\Lambda =\Sigma+\int_{\R^d}xx'\nu(dx)$, $\Sigma_{\Lambda_1} =\Sigma+\int_{\norm{x}\leq1}xx'\nu(dx)$ and $\mu_\Lambda =\gamma-\int_{\norm{x}\geq1}x\nu(dx)$.
\end{Proposition}
\begin{proof}
	See Section \ref{sec5-2}.
\end{proof}

In the next proposition, we consider a vector of a shifted real-valued $(A,\Lambda)$-influenced MMAF, and we show that it is $\theta$-lex weakly dependent. 
This result is necessary to analyze, for example, the asymptotic behavior of the sample autocovariances. Define the set of possible shifts 
\begin{gather}
	S_k=\{(a,b)'\in\{0,\ldots,k\}\times\{-k,\ldots,k\}^{m-1} \},\ k\in\N_0\label{equation:shiftedmmafieldvector}
\end{gather} 
and consider the enumeration $\{s_1,\ldots,s_{|S_k|}\}$ of $S_k$, where $|S_k|=(k+1)(2k+1)^{m-1}$.
Besides the hereditary properties from Proposition \ref{proposition:mmathetahereditary} we show that the field 
\begin{gather}\label{equation:shiftedrandomfield}
	Z_t=(X_t,X_{t+s_1},X_{t+s_2},\ldots,X_{t+s_{|S_k|}})
\end{gather}
inherits weak dependence properties.
\begin{Proposition}\label{proposition:vectorinfluencedweaklydep}
	Let $\Lambda$ be an $\R^d$-valued L\'evy basis with characteristic quadruplet $(\gamma,\Sigma,\nu,\pi)$ and $f:S\times\R^m\ra M_{1\times d}(\R)$ be a $\Lambda$-integrable, $\mathcal{B}(S\times\R^m)$-measurable function. Consider the $(A,\Lambda)$-influenced MMAF 
	\begin{gather*}
		X_{t}=\int_S\int_{A_{t}}f(A,t-s)\Lambda(dA,ds), \ t\in\R^{m}
	\end{gather*}
	with translation invariant sphere of influence $A$ such that (\ref{condition:scalarproduct}) holds. Then
	\begin{gather*}
		Z_t:=\int_S\int_{A_t}g(A,t-s)\Lambda(dA,ds), \ t\in\R^m,
	\end{gather*}
	where $g(A,s)=(f(A,s),f(A,s-s_1),\ldots, f(A,s-s_{|S_k|}))'$ is a $\mathcal{B}(S\times\R^m)$-measurable function with values in $M_{(k+1)(2k+1)^{m-1}\times d}(\R)$ for $k\in\N_0$, is an $(A,\Lambda)$-influenced MMAF.\\
	If $X$ additionally satisfies the conditions of Proposition \ref{proposition:mmathetaweaklydep} (i), (ii), (iii) or (iv), then $Z$ is $\theta$-lex-weakly dependent with coefficients respectively given by
	\begin{gather}
		\begin{aligned}\label{eq:mmainfluencedvectorweakly}
			\theta_Z^{(i)}(h)\leq&\mathcal{D}\hat{\theta}_X^{(i)}(h-\psi^{-1}(k)),\qquad\qquad 
			\theta_Z^{(ii)}(h)\leq&\mathcal{D} \hat{\theta}_X^{(ii)}(h-\psi^{-1}(k)),\\
			\theta_Z^{(iii)}(h)\leq&\mathcal{C} \hat{\theta}_X^{(iii)}(h-\psi^{-1}(k)),\text{ and}\qquad
			\theta_Z^{(iv)}(h)\leq&\mathcal{C}\hat{\theta}_X^{(iv)}(h-\psi^{-1}(k)),
		\end{aligned}
	\end{gather}
	where $\mathcal{D}=|S_k|^{m/2}$, $\mathcal{C}=|S_k|^{m}$ for $\psi(h)>k$ and $\hat{\theta}^{(\cdot)}(h)$ are defined as in Proposition \ref{proposition:mmathetaweaklydep}. 
\end{Proposition}
\begin{proof}
	See Section \ref{sec5-2}.
\end{proof}

\subsection{Sample moments of $(A,\Lambda)$-influenced MMAF}
\label{sec3-4}
Let us consider an $\R^n$-valued $(A,\Lambda)$-influenced MMAF
\begin{gather}
	X=(X_u)_{u\in\Z^m} \text{ with } X_u=\int_S\int_{A_u}f(A,u-s)\Lambda(dA,ds),\label{equation:discreteinfluencedMMAfield}
\end{gather}
translation invariant sphere of influence $A$, and initial sphere of influence $A_0\subset V_0$ such that (\ref{condition:scalarproduct}) holds. We assume that we observe $X$ on the finite sampling sets $D_n\subset \Z^m$, such that 
\begin{gather}\label{condition:samplingset}
	\lim_{n\rightarrow\infty} |D_n|=\infty \text{ and }\lim_{n\rightarrow\infty} \frac{|D_n|}{|\partial D_n|}=0.
\end{gather}
We note that this includes in particular the equidistant sampling 
\begin{gather}\label{equation:equidistantsample}
	E_n=(0,n]^m\cap\Z^m\text{ such that }|E_n|=n^{m}, n\in\N.
\end{gather} 
The sample mean of the random field $X$ is then defined as
\begin{gather}\label{equation:samplemean}
	\frac{1}{|D_n|}\sum_{u\in D_n} X_{u}.
\end{gather}
If $\int_{\norm{x}>1}\norm{x}\nu(dx)<\infty$, we define the centered MMAF $\tilde{X}_u=X_u-E[X_u]$ and the sample autocovariance on $E_n$ at lag $k\in\N_0\times\Z^{m-1}$
\begin{gather}\label{equation:sampleautocovariance}
	\frac{1}{|E_{n-\tilde{k}}|}\sum_{u\in E_{n-\tilde{k}}} \tilde{X}_{u}\tilde{X}_{u+k},\ k\in\N_0\times\Z^{m-1},
\end{gather}
where $\tilde{k}=|k|$. Let us start by analyzing the asymptotic properties of the sample mean (\ref{equation:samplemean}) for a centered $(A,\Lambda)$-influenced MMAF.

\begin{Theorem}\label{theorem:thetasamplemean}
	Let $X=(X_u)_{u\in\Z^m}$ be an $(A,\Lambda)$-influenced MMAF as defined in (\ref{equation:discreteinfluencedMMAfield}) such that $\int_{\norm{x}>1}\norm{x}^{2+\delta}\nu(dx)<\infty$, $\gamma+\int_{\norm{x}>1}x\nu(dx)=0$ and $f\in L^2(S\times\R^{m},\pi\otimes\lambda)\cap L^{2+\delta}(S\times\R^{m},\pi\otimes\lambda)$ for some $\delta>0$. Assume that $X$ has $\theta$-lex-coefficients satisfying $\theta_X(h)=\mathcal{O}(h^{-\alpha})$, where $\alpha>m(1+\frac{1}{\delta})$. Then,
	\begin{gather*}
		\Sigma=\sum_{k\in\Z^m}E[X_0X_k']
	\end{gather*}
	is finite, positive semidefinite and
	\begin{gather}\label{eq:mmathetaclt}
		\frac{1}{|D_n|^{\frac{1}{2}}}\sum_{j\in D_n}X_j\underset{n\ra\infty}{\xrightarrow{\makebox[2em][c]{d}}}N(0,\Sigma).
	\end{gather}
\end{Theorem}
\begin{proof}
	By \cite[Theorem 3.6]{PV2017}, it follows that an MMAF is ergodic. Then, the result follows from Corollary \ref{corollary:ergodicclt}.
\end{proof}

In the theorem above, the initial sphere of influence $A_0$ must satisfy (\ref{condition:scalarproduct}). Additionally, we observe a trade-off between moment conditions on $X$ and the decay rate of the $\theta$-lex coefficients. 
However, one can derive similar results for the sample mean of an MMAF by relaxing condition (\ref{condition:scalarproduct}) and exploiting the second order moment structure of an MMAF. On the other hand, the following technique does not carry over to higher-order moments.

\begin{Theorem}\label{theorem:thetasamplemeanspecial}
	Let $X=(X_u)_{u\in\Z^m}$ be an $(A,\Lambda)$-influenced MMAF defined by
	\begin{gather*}
		X_u=\int_S\int_{A_u}f(A,u-s)\Lambda(dA,ds),
	\end{gather*}
	with translation invariant sphere of influence $A$ and initial sphere of influence $A_0\subset V_0$ such that $\gamma+\int_{\norm{x}>1}x\nu(dx)=0$ and $E[\norm{X_0}^{2}]<\infty$. Assume that $X$ has $\theta$-lex-coefficients satisfying $\theta_X(h)=\mathcal{O}(h^{-\alpha})$, where $\alpha>m$. Then,
	\begin{gather*}
		\Sigma=\sum_{k\in\Z^m}E[X_0X_k']
	\end{gather*}
	is finite, positive definite and
	\begin{gather}\label{eq:mmathetacltspecial}
		\frac{1}{|D_n|^{\frac{1}{2}}}\sum_{j\in D_n}X_j\underset{n\ra\infty}{\xrightarrow{\makebox[2em][c]{d}}}N(0,\Sigma).
	\end{gather}
\end{Theorem}
\begin{proof}
	See Section \ref{sec5-3}.
\end{proof}

To lighten notation, we assume in the following that $X$ is real-valued and centered, i.e., $E[X_0]=0$. In order to derive asymptotic properties for the distribution of (\ref{equation:sampleautocovariance}), we need to show weak dependence properties of the random field $Y=(Y_{j,k})_{j\in\Z^m}$ defined as
\begin{gather}\label{equation:sampleautocovy}
	Y_{j,k}=X_jX_{j+k}-R(k),\ k\in\N_0\times\Z^{m-1},
\end{gather}
where 
\begin{gather*}
	R(k)=Cov(X_0,X_k)=E[X_0X_k']=\int_S\int_{A_0\cap A_k}\!\!\!\!f(A,-s)\Sigma_\Lambda f(A,k-s)'ds\pi(dA),
\end{gather*}
$k\in\N_0\times\Z^{m-1}$ with $\Sigma_\Lambda =\Sigma+\int_{\R^d}xx'\nu(dx)$ for an $(A,\Lambda)$-influenced MMAF $X$ with characteristic quadruplet $(\gamma,\Sigma,\nu,\pi)$. The last equality follows from Proposition \ref{proposition:MMAmoments}.

\begin{Proposition}\label{proposition:deltavererbungtheta}
	Let $X=(X_u)_{u\in\Z^m}$ be a real-valued $(A,\Lambda)$-influenced MMAF as defined in (\ref{equation:discreteinfluencedMMAfield}) such that $E[X_0]=0$ and $E[\norm{X_0}^{2+\delta}]<\infty$ for some $\delta>0$ with $\theta$-lex-coefficients $\theta_X$. Then, $(Y_{j,k})_{j\in\Z^m}$, $k\in\N_0\times\Z^{m-1}$ as defined in (\ref{equation:sampleautocovy}) is $\theta$-lex-weakly dependent with coefficients
	\begin{gather*}
		\theta_Y(h)\leq\mathcal{C}\left(\sqrt{2}\hat{\theta}^{(i)}_X\left(h-\psi^{-1}(|k|)\right)\right)^{\frac{\delta}{1+\delta}},
	\end{gather*}
	where $\mathcal{C}$ is a constant independent of $h$, $\psi$ as defined in $(\ref{equation:psi})$, and $\hat\theta_X^{(i)}(\cdot)$ is defined as in Proposition \ref{proposition:vectorinfluencedweaklydep}.
	
	Furthermore, in the finite variation case and for $\hat\theta_X^{(iii)}(\cdot)$ defined as in Proposition \ref{proposition:vectorinfluencedweaklydep}, it holds
	\begin{gather*}
		\theta_Y(h)\leq\mathcal{C}\left(2\hat{\theta}^{(iii)}_X\left(h-\psi^{-1}(|k|)\right)\right)^{\frac{\delta}{1+\delta}}.
	\end{gather*}
\end{Proposition}
\begin{proof}
	Consider the 2-dimensional process $Z=(X_j,X_{j+k})_{j\in\Z^m}$ with $k\in\N_0\times\Z^{m-1}$. Proposition \ref{proposition:vectorinfluencedweaklydep} implies that $Z$ is $\theta$-lex-weakly dependent and from the proof we obtain
	\begin{gather*}
		\theta_Z(h)\leq\sqrt{2}\hat{\theta}^{(i)}_X(h-\psi^{-1}(|k|))\text{ for $\psi(h)>|k|$}.
	\end{gather*} 
	Consider the function $h:\R^2\ra\R$ such that $h(x_1,x_2)=x_1x_2$. The function $h$ satisfies the assumptions of Proposition \ref{proposition:mmathetahereditary} for $p=2+\delta$, $c=1$ and $a=2$. Considering $h(Z)=X_jX_{j+k}$, we obtain the $\theta$-lex-coefficients of $(Y_{j,k})_{j\in\Z^m}$
	\begin{gather*}
		\theta_Y(h)\leq\mathcal{C}(\sqrt{2}\hat{\theta}^{(i)}_X(h-\psi^{-1}(|k|)))^{\frac{\delta}{1+\delta}} \text{ for $\psi(h)>|k|$.}
	\end{gather*} 
	The coefficients for the finite variation case can be obtained from Proposition \ref{proposition:mmathetahereditary} and (\ref{eq:mmainfluencedvectorweakly}).
\end{proof}

The next corollary gives asymptotic properties of the sample autocovariances (\ref{equation:sampleautocovariance}) for $(A,\Lambda)$-influenced MMAF, i.e. we can give a distributional limit theorem for the process $(Y_{j,k})_{j\in\Z^m}$ by determining the asymptotic distribution of
\begin{gather*}
	\frac{1}{|E_{n-\tilde{k}}|^{\frac{1}{2}}}\sum_{j\in E_{n-\tilde{k}}}Y_{j,k},\ k\in\N_0\times\Z^{m-1},
\end{gather*}
where $\tilde{k}=|k|$.

\begin{Corollary}\label{corollary:sampleautocovariancetheta}
	Let $X=(X_u)_{u\in\Z^m}$ be a real-valued $(A,\Lambda)$-influenced MMAF as defined in (\ref{equation:discreteinfluencedMMAfield}) such that $E[X_0]=0$ and $E[\norm{X_0}^{4+\delta}]<\infty$ for some $\delta>0$. Let $\hat\theta_X^{(i)}$ be defined as in Proposition \ref{proposition:vectorinfluencedweaklydep}.  If $\hat{\theta}_X^{(i)}(h)=\mathcal{O}(h^{-\alpha})$ for $\alpha>m\left(1+\frac{1}{\delta}\right)(\frac{3+\delta}{2+\delta})$, then
	\begin{gather*}
		\Sigma=\sum_{l\in\Z^m}Cov\left(\left(\begin{array}{c}
			Y_{0,0}\\
			\vdots \\
			Y_{0,k}
		\end{array}\right),
		\left(\begin{array}{c}
			Y_{l,0}\\
			\vdots \\
			Y_{l,k}
		\end{array}\right)
		\right)=\sum_{l\in\Z^m}Cov\left(\left(\begin{array}{c}
			X_0X_0\\
			\vdots \\
			X_0X_k
		\end{array}\right),
		\left(\begin{array}{c}
			X_lX_l\\
			\vdots \\
			X_lX_{l+k}
		\end{array}\right)
		\right)
	\end{gather*}
	is finite, positive semidefinite and 
	\begin{gather*}
		\frac{1}{|E_{n-\tilde{k}}|^{\frac{1}{2}}}\sum_{j\in E_{n-\tilde{k}}}
		\left(\begin{array}{c}
			Y_{j,0}\\
			\vdots \\
			Y_{j,k}
		\end{array}\right)
		\underset{N\ra\infty}{\xrightarrow{\makebox[2em][c]{d}}}N\left(0,\Sigma\right),
	\end{gather*}
	where $\tilde{k}=|k|$.
\end{Corollary}
\begin{proof}
	Analogous to Theorem \ref{theorem:thetasamplemean} we obtain the stated convergence using Proposition \ref{proposition:deltavererbungtheta}. 
\end{proof}

\begin{Corollary}\label{corollary:samplemomentsofhigherorder}
	Let $X=(X_u)_{u\in\Z^m}$ be a real-valued $(A,\Lambda)$-influenced MMAF as defined in (\ref{equation:discreteinfluencedMMAfield}) and $p\geq1$ such that $E[|X_0|^{2p+\delta}]<\infty$ for some $\delta>0$.
	Let $\hat\theta_X^{(i)}$ be defined as in Proposition \ref{proposition:vectorinfluencedweaklydep}.  If $\hat{\theta}_X^{(i)}(h)=\mathcal{O}(h^{-\alpha})$ for $\alpha>m\left(1+\frac{1}{\delta}\right)(\frac{2p-1+\delta}{p+\delta})$, then
	\begin{gather*}
		\Sigma=\sum_{k\in\Z^m}Cov(X_0^{p},{X_k^{p}})
	\end{gather*}
	is finite, positive semidefinite and
	\begin{gather*}
		\frac{1}{|E_{n}|^{\frac{1}{2}}}\sum_{j\in E_{n}}(X_j^p-E[X_0^p])\underset{N\ra\infty}{\xrightarrow{\makebox[2em][c]{d}}}N(0,\Sigma).
	\end{gather*}
	
\end{Corollary}

\begin{Remark}\label{remark:gmmestimation}
	The theory developed in this section is an essential step in showing the asymptotic normality of parametric estimators based on moment functions as the generalized method of moments (for a comprehensive introduction, see \cite{H2005}). The weak dependence properties and related central limit theorems analyzed in this section find application in the study of the GMM estimators presented in \cite[Section 6.1]{CS2018}, where the authors analyze parametric estimators of the supOU process.
\end{Remark}

\subsection{Weak dependence properties of non-influenced MMAF}
\label{sec3-5}
We now consider a general MMAF $X=(X_t)_{t\in\R^m}$ as defined in (\ref{equation:MMAfield}), i.e.
\begin{gather*}
	X_t=\int_S\int_{\R^m} f(A,t-s)\Lambda(dA,ds),\ t\in\R^m,
\end{gather*}
and discuss under which assumptions a non-influenced MMAF is $\eta$-weakly dependent.
Note that we do not demand any additional assumption on the structure of $X$ as assumed in Section \ref{sec3-2} and \ref{sec3-3}. 

\begin{Proposition}\label{proposition:mmaetaweaklydep}
	Let $\Lambda$ be an $\R^d$-valued L\'evy basis with characteristic quadruplet $(\gamma,\Sigma,\nu,\pi)$ and $f:S\times\R^{m}\ra M_{n\times d}(\R)$ a $\mathcal{B}(S\times\R^m)$-measurable function. Consider the MMAF $X=(X_t)_{t\in\R^m}$ with
	\begin{gather*}
		X_t=\int_S\int_{\R^m}f(A,t-s)\Lambda(dA,ds), \ t\in\R^m.
	\end{gather*}
	\begin{enumerate}[(i)]
		\item If $\int_{\norm{x}>1}\norm{x}^2\nu(dx)<\infty$, $\gamma+\int_{\norm{x}>1}x\nu(dx)=0$ and $f\in L^2(S\times\R^{m},\pi\otimes\lambda)$, then $X$ is $\eta$-weakly dependent with $\eta$-coefficients satisfying 
		\begin{gather*}
			\eta_X(h)\leq\Bigg(\int_S\int_{\left(\left(-\frac{h}{2},\frac{h}{2}\right)^m\right)^c}\textup{tr}(f(A,-s)\Sigma_\Lambda f(A,-s)')ds\pi(dA)\Bigg)^{\frac{1}{2}}=\hat{\eta}_X^{(i)}(h).
		\end{gather*}
		\item If $\int_{\norm{x}>1}\norm{x}^2\nu(dx)<\infty$ and $f\in L^2(S\times\R^{m},\pi\otimes\lambda)\cap L^1(S\times\R^{m},\pi\otimes\lambda)$, then $X$ is $\eta$-weakly dependent with $\eta$-coefficients satisfying 
		\begin{align*}
			\eta_X(h) & \leq 
			\bigg(\int_S\int_{\left(\left(-\frac{h}{2},\frac{h}{2}\right)^m\right)^c}\textup{tr}(f(A,-s)\Sigma_\Lambda f(A,-s)')ds\pi(dA)\\
			&+\Big\lVert\int_S\int_{\left(\left(-\frac{h}{2},\frac{h}{2}\right)^m\right)^c} f(A,-s)\mu_\Lambda  ds\pi(dA)\Big\rVert^2\bigg)^{\frac{1}{2}}=\hat{\eta}_X^{(ii)}(h).
		\end{align*}
		\item If $\int_{\R^d}\norm{x}\nu(dx)<\infty$, $\Sigma=0$ and $f\in L^1(S\times\R^m,\pi\otimes\lambda)$ with $\gamma_0$ as in (\ref{equation:gammazero}), then $X$ is $\eta$-weakly dependent with $\eta$-coefficients satisfying 
		\begin{align*}
			\eta_X(h) &\leq
			\int_S\int_{\left(\left(-\frac{h}{2},\frac{h}{2}\right)^m\right)^c}\norm{f(A,-s)\gamma_0}ds\pi(dA)\\
			&+\int_S\int_{\left(\left(-\frac{h}{2},\frac{h}{2}\right)^m\right)^c} \int_{\R^d}\norm{f(A,-s)x} \nu(dx)ds\pi(dA)=\hat{\eta}_X^{(iii)}(h).
		\end{align*}
		\item If $\int_{\norm{x}>1}\norm{x}\nu(dx)<\infty$ and $f\in L^1(S\times\R^m,\pi\otimes\lambda)\cap L^2(S\times\R^m,\pi\otimes\lambda)$, then $X$ is $\eta$-weakly dependent with $\eta$-coefficients satisfying 
		\begin{align*}
			\eta_X(h) &\leq
			\bigg(\int_S\int_{\left(\left(-\frac{h}{2},\frac{h}{2}\right)^m\right)^c}\textup{tr}(f(A,-s)\Sigma_{\Lambda_1} f(A,-s)')ds\pi(dA)\\
			&+\Big\lVert\int_S\int_{\left(\left(-\frac{h}{2},\frac{h}{2}\right)^m\right)^c} f(A,-s)\gamma  ds\pi(dA)\Big\rVert^2\bigg)^{\frac{1}{2}}\\
			&+\int_S\int_{\left(\left(-\frac{h}{2},\frac{h}{2}\right)^m\right)^c} \int_{\norm{x}>1}\norm{f(A,-s)x} \nu(dx)ds\pi(dA)=\hat{\eta}_X^{(iv)}(h).
		\end{align*}
	\end{enumerate}
	The results above hold for all $h>0$, where $\Sigma_\Lambda =\Sigma+\int_{\R^d}xx'\nu(dx)$, $\Sigma_{\Lambda_1} =\Sigma+\int_{\norm{x}\leq1}xx'\nu(dx)$ and $\mu_\Lambda =\gamma-\int_{\norm{x}\geq1}x\nu(dx)$.
\end{Proposition}
\begin{proof}
	See Section \ref{sec5-4}.
\end{proof}

Analogous to Proposition \ref{proposition:vectorinfluencedweaklydep} we obtain the following result.

\begin{Proposition}\label{proposition:vectorgeneralmmaweaklydep}
	Let $\Lambda$ be an $\R^d$-valued L\'evy basis with characteristic quadruplet $(\gamma,\Sigma,\nu,\pi)$ and $f:S\times\R^m\ra M_{1\times d}(\R)$ be a $\Lambda$-integrable, $\mathcal{B}(S\times\R^m)$-measurable function. Consider the real-valued MMAF 
	\begin{gather*}
		X_{t}=\int_S\int_{\R^m}f(A,t-s)\Lambda(dA,ds), \ t\in\R^{m}.
	\end{gather*}
	Then,
	\begin{gather*}
		Z_t:=\int_S\int_{\R^m}g(A,t-s)\Lambda(dA,ds), \ t\in\R^m,
	\end{gather*}
	is an MMAF, where $g(A,s)=(f(A,s),f(A,s-s_1),\ldots, f(A,s-s_{|S_k|}))'$ is a $\mathcal{B}(S\times\R^m)$-measurable function with values in $M_{(k+1)(2k+1)^{m-1}\times d}(\R)$ for $k\in\N_0$.\\
	If $X$ additionally satisfies the conditions of Proposition \ref{proposition:mmaetaweaklydep} (i), (ii), (iii) or (iv), then $Z$ is $\eta$-weakly dependent with coefficients respectively given by
	\begin{gather}
		\begin{aligned}\label{eq:mmavectorweakly}
			\eta_Z^{(i)}(h)\leq&\mathcal{D}\hat{\eta}_X^{(i)}(h-2k),\qquad\qquad 
			\eta_Z^{(ii)}(h)\leq&\mathcal{D} \hat{\eta}_X^{(ii)}(h-2k),\\
			\eta_Z^{(iii)}(h)\leq&\mathcal{C} \hat{\eta}_X^{(iii)}(h-2k)\text{ and}\qquad
			\eta_Z^{(iv)}(h)\leq&\mathcal{C}\hat{\eta}_X^{(iv)}(h-2k),
		\end{aligned}
	\end{gather}
	where $\mathcal{D}=|S_k|^{m/2}$, $\mathcal{C}=|S_k|^{m}$ for $h>2k$, and  $\hat{\eta}^{(\cdot)}(h)$ are defined as in Proposition \ref{proposition:mmaetaweaklydep}.
\end{Proposition}
\begin{proof}
	Analogous to Proposition \ref{proposition:vectorinfluencedweaklydep}. 
\end{proof}

\subsection{Sample moments of non-influenced MMAF}
\label{sec3-6}
Let us consider an $\R^n$-valued MMAF
\begin{gather}
	X=(X_u)_{u\in\Z^m} \text{ with } X_u=\int_S\int_{\R^m}f(A,u-s)\Lambda(dA,ds).\label{equation:discretegeneralMMAfield}
\end{gather}
As in Section \ref{sec3-2} we assume that we observe $X$ on a sequence of finite sampling sets $D_n\subset\Z^m$, such that (\ref{condition:samplingset}) holds.

\begin{Theorem}\label{theorem:etasamplemean}
	Let $(X_u)_{u\in\Z^m}$ be an MMAF as defined in (\ref{equation:discretegeneralMMAfield}) such that $E[X_0]=0$ and $E[\norm{X_0}^{2+\delta}]<\infty$ for some $\delta>0$. Assume that $X$ has $\eta$-coefficients satisfying $\eta_X(h)=\mathcal{O}(h^{-\beta})$, where $\beta>m\max\left(2,\left(1+\frac{1}{\delta}\right)\right)$.
	Then,
	\begin{gather}\label{equation:sigmaeta}
		\Sigma=\sum_{u\in\Z^m}Cov(X_0,X_u)=\sum_{u\in\Z^m}E[X_0X_u']
	\end{gather}
	is finite, positive semidefinite and 
	\begin{gather}\label{equation:etaasympnorm}
		\frac{1}{|D_n|^{\frac{1}{2}}}\sum_{u\in D_n}X_u\underset{n\ra\infty}{\xrightarrow{\makebox[2em][c]{d}}}N(0,\Sigma).
	\end{gather}
\end{Theorem}
\begin{proof}
	$X$ is $\lambda$-weakly dependent, see \cite[Definition 1]{DMT2008}. Then, \cite[Theorem 2]{DMT2008} implies the summability of $\sigma^2$ and the result stated in (\ref{equation:etaasympnorm}). The multivariate extension follows by using the Cram\'er-Wold device.
\end{proof}

\begin{Remark}
	Theorem \ref{theorem:etasamplemean} can be formulated as a functional central limit theorem, following \cite[Theorem 3]{DMT2008}. For $t\in [0,1]^m$, where $[0,1]^m$ denotes the $m$-fold Cartesian product of $[0,1]$, we set $S_n(t)=\sum_{j\in tE_n}X_j$ with $E_n$ as defined in (\ref{equation:equidistantsample}) and the additional assumption that $S_n(t)=0$ if one coordinate of $t$ equals zero. The product $tE_n$ has to be understood coordinatewise. Then, under the assumptions of Theorem \ref{theorem:etasamplemean} it holds that 
	\begin{gather}
		\frac{1}{n^{\frac{m}{2}}}S_n(t)  \underset{n\ra\infty}{\xrightarrow{\mathcal{D}([0,1]^m)}}\sigma W(t),
	\end{gather}
	where $W=\{W(t), t\in[0,1]^m\}$ denotes a Brownian sheet, i.e., a centered Gaussian process such that $Cov(W(t_1,\ldots,t_m),W(s_1,\ldots,s_m)')=\prod_{i=1}^m t_i \wedge s_i$ for all $t_1,\ldots,t_m,s_1,\ldots,s_m\in[0,1]$, and $\underset{n\ra\infty}{\xrightarrow{\mathcal{D}([0,1]^m)}}$ denotes the convergence in the Skorokhod space (see e.g. \cite[Section 3]{BW1971} for a definition of the Skorokhod topology on $[0,1]^m$).
\end{Remark}

Analogous to Proposition \ref{proposition:deltavererbungtheta} we show the following result.

\begin{Proposition}\label{proposition:deltavererbungeta}
	Let $(X_u)_{u\in\Z^m}$ be a real-valued MMAF as defined in (\ref{equation:discretegeneralMMAfield}) such that $E[X_0]=0$ and $E[\norm{X_0}^{2+\delta}]<\infty$ for some $\delta>0$.
	Then, $(Y_{j,k})_{j\in\Z^m}$, $k\in\N_0\times\Z^{m-1}$ as defined in (\ref{equation:sampleautocovy}) is $\eta$-weakly dependent with coefficients
	\begin{gather*}
		\eta_Y(h)\leq\mathcal{C}(\sqrt{2}\hat{\eta}^{(i)}_X(h-2|k|))^{\frac{\delta}{1+\delta}},
	\end{gather*}
	where $\mathcal{C}$ is a constant independent of $h$ and $\hat\eta_X^{(i)}(\cdot)$ is defined as in Proposition \ref{proposition:vectorgeneralmmaweaklydep}. 
	
	Furthermore, in the finite variation case and for $\hat\eta_X^{(iii)}(\cdot)$ defined as in Proposition \ref{proposition:vectorgeneralmmaweaklydep}, it holds
	\begin{gather*}
		\eta_Y(h)\leq\mathcal{C}(2\hat{\eta}^{(iii)}_X(h-2|k|))^{\frac{\delta}{1+\delta}}.
	\end{gather*}
\end{Proposition}

In the following, we give asymptotic properties of the sample autocovariances (\ref{equation:sampleautocovariance}).

\begin{Corollary}\label{corollary:sampleautocovarianceeta}
	Let $(X_u)_{u\in\Z^m}$ be a real-valued MMAF as defined in (\ref{equation:discretegeneralMMAfield}) such that $\int_{\norm{x}>1}\norm{x}^{4+\delta}\nu(dx)<\infty$, $\gamma+\int_{\norm{x}>1}x\nu(dx)=0$ and $f:S\times\R^m\ra M_{1\times d}(\R)$ satisfies $f\in L^2(S\times\R^{m},\pi\otimes\lambda)\cap L^{4+\delta}(S\times\R^{m},\pi\otimes\lambda)$ for some $\delta>0$. If $\hat{\eta}^{(i)}_X(h)=\mathcal{O}(h^{-\beta})$, with $\hat\eta_X^{(i)}$ defined as in Proposition \ref{proposition:vectorgeneralmmaweaklydep}, and $\beta>m\max\left(2,\left(1+\frac{1}{\delta}\right)\right)(\frac{3+\delta}{2+\delta})$, then
	\begin{gather*}
		\Sigma=\sum_{l\in\Z^m}Cov\left(\left(\begin{array}{c}
			Y_{0,0}\\
			\vdots \\
			Y_{0,k}
		\end{array}\right),
		\left(\begin{array}{c}
			Y_{l,0}\\
			\vdots \\
			Y_{l,k}
		\end{array}\right)
		\right)=\sum_{l\in\Z^m}Cov\left(\left(\begin{array}{c}
			X_0X_0\\
			\vdots \\
			X_0X_k
		\end{array}\right),
		\left(\begin{array}{c}
			X_lX_l\\
			\vdots \\
			X_lX_{l+k}
		\end{array}\right)
		\right)
	\end{gather*}
	is finite, positive semidefinite and 
	\begin{gather*}
		\frac{1}{|E_{n-\tilde{k}}|^{\frac{1}{2}}}\sum_{j\in E_{n-\tilde{k}}}
		\left(\begin{array}{c}
			Y_{j,0}\\
			\vdots \\
			Y_{j,k}
		\end{array}\right)
		\underset{N\ra\infty}{\xrightarrow{\makebox[2em][c]{d}}}N\left(0,\Sigma\right),
	\end{gather*}
	where $\tilde{k}=|k|$.
\end{Corollary}
\begin{proof}
	Analogous to Theorem \ref{theorem:etasamplemean} we obtain the stated convergence using Proposition \ref{proposition:deltavererbungeta}. 
\end{proof}

\begin{Remark}
	Note that for $m=1$ Theorem \ref{theorem:etasamplemean} improves the only existing central limit theorem for MMA processes based on $\eta$-weak dependence (see \cite[Theorem 4.1]{CS2018}) by reducing the necessary decay of the $\eta$-coefficients from $\beta>4+\frac{2}{\delta}$ to $\beta>\max\left(2,\left(1+\frac{1}{\delta}\right)\right)$.
\end{Remark}

\begin{Remark}
	Let $X$ be an $(A,\Lambda)$-influenced MMAF satisfying the conditions of Proposition \ref{proposition:mmathetaweaklydep} (i). Then $X$ is $\theta$-lex- and $\eta$-weakly dependent with the same weak dependence coefficients and both the asymptotic results in Section \ref{sec3-4} and \ref{sec3-6} can be applied. \\
	Note that the asymptotic results in Section \ref{sec3-4} hold under weaker decay demands for the weak dependence coefficients compared to the results in Section  \ref{sec3-6}.
\end{Remark}

\subsection{Example of $(A,\Lambda)$-influenced MMAF: MSTOU processes}
\label{sec3-7}
We apply the developed asymptotic theory to mixed spatio-temporal Ornstein-Uhlenbeck (MSTOU) processes. MSTOU processes were introduced in \cite{NV2017} and extend the spatio-temporal Ornstein-Uhlenbeck (STOU) processes (see \cite{BS2004,NV2015}) by additionally mixing the mean reversion parameter. This extension allows versatile modeling of short-range as well as long-range dependence structures in space-time. 

In the following, we will treat the temporal and spatial domains separately. MSTOU processes are an example of $(A,\Lambda)$-influenced MMAF where the sphere of influence is a family of ambit sets, i.e. $A_t(x)\subset \R\times\R^m$ such that
\begin{gather}\label{equation:ambitset}
	\begin{cases}
		A_t(x)=A_0(0)+(t,x), \text{ (Translation invariant)}\\
		A_s(x)\subset A_t(x),\\
		A_t(x)\cap((t,\infty))\times\R^m=\emptyset.\text{ (Non-anticipative)}.
	\end{cases}
\end{gather}

\begin{Proposition}
	\label{mstou}
	Let $\Lambda$ be a real-valued L\'evy basis on $(0,\infty)\times\R\times\R^m$ with characteristic quadruplet $(\gamma,\Sigma,\nu,\pi)$ such that $\int_{|x|>1}x^2\nu(dx)<\infty$ and  $f(\lambda)$ be the density function of $\pi$ (i.e. the mean reversion parameter $\lambda$) with respect to the Lebesgue measure. Furthermore, let $A=(A_t(x))_{(t,x)\in\R\times\R^m}$ be an ambit set. If
	\begin{gather*}\label{equations:MSTOUcond1}
		\int_0^\infty \int_{A_t(x)} \exp(-\lambda(t-s)) \,ds\, d\xi\, f(\lambda) \,d\lambda<\infty,
	\end{gather*}
	then the $(A,\Lambda)$-influenced MMAF
	\begin{gather*}\label{equation:MSTOU}
		Y_t(x)=\int_0^\infty \int_{A_t(x)}\exp(-\lambda(t-s)) \, \Lambda(d\lambda,ds,d\xi),\quad (t,x)\in\R\times\R^d
	\end{gather*}
	is well defined and we call $Y_t(x)$ a mixed spatio-temporal Ornstein-Uhlenbeck (MSTOU) process.
\end{Proposition}
\begin{proof}
	Follows immediately from \cite[Corollary 1]{NV2017}.
\end{proof}

To calculate the assumptions under which the asymptotic results of Section \ref{sec3-3} hold, it becomes necessary to specify a family of ambit sets. In the following, we will consider c-class MSTOU processes, a sub-class of the $g$-class MSTOU processes as given in \cite[Definition 9]{NV2017}.

\begin{Definition}
	Let $Y_t(x)$ be an MSTOU process as in Proposition \ref{mstou}. If, for a constant 
	$c>0$,
	\begin{gather*}
		A_t(x)=\{ (s,\xi): s\leq t, \norm{x-\xi}\leq c|t-s| \},
	\end{gather*}
	then $Y_t(x)$ is called a c-class MSTOU process.
	A c-class MSTOU process is well defined if 
	\begin{gather}
		\int_0^\infty \frac{1}{\lambda^{m+1}}f(\lambda) d\lambda<\infty.
		\label{equations:gclassMSTOUcond1}
	\end{gather}
\end{Definition}

The next theorem expresses the $\theta$-lex coefficients of c-class MSTOU processes in terms of the characteristic quadruplet of the driving L\'evy basis. We note that $A_0(0)$ is a closed convex proper cone with Lebesgue measure strictly greater than zero satisfying (\ref{condition:scalarproduct}). From (\ref{equation:psi}) it follows that $\psi(h)=\frac{1}{\sqrt{c^2+1}}\frac{h}{\sqrt{m+1}}$.

\begin{Theorem}
	Let $(Y_t(x))_{(t,x)\in\R\times\R^m}$ be a c-class MSTOU process and $(\gamma,\Sigma,\nu,\pi)$ the characteristic quadruplet of its driving L\'evy basis. Moreover, let $f(\lambda)$ be the density of $\pi$ with respect to the Lebesgue measure. 
	\begin{enumerate}[(i)]
		\item If $\int_{|x|>1} x^2\nu(dx)<\infty$ and $\gamma+\int_{|x|>1}x \,\nu(dx)=0$, then $Y_t(x)$ is $\theta$-lex-weakly dependent. Let $c\in(0,1]$, then for
		\begin{alignat*}{2}
			m&=1,\qquad &&\theta_Y(h)\leq\Bigg(2c\Sigma_\Lambda\int_0^\infty \frac{(2\lambda \psi(h)+1)}{\lambda^2}e^{-2\lambda \psi(h)}f(\lambda)d\lambda \Bigg)^{\frac{1}{2}},\\
			\textrm{and for} & && \\   
			m& \geq 2,\qquad &&\theta_Y(h)\leq2\left(V_m(c)\Sigma_\Lambda  \int_0^\infty 
			\frac{m!\sum_{k=0}^m \frac{1}{k!}(2\lambda \psi(h))^k }{(2\lambda)^{m+1}} e^{-2\lambda \psi(h)}
			f(\lambda)d\lambda\right)^{\frac{1}{2}}.
		\end{alignat*}
		Let $c>1$, then for 
		\begin{alignat*}{2}
			m&=1,\qquad &&\theta_Y(r) \leq 2 \Big( c\Sigma_\Lambda\int_0^\infty  \frac{(2\lambda \frac{\psi(h)}{c}+1)}{2\lambda^2}e^{-2\lambda \frac{\psi(h)}{c}}ds f(\lambda)d\lambda\Big)^{\frac{1}{2}},\\
			\textrm{and for} & && \\   
			m&\geq 2,\qquad &&\theta_Y(h)\leq2\Bigg(V_m(c)\Sigma_\Lambda  \int_0^\infty \frac{m!\sum_{k=0}^m \frac{1}{k!}(2\frac{\lambda \psi(h)}{c})^k }{(2\lambda)^{m+1}} e^{-2\frac{\lambda \psi(h)}{c}}f(\lambda)d\lambda\Bigg)^{\frac{1}{2}}.
		\end{alignat*}
		\item If $\int_{\R} |x| \,\nu(dx)<\infty$, $\Sigma=0$ and $\gamma_0$ as defined in (\ref{equation:gammazero}), then $Y_t(x)$ is $\theta$-lex-weakly dependent. Let $c \in (0,1]$, then for
		\begin{alignat*}{2}
			m&\in\N,\qquad && \theta_Y(h)\leq2V_m(c) \gamma_{abs} \left(\int_0^\infty \frac{m!\sum_{k=0}^m \frac{1}{k!}(\lambda \psi(h))^k }{\lambda^{m+1}} e^{-\lambda \psi(h)}f(\lambda)d\lambda\right),
		\end{alignat*}
		whereas for $c>1$ and
		\begin{alignat*}{2}
			m&\in\N, \qquad && \theta_Y(h)\leq 2 V_m(c)\gamma_{abs}\Bigg(\int_0^\infty \frac{m!\sum_{k=0}^m \frac{1}{k!}(\frac{\lambda \psi(h)}{c})^k }{\lambda^{m+1}} e^{-\frac{\lambda \psi(h)}{c}}f(\lambda)d\lambda\Bigg).
		\end{alignat*}
	\end{enumerate}
	The results above hold for all $h>0$, where $\gamma_{abs}=|\gamma_0|+\int_\R|x|\nu(dx)$, $V_m(c)=\frac{\left(\Gamma(\frac{1}{2})c\right)^m }{\Gamma(\frac{m}{2}+1)}$ denotes the volume of the $m$-dimensional ball with radius $c$, $\psi(h)=\frac{1}{\sqrt{c^2+1}}\frac{h}{\sqrt{m+1}}$ and $\Sigma_\Lambda =\Sigma+\int_{\R}x^2\nu(dx)$. 
\end{Theorem}
\begin{proof}
	\begin{enumerate}
		\item [(i)] Let us consider the case $m=1$. From Proposition \ref{proposition:mmathetaweaklydep} we deduce
		\begin{gather}\label{eq:MSTOUtemp}
			\theta_Y(h)\leq 2 \Big( \Sigma_\Lambda\int_0^\infty \int_{A_0(0)\cap V_{0}^{\psi(h)}} \exp(2s\lambda) ds \, d\xi \, f(\lambda)d\lambda\Big)^{\frac{1}{2}}.
		\end{gather}
		As first step, one has to evaluate the truncated integration set $A_0(0)\cap V_{0}^{\psi(h)}$. Depending on the width of $A_0(0)$, we distinguish the two cases illustrated in the following figures. Figure \ref{Plotcases1} and \ref{Plotcases2} consider the case $c\in(0,1]$ and Figure \ref{Plotcases3} and \ref{Plotcases4} cover the case $c>1$.
		\begin{figure}[H]
			\begin{minipage}[H]{3cm}
				\centering
				\includegraphics[width=3.5cm]{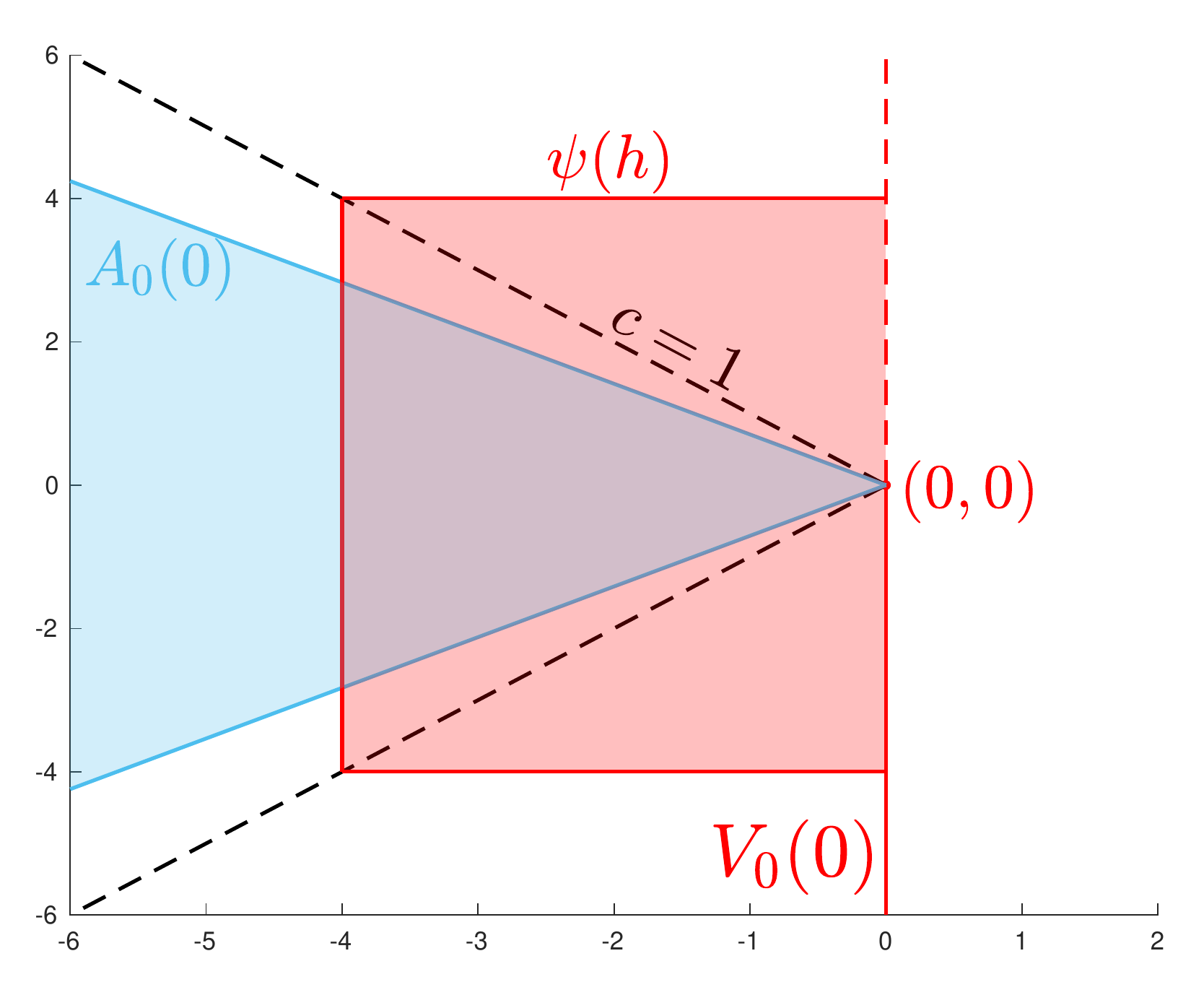}\label{figurePlotcases1}\\
				\caption{\small{Integration set $A_0(0)$ with $(V_{(0,0)}^h)^c$ for $c=\frac{1}{\sqrt{2}}$ and $h=4\sqrt{3}$.}}\label{Plotcases1}
			\end{minipage}
			\hfill
			\begin{minipage}[H]{3cm}
				\centering
				\includegraphics[width=3.5cm]{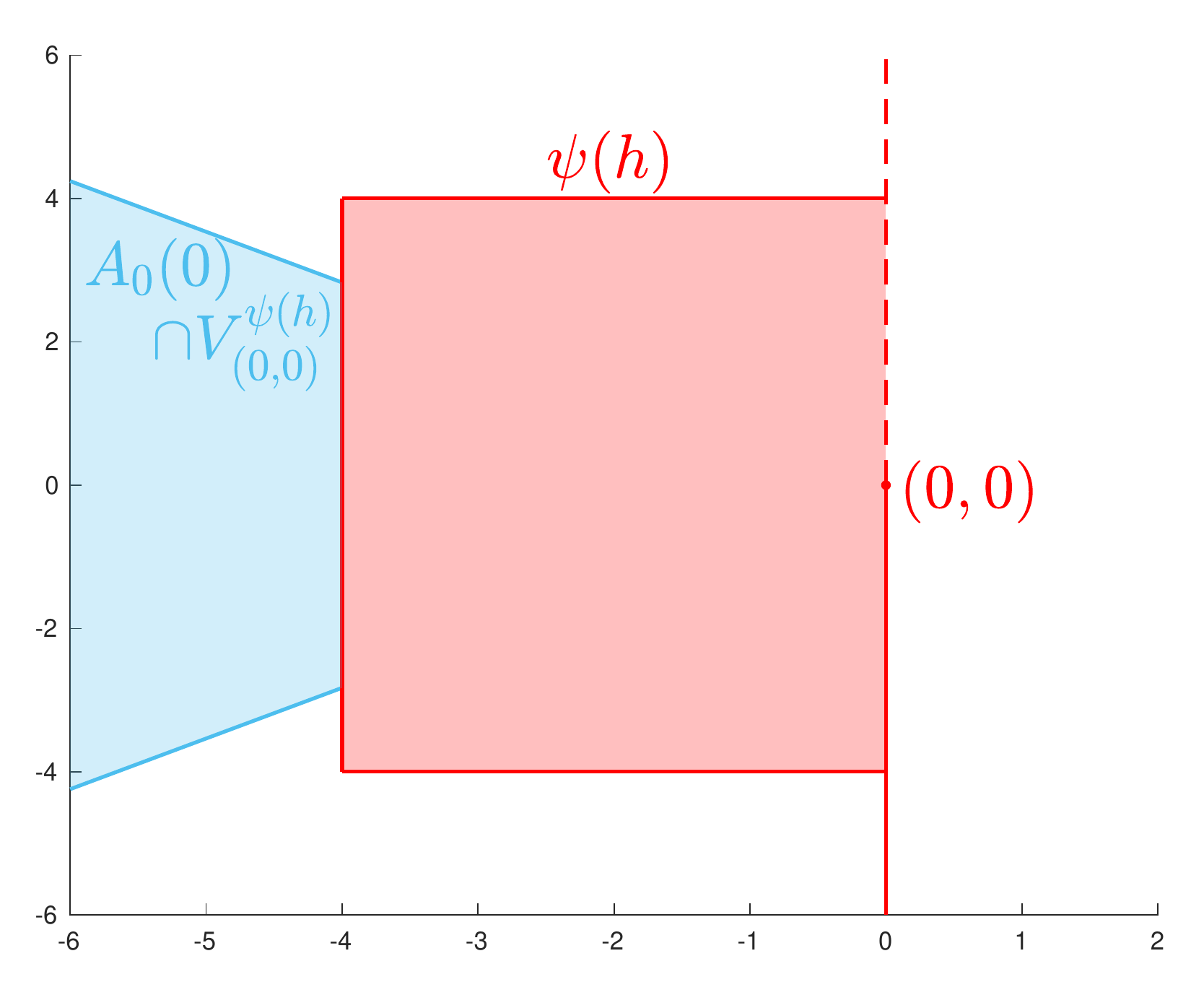}\label{figurePlotcases2}\\
				\caption{\small{Truncated set $A_0(0)\cap V_{(0,0)}^{\psi(h)}$ for $c=\frac{1}{\sqrt{2}}$ and $h=4\sqrt{3}$.}}
				\label{Plotcases2}
			\end{minipage}
			\hfill
			\begin{minipage}[H]{3cm}
				\centering
				\includegraphics[width=3.5cm]{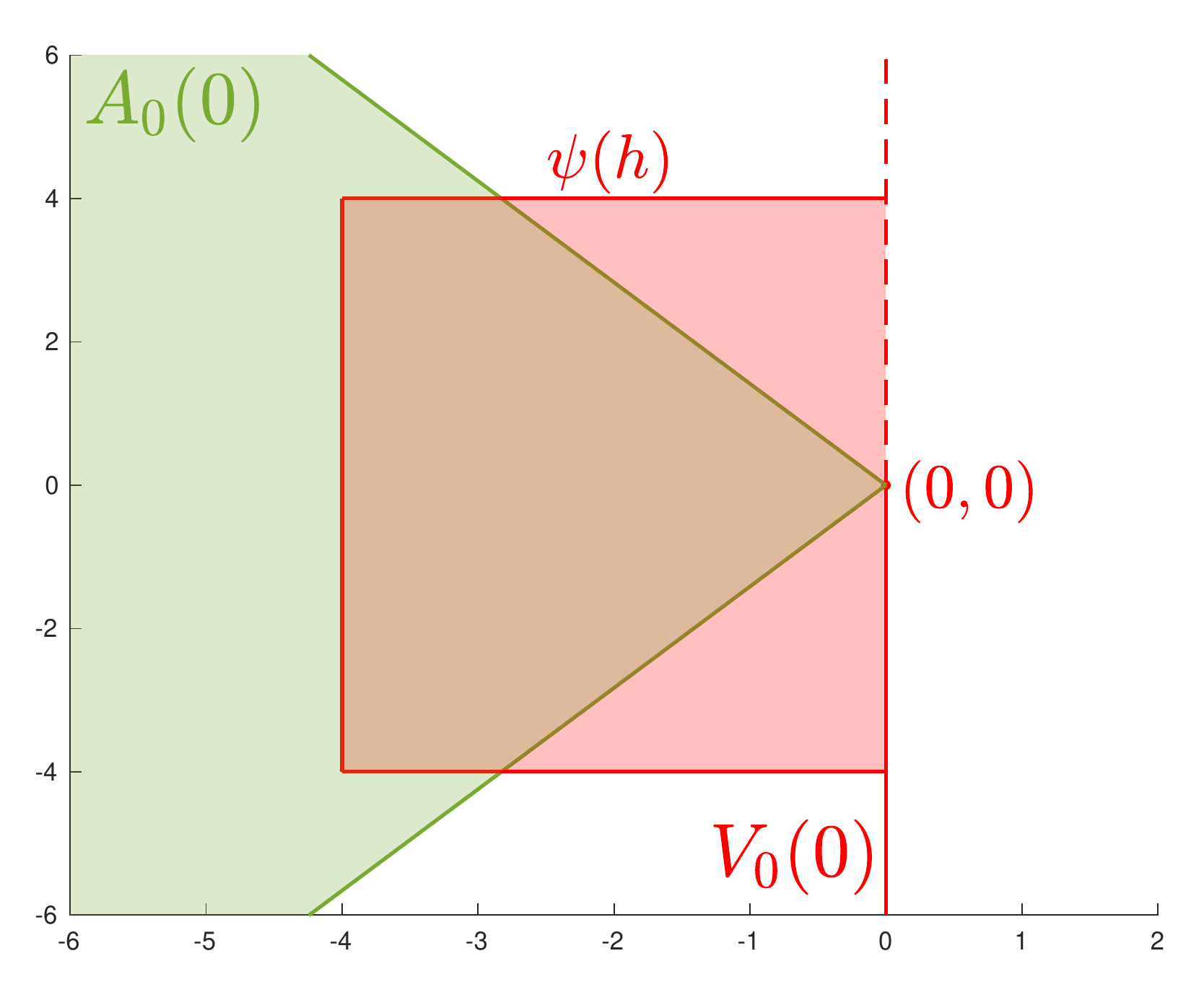}\label{figurePlotcases3}\\
				\caption{\small{Integration set $A_0(0)$ with $(V_{(0,0)}^h)^c$ for $c=\sqrt{2}$ and $h=4\sqrt{6}$.}}
				\label{Plotcases3}
			\end{minipage}
			\hfill
			\begin{minipage}[H]{3cm}
				\centering
				\includegraphics[width=3.5cm]{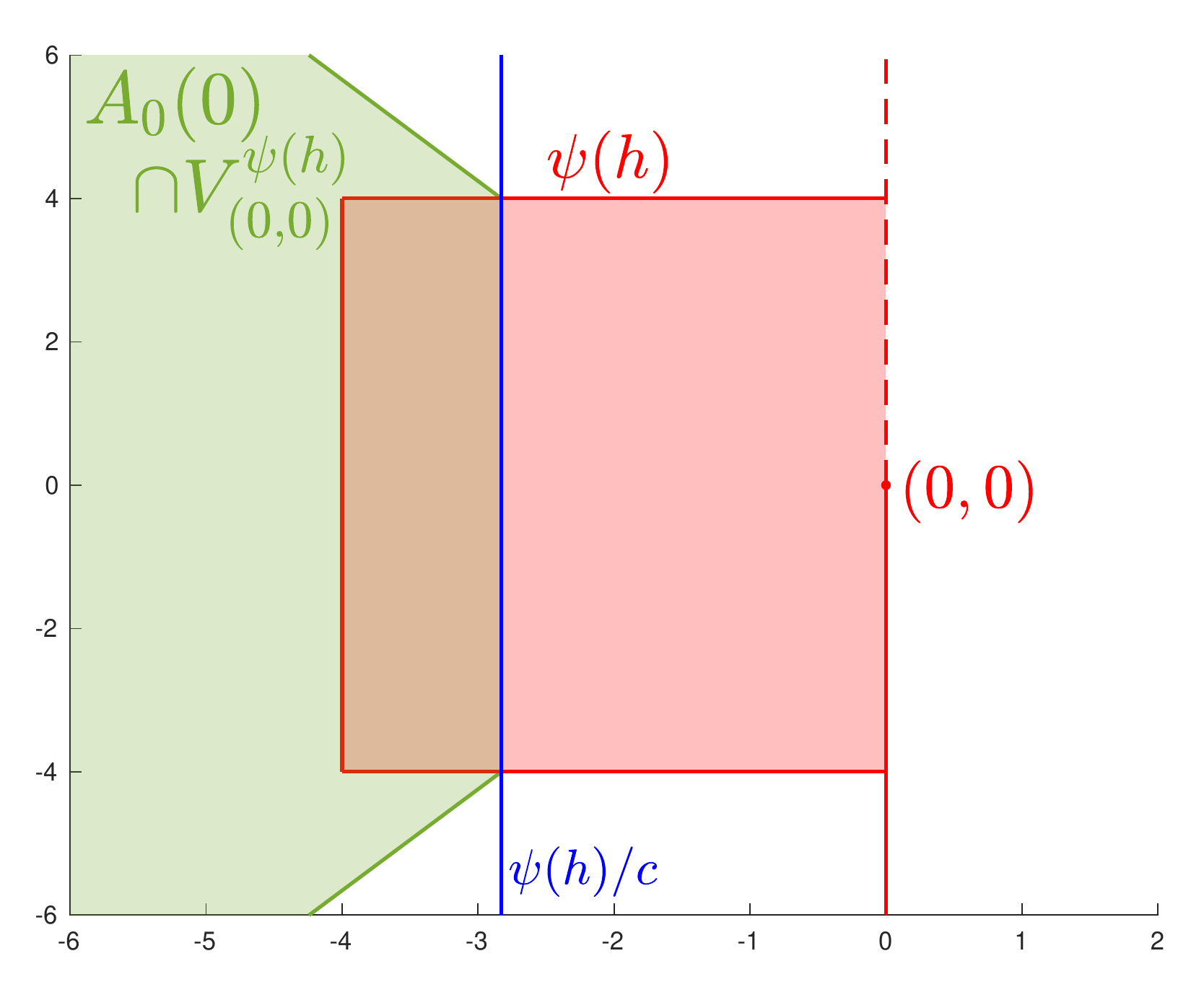}\label{figurePlotcases4}\\
				\caption{\small{Truncated set $A_0(0)\cap V_{(0,0)}^{\psi(h)}$ for $c=\sqrt{2}$ and $h=4\sqrt{6}$.}}
				\label{Plotcases4}
			\end{minipage}
		\end{figure} 
		Let $c\in (0,1]$, then (\ref{eq:MSTOUtemp}) is equal to
		\begin{align*}
			&2 \Big(\! \Sigma_\Lambda\!\!\int_0^\infty\!\! \int_{-\infty}^{-\psi(h)}\!\! \int_{\|\xi\|\leq cs} d\xi \,  e^{2s\lambda} ds f(\lambda)d\lambda\Big)^{\frac{1}{2}}\!\! =2 \Big( \Sigma_\Lambda\int_0^\infty \!\!\int_{-\infty}^{-\psi(h)} (-2cs) e^{2s\lambda} ds f(\lambda)d\lambda\Big)^{\frac{1}{2}} \nonumber\\
			&= \Big( 2c\Sigma_\Lambda\int_0^\infty  \frac{(2\lambda \psi(h)+1)}{\lambda^2}e^{-2\lambda \psi(h)} f(\lambda)d\lambda\Big)^{\frac{1}{2}}.\nonumber
		\end{align*}
		The integral $\int_{\|\xi\|\leq cs} d\xi$ is the volume of an $m$-dimensional ball of radius $cs$, which for $m=1$ is equal to $-2cs$. For $c>1$ we can bound (\ref{eq:MSTOUtemp}) by 
		\begin{align*}
			2 \Big( \Sigma_\Lambda\!\!\int_0^\infty\! \!\!\int_{-\infty}^{-\frac{\psi(h)}{c}}\!\!\!\! \!(-2cs) e^{2s\lambda} ds f(\lambda)d\lambda \Big)^{\frac{1}{2}}\!
			=2 \Big( \!c\Sigma_\Lambda\!\!\int_0^\infty \!\! \frac{(2\lambda \frac{\psi(h)}{c}+1)}{2\lambda^2}e^{-2\lambda \frac{\psi(h)}{c}}ds f(\lambda)d\lambda\Big)^{\frac{1}{2}}. 	
		\end{align*}
		In a similar way, one can derive the $\theta$-lex coefficients for $m \geq 2$.
		\item [(ii)] Analogous to (i).
	\end{enumerate}
\end{proof}

We now give explicit computations of the $\theta$-lex-coefficients of a c-class MSTOU process in the case in which the mean reverting parameter $\lambda$ is gamma distributed. For a $Gamma(\alpha,\beta)$ distributed mean reversion parameter $\lambda$, i.e. $f(\lambda)=\frac{\beta^\alpha}{\Gamma(\alpha)}\lambda^{\alpha-1}e^{-\beta\lambda}$ $\mathbb{1}_{[0,\infty)}(\lambda)$, the c-class MSTOU process is well defined if $\alpha>m+1$ and $\beta>0$ due to condition (\ref{equations:gclassMSTOUcond1}).

\begin{Theorem}
	Let $(Y_t(x))_{(t,x)\in\R\times\R^m}$ be a c-class MSTOU process and $(\gamma,\Sigma,\nu,\pi)$ the characteristic quadruplet of its driving L\'evy basis. Moreover, let the mean reversion parameter $\lambda$ be $Gamma(\alpha,\beta)$ distributed with $\alpha >m+1$ and $\beta>0$.
	\begin{enumerate}[(i)]
		\item If $\int_{|x|>1} x^2 \,\nu(dx)<\infty$ and $\gamma+\int_{|x|>1}x\nu(dx)=0$, then $Y_t(x)$ is $\theta$-lex-weakly dependent. Let $c\in[0,1]$, then for
		\begin{alignat*}{2}
			m&=1,\quad && \theta_Y(h)\leq2\left(\frac{c\Sigma_\Lambda \beta^\alpha}{2\Gamma(\alpha)} \left(\frac{\Gamma(\alpha-2)}{(2\psi(h)+\beta)^{\alpha-2}}+\frac{2\psi(h)\Gamma(\alpha-1)}{(2\psi(h)+\beta)^{\alpha-1}} \right)\right)^{\frac{1}{2}}, \\
			\textrm{and for}  & && \\
			m& \geq 2,\quad &&\theta_Y(h)\leq2\left(V_m(c)\frac{m!\Sigma_\Lambda \beta^{\alpha} }{2^{m+1}} \sum_{k=0}^m \frac{(2\psi(h))^k}{k!(2\psi(h)+\beta)^{\alpha-m-1+k}}\frac{\Gamma(\alpha-m-1+k)}{\Gamma(\alpha)}\right)^{\frac{1}{2}}.
		\end{alignat*}
		Let $c>1$, then for
		\begin{alignat*}{2}
			m&\in\N,\ \ &&\theta_Y(h)\leq2\Bigg(V_m(c)\frac{m!\Sigma_\Lambda \beta^{\alpha} }{2^{m+1}} \sum_{k=0}^m \frac{\left(\frac{2\psi(h)}{c}\right)^k}{k!\left(\frac{2\psi(h)}{c}+\beta\right)^{\alpha-m-1+k}}\frac{\Gamma(\alpha-m-1+k)}{\Gamma(\alpha)}\Bigg)^{\frac{1}{2}}.
		\end{alignat*}
		The above implies that, in general, $\theta_Y(h)=\mathcal{O}(h^{\frac{(m+1)-\alpha}{2}})$. 
		\item If $\int_{\R} |x| \,\nu(dx)<\infty$, $\Sigma=0$ and $\gamma_0$ as defined in (\ref{equation:gammazero}), then $Y_t(x)$ is $\theta$-lex-weakly dependent. Let $c \in (0,1]$, then for 
		\begin{alignat*}{2}
			m&\in\N,\qquad&& \theta_Y(h)\leq2V_m(c) m! \beta^{\alpha} \gamma_{abs}  \sum_{k=0}^m \frac{\psi(h)^k}{k!(\psi(h)+\beta)^{\alpha-m-1+k}}\frac{\Gamma(\alpha-m-1+k)}{\Gamma(\alpha)},
		\end{alignat*}
		whereas for $c>1$ and
		\begin{alignat*}{2}
			m&\in\N,\qquad &&\theta_Y(h)\leq2V_m(c)d! \beta^{\alpha}\gamma_{abs} \sum_{k=0}^m \frac{\left(\frac{\psi(h)}{c}\right)^k}{k!\left(\frac{\psi(h)}{c}+\beta\right)^{\alpha-m-1+k}}\frac{\Gamma(\alpha-m-1+k)}{\Gamma(\alpha)},
		\end{alignat*}
		where $\gamma_{abs}=|\gamma_0|+\int_\R|x|\nu(dx)$, and $V_m(c)$ denotes the volume of the $m$-dimensional ball with radius $c$. the above implies that, in general, $\theta_Y(h)=\mathcal{O}(h^{(m+1)-\alpha})$.
	\end{enumerate}
\end{Theorem}

This implies the following sufficient conditions for the asymptotic normality of the sample mean and the sample autocovariance function.

\begin{Corollary}\label{corollary:cltMSTOU}
	Let $(Y_t(x))_{(t,x)\in\R\times\R^m}$ be a c-class MSTOU process and $(\gamma,\Sigma,\nu,\pi)$ the characteristic quadruplet of its driving L\'evy basis. Moreover, let the mean reversion parameter $\lambda$ be $Gamma(\alpha,\beta)$ distributed with $\alpha>m+1$ and $\beta>0$. 
	\begin{enumerate}
		\item[(i)] If  $\gamma+\int_{ |x|>1}x\nu(dx)=0$, $\int_{|x|>1} |x|^{2+\delta}\nu(dx)<\infty$ for some $\delta>0$ and $\alpha>(m+1) \left(3+\frac{2}{\delta} \right)$, then the sample mean of $Y_t(x)$ as defined in (\ref{equation:samplemean}) is asymptotically normal.
		\item[(ii)] If $\gamma+\int_{ |x|>1}x\nu(dx)=0$, $\int_{|x|>1} |x|^{4+\delta}\nu(dx)<\infty$ for some $\delta>0$ and $\alpha>(m+1)\left(\frac{3+\delta}{2+\delta}\right)\left(3+\frac{2}{\delta}\right)$, then the sample autocovariances as defined in (\ref{equation:sampleautocovariance}) are asymptotically normal.
	\end{enumerate}
\end{Corollary}

\begin{Corollary}\label{corollary:cltMSTOU2}
	Let $(Y_t(x))_{(t,x)\in\R\times\R^m}$ be a c-class MSTOU process and $(\gamma,\Sigma,\nu,\pi)$ the characteristic quadruplet of its driving L\'evy basis. Moreover, let the mean reversion parameter $\lambda$ be $Gamma(\alpha,\beta)$ distributed such that $\alpha>m+1$ and $\beta>0$. 
	\begin{enumerate}
		\item [(i)] If $\int_{\R} |x|\nu(dx)<\infty$, $\Sigma=0$, and $\alpha>(m+1) \left(2+\frac{2}{\delta} \right)$, then the sample mean of $Y_t(x)$ as defined in (\ref{equation:samplemean}) is asymptotically normal.
		\item [(ii)]If $\int_{\R} |x|\nu(dx)<\infty$, $\Sigma=0$, $\int_{|x|>1} |x|^{4+\delta}\nu(dx)<\infty$ for some $\delta>0$ and $\alpha>(m+1)\left(\frac{3+\delta}{2+\delta}\right)\left(2+\frac{2}{\delta}\right)$, then the sample autocovariances as defined in (\ref{equation:sampleautocovariance}) are asymptotically normal.
	\end{enumerate}
\end{Corollary}

\begin{Remark}
	Since the c-class MSTOU processes satisfy the assumptions of Theorem \ref{theorem:thetasamplemeanspecial}, we can derive asymptotic normality of its sample mean under the weaker assumptions $E[Y_t(x)^{2}]<\infty$ and $\alpha>3(m+1)$.
\end{Remark}	

We conclude with some remarks regarding the short and long range dependence of  an MSTOU process. 
\begin{Definition}
	A stationary random field $Y=(Y_t(x))_{(t,x)\in\R\times\R^m}$ is said to have temporal short-range dependence if 
	\begin{gather*}
		\int_0^{\infty}Cov(Y_t(x),Y_{t+\tau}(x))d\tau<\infty,
	\end{gather*}
	and temporal long-range dependence if the integral is infinite.\\ 
	If $Cov(Y_t(x),Y_t(x+m_x))=C(|m_x|)$ for all $m_x\in\R^m$ and a positive definite function $C$ the random field $Y$ is called isotropic. Now, an isotropic random field is said to have spatial short-range dependence if 
	\begin{gather*}
		\int_0^{\infty}C(r)dr<\infty.
	\end{gather*}
\end{Definition}	

We have that an MSTOU process is a stationary and isotropic random field, see Theorem 5 \cite{NV2017}.
By assuming a $Gamma(\alpha,\beta)$-distributed mean reversion parameter $\lambda$, we have the following results, as shown in Section 6 \cite{CS2018} and Section 3.3 \cite{NV2017}:
\begin{enumerate}[(i)]
	\item For $m=0$, we have that $Y$ is a supOU process which is well-defined for $\alpha >1$ and $\beta >0$. Thus, we obtain a long-memory process for $1<\alpha \leq 2$ and a short memory one for $\alpha > 2$.
	\item For $m=1$, $Y$ is well-defined if $\alpha>2$ and $\beta>0$. $Y$ exhibits temporal as well as spatial long-range dependence for $2<\alpha\leq3$. If $\alpha>3$ we observe temporal and spatial short-range dependence.
	\item For $m=3$, $Y$ is well-defined if $\alpha >4$ and $\beta >0$. $Y$ exhibits temporal as well as spatial long-range dependence for $4<\alpha\leq 5$. If $\alpha>5$ we observe temporal and spatial short-range dependence.
\end{enumerate}

It is then easy to see that the assumptions in the Corollaries \ref{corollary:cltMSTOU} and \ref{corollary:cltMSTOU2} imply that we are in the realm of short-range dependence.

\begin{Remark}(GMM estimator)
	For $m=0$, a consistent GMM estimator for the supOU process is defined in \cite{STW2015}. In \cite{CS2018}, the authors show asymptotic normality of the estimator and that if the underlying L\'evy process is of finite variation and all moments exist, then the GMM estimator is asymptotic normally distributed for $\alpha >2$.
	
	For $m\geq 1$, a consistent GMM estimator for a c-class MSTOU process is introduced in \cite{NV2017}. The results in Corollaries \ref{corollary:cltMSTOU} and \ref{corollary:cltMSTOU2} should pave the way for an analysis of the asymptotic normality of the GMM estimator defined in \cite{NV2017} using arguments similar to \cite{CS2018}.
	For example, when $m=1$, in the finite variation case and when all moments exist, we can apply our results to short-range dependent MSTOU processes with $\alpha>4$.
\end{Remark}

\subsection{Example of non-influenced MMAF: L\'evy-driven CARMA fields}
\label{sec3-8}
We conclude the section by showing that our developed asymptotic theory can be applied to the class of L\'evy-driven CARMA fields defined on $\R^m$. 

CARMA (continuous autoregressive moving average) fields are an extension of the well-known CARMA processes (see, e.g., \cite{B2014} for a comprehensive introduction) and have been introduced in \cite{B2019b, BM2017,KP2019, P2018}.	

In \cite{BM2017}, the authors define CARMA fields as isotropic random fields 
\begin{gather}\label{equation:isotropicCARMA}
	Y(t)=\int_{\R^m}g(t-s)dL(s),\quad t\in\R^m,
\end{gather}
where $g$ is a radially symmetric kernel and $L$ a real-valued L\'evy basis on $\R^m$.
When the L\'evy basis $L$ has a finite second-order structure, the CARMA fields generate a rich family of isotropic covariance functions on $\R^m$, which are not necessarily non-negative or monotone.

On the other hand, in \cite{P2018}, the author defines CARMA(p,q) fields based on a system of stochastic partial differential equations. For $0 \leq q <p$, the mild solution of the system is called a causal CARMA field and is given by 
\begin{gather}\label{equation:causalCARMA}
	Y(t)= b^T \int_{-\infty}^{t_1}\cdots \int_{-\infty}^{t_m} e^{A_1(t_1-s_1)}\cdots e^{A_m(t_m-s_m)}c~dL(s), \quad (t_1,\ldots,t_m)\in\R^m,
\end{gather}
where $A_1,\ldots,A_m$ are companion matrices, $L$ is a real-valued L\'evy basis on $\R^m$, $b=(b_0,\ldots,b_{p-1})^T \in \R^p$ with $b_q \neq 0$ and $b_i=0$ for $i >q$ and $c=(0,\ldots,0,1)^T\in\R^p$, see  \cite[Definition 3.3]{P2018}.\\
In \cite{B2019b}, the author shows the existence of a mild solution for the CARMA stochastic partial differential equation, c.f. \cite[equation (1.7)]{B2019b}, in \cite[Theorem 5.3]{B2019b}. The causal CARMA fields presented in \cite{P2018} can be seen as a special case of the CARMA random fields defined in \cite{B2019b}. A more subtle relationship exists between the definition of CARMA field in \cite{B2019b} and \cite{BM2017} just when $m$ is odd, see \cite[Section 7]{B2019b}. 

In general, our framework can be applied to the class of CARMA fields introduced in \cite{B2019b} and \cite{BM2017} when the assumption of the theorem below holds.	

\begin{Theorem}\label{theorem:levydrivenexpdecaykernel}
	Let $L$ be an $\R^d$-valued L\'evy basis with characteristic quadruplet $(\gamma,\Sigma,\nu,\pi)$ such that $\int_{\norm{x}>1}\norm{x}^2\nu(dx)<\infty$ and $\gamma+\int_{\norm{x}>1}x\nu(dx)=0$. Let $g:\R^m\ra M_{n\times d}(\R)$ such that $g$ is exponentially bounded in norm, i.e. there exists $M,K\in\R^+$ such that
	\begin{gather}\label{equation:expdecaykernel}
		\norm{g(t)}^2\leq M e^{-K\norm{t}}, \text{ for all }t\in\R^m.
	\end{gather}
	Then, the moving average field $X_t=\int_{\R^m}g(t-s)L(ds)$, $t\in\R^m$ is an $\eta$-weakly dependent field with exponentially decaying $\eta$-coefficients.
\end{Theorem}
Due to the equivalence of norms, the result does not depend on a specific choice of norms.
\begin{proof}
	See Section \ref{sec5}.
\end{proof}

\begin{Remark}
	Since the kernels in (\ref{equation:isotropicCARMA}) and (\ref{equation:causalCARMA}) satisfy equation (\ref{equation:expdecaykernel}), for example, we can show that these fields are $\eta$-weakly dependent by applying Theorem \ref{theorem:levydrivenexpdecaykernel}. 
\end{Remark}

\section{Ambit fields}
\label{sec4}
In the following, we will briefly introduce stationary ambit fields. We discuss weak dependence properties of such fields and give sufficient conditions for the applicability of the results in Section \ref{sec2-3}.

\subsection{The ambit framework}
\label{sec4-1}	
Let $A_t(x)\subset\R\times\R^m$ for $(t,x)\in\R\times\R^m$ be an ambit set as defined in (\ref{equation:ambitset}). By $\mathcal{P}'$ we denote the usual predictable $\sigma$-algebra on $\R$, i.e. the $\sigma$-algebra generated by all left-continuous adapted processes. Then, a random field $X:\Omega\times\R\times\R^m\rightarrow\R$ is called predictable if it is measurable with respect to the $\sigma$-algebra $\mathcal{P}$ defined by $\mathcal{P}=\mathcal{P}' \otimes\mathcal{B}(\R^m)$. 

\begin{Definition}\label{definition:ambitfield}
	Let $\Lambda$ be a real-valued L\'evy basis on $\R\times\R^m$ with characteristic quadruplet $(\gamma,\Sigma,\nu,\pi)$, $\sigma$ a predictable stationary random field on $\R\times\R^m$ independent of $\Lambda$.  
	Furthermore, let $l:\R^m\times\R\rightarrow\R$ be a measurable function and $A_t(x)$ an ambit set. We assume that $f(\xi,s)=\mathbb{1}_{A_t(x)}(\xi,s)l(\xi,s)\sigma_s(\xi)$ satisfies (\ref{equation:intcond1}), (\ref{equation:intcond2}) and (\ref{equation:intcond3})  almost surely.
	Then, the random field
	\begin{gather}\label{eq:ambitfield}
		\begin{align}
			\begin{aligned}
				Y_t(x)=\int_{A_t(x)} l(x-\xi,t-s) \sigma_s(\xi) \Lambda(d\xi,ds),~ (t,x)\in\R\times\R^m,
			\end{aligned}
		\end{align}
	\end{gather}
	is called an ambit field and it is stationary (see p.~185 \cite{BBV2018}).
\end{Definition}

\begin{Remark}\label{remark:stochintegrand}
	Ambit fields require us to define integrals with respect to L\'evy bases where the integrand is stochastic. Although the integration theory from Rajput and Rosinski just enables us to define stochastic integrals with respect to deterministic integrands \cite{RR1989}, one can extend this theory to stochastic integrands which are predictable and independent of the L\'evy basis. We can condition on the $\sigma$-algebra generated by the field $\sigma$ and use again the integration theory introduced in \cite{RR1989}. Then, such integrals are well defined if the kernel function satisfies the sufficient conditions (\ref{equation:intcond1}), (\ref{equation:intcond2}) and (\ref{equation:intcond3}) almost surely.
	Allowing for dependence between the volatility field and the L\'evy basis demands the use of a different integration theory as presented in \cite[Section 1.2.1]{BBPV2014}, \cite[Proposition 39]{BBV2018}, \cite[Theorem 3.2]{BGP2013} and \cite{CK2015}.
\end{Remark}	

We conclude this section by giving explicit formulas for the first and second moment of an ambit field.

\begin{Proposition}\label{proposition:ambitmoments}
	Let $Y$ be an ambit field as defined in (\ref{eq:ambitfield}) driven by a real-valued L\'evy basis with characteristic quadruplet $(\gamma,\Sigma,\nu,\pi)$ and $\Lambda$-integrable kernel function $f(\xi,s)=\mathbb{1}_{A_t(x)}(\xi,s)l(\xi,s)\sigma_s(\xi)$, where $\sigma$ is predictable, stationary and independent of $\Lambda$. 
	\begin{enumerate}[(i)]
		\item If $E[|Y_t(x)|]<\infty$, the first moment of $Y$ is given by
		\begin{align*}
			E[Y_t(x)]= \mu_\Lambda E[\sigma_t(x)] \int_{A_t(x)}l(x-\xi,t-s) d\xi\, ds,
		\end{align*}
		where $\mu_\Lambda =\gamma+\int_{|x|\geq1}x\nu(dx)$.
		\item If $E[Y_t(x)^2]<\infty$, it holds
		\begin{align*}
			Var(Y_t(x))= &\Sigma_\Lambda E[\sigma_t(x)^2]\int_{A_t(x)}  l(x-\xi,t-s)^2 d\xi\, ds \\
			&+\mu_\Lambda ^2 \int_{A_t(x)}\int_{A_t(x)} l(x-\xi,t-s)l(x-\tilde{\xi},t-\tilde{s}) \rho(s,\tilde{s},\xi,\tilde{\xi}) d\xi\, ds\, d\tilde{\xi}\, d\tilde{s}\,  \text{ and}\\
			Cov(Y_t(x),Y_{\tilde{t}}(\tilde{x}))=&\Sigma_\Lambda E[\sigma_t(x)^2] \int_{A_t(x)\cap A_{\tilde{t}}(\tilde{x})}  l(x-\xi,t-s)l(\tilde{x}-\xi,\tilde{t}-s) d\xi\, ds\, \\
			&+\mu_\Lambda ^2 \int_{A_t(x)}\int_{A_{\tilde{t}}(\tilde{x})} l(x-\xi,t-s)l(\tilde{x}-\tilde{\xi},\tilde{t}-\tilde{s}) \rho(s,\tilde{s},\xi,\tilde{\xi}) d\xi\, ds\, d\tilde{\xi}\, d\tilde{s},
		\end{align*}
		where $\Sigma_\Lambda =\Sigma+\int_{\R}x^2\nu(dx)$ and $\rho(s,\tilde{s},\xi,\tilde{\xi})= E[\sigma_s(\xi)\sigma_{\tilde{s}}(\tilde{\xi})]-E[\sigma_s(\xi)]E[\sigma_{\tilde{s}}(\tilde{\xi})]$. 
	\end{enumerate}
\end{Proposition}
\begin{proof}
	Immediate from \cite[Proposition 41]{BBV2018}.
\end{proof}

\subsection{Weak dependence properties of ambit fields}
\label{sec4-2}
Let us consider a stationary ambit field $Y=(Y_t(x))_{(t,x)\in\R\times\R^m}$ as defined in (\ref{eq:ambitfield}). 
To analyze the covariance structure of $Y$, it becomes necessary to specify a model for $\sigma$. 
In \cite{BBV2012} the authors proposed to model $\sigma$ by kernel-smoothing of a homogeneous L\'evy basis, i.e. a moving average random field
\begin{gather}
	\label{sigma}
	\sigma_t(x)=\int_{A_t^\sigma(x)}j(x-\xi,t-s) \Lambda^\sigma(d\xi,ds),
\end{gather} 
where $\Lambda^\sigma$ is a real valued L\'evy basis independent of $\Lambda$ with characteristic quadruplet $(\mu_{\sigma},\Sigma_{\sigma},\nu_{\sigma},\pi_{\sigma})$, $A^\sigma=(A^\sigma_t(x))_{(t,x)\in\R\times\R^m}$ an ambit set as defined in (\ref{equation:ambitset}) and $j$ a real valued $\Lambda^\sigma$-integrable function. In the following, we extend this model and assume $\sigma$ to be an $(A^\sigma,\Lambda^\sigma)$-influenced MMAF, i.e.
\begin{gather}\label{equation:volatilityfield}
	\sigma_t(x)=\int_S\int_{A_t^\sigma(x)}j(A,x-\xi,t-s) \Lambda^\sigma(dA,d\xi,ds).
\end{gather}

\begin{Proposition}\label{proposition:ambitthetaweaklydep}
	Let $Y=(Y_t(x))_{(t,x)\in\R\times\R^m}$ be an ambit field as defined in (\ref{eq:ambitfield}) with $\sigma=(\sigma_t(x))_{(t,x)\in\R\times\R^m}$ being a predictable $(A^\sigma,\Lambda^\sigma)$-influenced MMAF as defined in (\ref{equation:volatilityfield}) and such that $A_0(0)$ and $A_0^\sigma(0)$ satisfy (\ref{condition:scalarproduct}), $j\in L^1(S\times\R\times\R^m, \pi\otimes\lambda) \cap L^2(S\times\R^m\times\R, \pi\otimes\lambda)$, where $\lambda$ indicates the Lebesgue measure on $\R^{m+1}$, and $\int_{|x|>1}|x|^2\nu_\sigma(dx)<\infty$.
	\begin{enumerate}[(i)]
		\item If $l\in L^2(\R\times\R^m)$, $\int_{|x|>1}|x|^2\nu(dx)<\infty$ and $\gamma + \int_{|x|>1}x\nu(dx)=0$, then $Y$ is $\theta$-lex-weakly dependent with $\theta$-lex-coefficients $\theta_Y(h)$ satisfying
		\begin{gather}\label{equation:thetalexcoefficientsambitfield}
			\begin{aligned}
				\theta_Y(h)\leq&2 \Big(\Sigma_\Lambda  E[\sigma_0(0)^2]
				\int_{A_0(0) \cap V_{(0,0)}^{\psi(h)}}l(-\xi,-s)^2 d\xi ds\Big)^{\frac{1}{2}}\\
				&+ 2\Bigg( \Bigg(\Sigma_{\Lambda^\sigma} \int_S\int_ {A_0^\sigma(0)\cap V_{(0,0)}^{\psi(h)}} j(A,-\xi,-s)^2 d\xi ds \pi(dA)\\
				&+ \mu_{\Lambda^\sigma}^2 \Bigg(\!\int_S\int_ {A_0^\sigma(0)\cap V_{(0,0)}^{\psi(h)}}  j(A,-\xi,-s) d\xi ds \pi(dA)\Bigg)^2 \ \Bigg)\\
				&\qquad\times \Sigma_\Lambda\!\!   \int_{A_0(0) \backslash V_{(0,0)}^{\psi(h)}}l(-\xi ,-s)^2 d\xi ds\Bigg)^{\frac{1}{2}}.
			\end{aligned}
		\end{gather}
		\item If $l\in L^1(\R\times\R^m)\cap L^2(\R\times\R^m)$ and $\int_{|x|>1}|x|^2\nu(dx)<\infty$, then $Y$ is $\theta$-lex-weakly dependent with $\theta$-lex-coefficients $\theta_Y(h)$ satisfying
		\begin{align}
			\theta_Y(h)\leq&
			2 \Bigg(\Big(\Sigma_\Lambda  E[\sigma_0(0)^2]
			\int_{A_0(0) \cap V_{(0,0)}^{\psi(h)}}l(-\xi,-s)^2 d\xi ds  \nonumber\\ 
			&+\mu_\Lambda E[\sigma_0(0)^2] \Big(\int_{A_0(0) \cap V_{(0,0)}^{\psi(h)}}l(-\xi,-s) d\xi ds\Big)^2  \Bigg)^{\frac{1}{2}} \label{equation:thetalexcoefficientszeroambitfield}
		\end{align}
		\begin{align*}
			\hspace{2cm} &+ 2\Bigg( \Bigg(\Sigma_{\Lambda^\sigma} \int_S\int_ {A_0^\sigma(0)\cap V_{(0,0)}^{\psi(h)}} j(A,-\xi,-s)^2 d\xi ds \pi(dA) \\&+\mu_{\Lambda^\sigma}^2\Bigg( \int_S\int_ {A_0^\sigma(0)\cap V_{(0,0)}^{\psi(h)}} j(A,-\xi,-s) d\xi ds \pi(dA)\Bigg)^2 \Bigg)\\ &\times\Bigg( \Sigma_\Lambda   \int_{A_0(0) \backslash V_{(0,0)}^{\psi(h)}}l(-\xi ,-s)^2 d\xi ds\\
			&\qquad+ \mu_\Lambda  \Big(\int_{A_0(0) \backslash V_{(0,0)}^{\psi(h)}}l(-\xi ,-s) d\xi ds\Big)^2 \Bigg)\Bigg)^{\frac{1}{2}}.
		\end{align*}
		\item If $l\in L^1(\R\times\R^m)$, $\int_{\R}|x|\nu(dx)<\infty$ and $\Sigma=0$, then $Y$ is $\theta$-lex-weakly dependent with $\theta$-lex-coefficients $\theta_Y(h)$ satisfying
		\begin{gather}\label{equation:thetalexcoefficientsfinitevariationambitfield}
			\begin{aligned}
				\theta_Y(h)\leq&2\Sigma_\sigma \Big(|\gamma_0|+\int_{\R}|x|\nu(dx)\Big) \Big(\int_{A_0(0) \cap V_{(0,0)}^{\psi(h)}}|l(-\xi,-s)| d\xi ds\Big)\\
				&+2 \Bigg(\Sigma_{\Lambda^\sigma} \int_S\int_ {A_0^\sigma(0)\cap V_{(0,0)}^{\psi(h)}} j(A,-\xi,-s)^2 d\xi ds \pi(dA)\\
				&+\mu_{\Lambda^\sigma}^2\Bigg( \int_S\int_ {A_0^\sigma(0)\cap V_{(0,0)}^{\psi(h)}} j(A,-\xi,-s) d\xi ds \pi(dA)\Bigg)^2 \Bigg)\\
				&\times \Big(|\gamma_0|+\int_{\R}|x|\nu(dx)\Big)
				\Big(\int_{A_0(0) \backslash V_{(0,0)}^{\psi(h)}}|l(-\xi,-s)| d\xi ds\Big).
			\end{aligned}
		\end{gather}
	\end{enumerate}
	The above results hold for all $h>0$, where $\psi(h)=\frac{-bh}{2\sqrt{m+1}}$ and $b$ are defined as in $(\ref{equation:xi})$, $\mu_\Lambda =\gamma+\int_{|x|\geq1}x\nu(dx)$, $\Sigma_\Lambda =\Sigma+\int_{\R}x^2\nu(dx)$, $\mu_{\Lambda^\sigma} =\gamma_\sigma+\int_{|x|\geq1}x\nu_\sigma(dx)$ and $\Sigma_{\Lambda^\sigma}=\Sigma_\sigma+\int_{\R}x^2\nu_\sigma(dx)$.
\end{Proposition}
\begin{proof}
	See Section \ref{sec5-5}.
\end{proof}

We now analyze the case in which $\sigma$ is a $p$-dependent random field for $p\in\N$.

\begin{Proposition}\label{proposition:ambitmdepthetaweaklydep}
	Let $Y=(Y_t(x))_{(t,x)\in\R\times\R^m}$ be an ambit field as defined in (\ref{eq:ambitfield}) with a predictable $p$-dependent stationary random field $\sigma_t(x)$ for $p\in\N$. Assume that $A_0(0)$ satisfies (\ref{condition:scalarproduct}). Additionally assume that $l\in L^2(\R^m\times\R)$, $\int_{|x|>1}|x|^2\nu(dx)<\infty$ and $\gamma + \int_{|x|>1}x\nu(dx)=0$. Then, for sufficiently big $h$, $Y$ is $\theta$-lex-weakly dependent with coefficients
	\begin{gather}\label{equation:thetalexcoefficientsambitfieldiid}
		\begin{gathered}
			\theta_Y(h)\leq2 \Big(\Sigma_\Lambda  E[\sigma_0(0)^2]\int_{A_0(0) \cap V_{(0,0)}^{\psi(h)}}l(-\xi,-s)^2 d\xi ds\Big)^{\frac{1}{2}},
		\end{gathered}
	\end{gather}
	with $\psi(h)=\frac{-bh}{2\sqrt{m+1}}$ and $b$ as defined in $(\ref{equation:xi})$, $\Sigma_\Lambda =\Sigma+\int_{\R}x^2\nu(dx)$.
\end{Proposition}
\begin{proof}
	See Section \ref{sec5-5}.
\end{proof}

\subsubsection{Volatility fields}

Let $\sigma$ be a $(A^\sigma,\Lambda^\sigma)$-influenced MMAF as defined in (\ref{sigma}), $j$ a non-negative kernel function and the following assumptions hold

\begin{equation*}
	\textrm{(H)}: \left\{\begin{array}{l}
		\textrm{The L\'evy basis $\Lambda^{\sigma}$ has generating quadruple $(\gamma_{\sigma},0, \nu_{\sigma},\pi_{\sigma})$ such that}\\
		\int_{\R} |x| \nu_{\sigma}(dx) < \infty,\,\,\ \gamma_{\sigma}-\int_{|x|\leq 1} x \nu_{\sigma}(dx) \geq 0\,\,\,\textrm{and}\,\,\,\nu_{\sigma}(\R^{-})=0.
	\end{array}\right.
\end{equation*}
Then, $\sigma$ has values in $\R^+$, and we call it volatility or intermittency field. Note that Assumption (H) implies that $\Lambda^\sigma$ satisfies the finite variation case and that this model is used in several applications of the ambit fields, see \cite{BBV2018}.

By assuming additionally that $j\in L^1(S\times\R\times\R^m, \pi\otimes\lambda) \cap L^2(S\times\R^m\times\R, \pi\otimes\lambda)$ and $\int_{|x|>1}|x|^2\nu_\sigma(dx)<\infty$, the results in Proposition  \ref{proposition:ambitthetaweaklydep} (i) and (ii) hold.
On the other hand, the bound in Proposition \ref{proposition:ambitthetaweaklydep} (iii) can be tightened.	

\begin{Corollary}\label{corollary:finitefiniteambitthetaweaklydep}
	Let $Y=(Y_t(x))_{(t,x)\in\R\times\R^m}$ be an ambit field as defined in (\ref{eq:ambitfield}) with predictable volatility field $\sigma_t(x)$ being an $(A^\sigma,\Lambda^\sigma)$-influenced MMAF such that $A_0(0)$ and $A_0^\sigma(0)$ satisfy (\ref{condition:scalarproduct}), $j \in L^1(S\times\R\times\R^m, \pi\otimes\lambda)$, $l \in L^1(\R\times\R^m)$ and Assumption (H) holds. Let $\gamma_0$ with respect to $\Lambda$ and $\gamma_{0,\sigma}$ with respect to $\Lambda^\sigma$ be defined as in (\ref{equation:gammazero}). Then, $Y$ is $\theta$-lex-weakly dependent with coefficients 
	\begin{gather}\label{equation:thetalexcoefficientsfinitefinitevariationambitfield}
		\begin{aligned}
			\theta_Y(h)\leq&2 \Big(|\gamma_{0,\sigma}|+\int_{\R}|x|\nu_\sigma(dx)\Big)\Big(|\gamma_0|+\int_{\R}|x|\nu(dx)\Big) \\&\quad\times\Big(\int_S\int_{A_0(0)}|j(A,-\xi,-s)| d\xi ds\pi(dA)\Big)\Big( \int_{A_0(0) \cap V_{(0,0)}^{\psi(h)}}|l(-\xi,-s)| d\xi ds\Big)\\
			&+2\Big(|\gamma_{0,\sigma}|+\int_{\R}|x|\nu_\sigma(dx)\Big)\Big(|\gamma_0|+\int_{\R}|x|\nu(dx)\Big)\\
			&\quad\times\Big( \int_{A_0(0) \backslash V_{(0,0)}^{\psi(h)}}|l(-\xi,-s)| d\xi ds\Big) \Big(\int_S \int_{A_0(0) \cap V_{(0,0)}^{\psi(h)}}|j(A,-\xi,-s)| d\xi ds\pi(dA)\Big),
		\end{aligned}
	\end{gather}
	for all $h>0$, with $\psi(h)=\frac{-bh}{2\sqrt{m+1}}$ and $b$ as defined in $(\ref{equation:xi})$.
\end{Corollary}
\begin{proof}
	Analogous to Proposition \ref{proposition:ambitthetaweaklydep}.
\end{proof}

\subsection{Sample moments of ambit fields}
\label{sec4-3}
In this section, we study the asymptotic distribution of sample moments of $Y$. As in Section \ref{sec3-2} we assume that we observe $Y$ on a sequence of finite sampling sets $D_n\subset\Z\times\Z^m$, such that (\ref{condition:samplingset}) holds.

\begin{Theorem}\label{theorem:ambitthetasamplemean}
	Let $Y=(Y_t(x))_{(t,x)\in\Z\times\Z^m}$ be an ambit field as defined in (\ref{eq:ambitfield}) such that $E[Y_t(x)]=0$, $E[|Y_t(x)|^{2+\delta}]<\infty$ for some $\delta>0$. Additionally assume that $Y$ is $\theta$-lex-weakly dependent with $\theta$-lex-coefficients satisfying $\theta_Y(h)=\mathcal{O}(h^{-\alpha})$ for $\alpha>m(1+\frac{1}{\delta})$ and defined $\mathcal{I}$ as in Theorem \ref{theorem:clt}. Then,
	\begin{gather*}
		\sigma^2=\sum_{(u_t,u_x)\in\Z\times\Z^m}E[Y_{0}(0)Y_{u_t}(u_x)|\mathcal{I}]
	\end{gather*}
	is finite, non-negative and
	\begin{gather}\label{eq:ambitthetaclt}
		\frac{1}{|D_n|^{\frac{1}{2}}}\sum_{(u_t,u_x)\in D_n}Y_{u_t}(u_x)\underset{n\ra\infty}{\xrightarrow{\makebox[2em][c]{d}}}\varepsilon \sigma,
	\end{gather}
	where $\varepsilon$ is a standard normally distributed random variable independent of $\sigma^2$.
\end{Theorem}
\begin{proof}
	The result follows from Theorem \ref{theorem:clt}.
\end{proof}

\begin{Corollary}\label{corollary:ambitsamplemomentsofhigherorder}
	Let $Y=(Y_t(x))_{(t,x)\in\Z\times\Z^m}$ be an ambit field as defined in (\ref{eq:ambitfield}) such that $E[|Y_t(x)|^{2p+\delta}]<\infty$ for $p\geq1$ and some $\delta>0$. Additionally, let us assume that $Y$ is $\theta$-lex-weakly dependent with $\theta$-lex-coefficients satisfying $\theta_Y(h)=\mathcal{O}(h^{-\alpha})$, for $\alpha>m\left(1+\frac{1}{\delta}\right)(\frac{2p-1+\delta}{p+\delta})$ and define $\mathcal{I}$ as in Theorem \ref{theorem:clt}. Then,
	\begin{gather*}
		\Sigma=\sum_{(u_t,u_x)\in\Z\times\Z^m}Cov(Y_0(0)^{p},Y_{u_t}(u_x)^{p} |\mathcal{I}) 
	\end{gather*}
	is finite, non-negative and
	\begin{gather}\label{eq:ambitthetaclthighermoments}
		\frac{1}{|D_n|^{\frac{1}{2}}}\sum_{(u_t,u_x)\in D_n}Y_{u_t}(u_x)^p-E[Y_{0}(0)^p]\underset{n\ra\infty}{\xrightarrow{\makebox[2em][c]{d}}}\varepsilon \sigma,
	\end{gather}
	where $\varepsilon$ is a standard normally distributed random variable which is independent of $\sigma^2$.
\end{Corollary}
\begin{proof}
	Analogous to Corollary \ref{corollary:samplemomentsofhigherorder}.  
\end{proof}

\begin{Remark}
	Theorem \ref{theorem:ambitthetasamplemean} and Corollary \ref{corollary:ambitsamplemomentsofhigherorder} are important first steps to develop statistical inference for the class of ambit fields. However, we note that the limits in (\ref{eq:ambitthetaclt}) and (\ref{eq:ambitthetaclthighermoments}) are of mixed Gaussian type. Conditions that ensure the ergodicity of an ambit field with a deterministic kernel can be found in \cite[Theorem 3.6]{PV2017} whereas, for the case of a non-deterministic kernel, this remains an open problem.
\end{Remark}

\section{Proofs}
\label{sec5}
\subsection{Proofs of Section \ref{sec2-2} and \ref{sec2-3}}
\label{sec5-1}
We first extend some of the results obtained in \cite{DD2003} to random fields. This will enable us to connect the sufficient conditions given in \cite{D1998} with our definition of the $\theta$-lex-coefficients.

Let us define the space of bounded, Lipschitz continuous functions $\mathscr{L}_1=\{g:\R\rightarrow \R,$  bounded and Lipschitz continuous with $Lip(g)\leq1\}$. For a $\sigma$-algebra $\mathscr{M}$ and an $\R^n$-valued integrable random field $X=(X_t)_{t\in\Z^m}$ we define the mixingale-type measures of dependence
$$\gamma(\mathscr{M},X)=\norm{E[X|\mathscr{M}]-E[X]}_{1}, \quad \textrm{and} \quad \theta(\mathscr{M},X)=\sup_{g\in  \mathscr{L}_1} \norm{E[g(X)|\mathscr{M}]-E[g(X)]}_{1}.$$
Using the above measures of dependence we define the following dependence coefficients
\begin{equation}
	\label{eq:gammathetanonstat}
	\gamma_h= \sup_{j\in\Z^m} \gamma\big(\mathcal{F}_{V_j^h},X_j\big), \quad \textrm{and} \quad \theta_h= \sup_{j\in\Z^m} \theta\big(\mathcal{F}_{V_j^h},X_j\big),
\end{equation}
for $h\in\N$. Obviously, it holds $\gamma(\mathscr{M},X)\leq2\norm{X}_1$ and $\gamma(\mathscr{M},X)\leq\theta(\mathscr{M},X)$ such that $\gamma_h\leq\theta_h$ for all $h\in\N$.
If $X$ is stationary we can write $\gamma_h$ and $\theta_h$ from (\ref{eq:gammathetanonstat}) as 
\begin{equation}
	\label{eq:gammathetastat}
	\gamma_h= \gamma\big(\mathcal{F}_{V_0^h},X_0\big), \quad \text{and} \quad \theta_h= \theta\big(\mathcal{F}_{V_0^h},X_0\big),
\end{equation}
for $h\in\N$. First, we extend Proposition 2.3 from \cite{DDLLLP2008} and connect the $\theta$-lex-coefficients $\theta(h)$ from Definition \ref{thetaweaklydependent} with the mixingale-type coefficient $\theta_h$ defined above. 

\begin{Lemma}\label{lemma:thetaiscondexp}
	Let $X=(X_t)_{t\in\Z^m}$ be a real-valued random field. Then it holds that
	\begin{gather*}
		\theta(h)=\theta_h, \quad h\in\N.
	\end{gather*}
\end{Lemma} 
\begin{proof}
	Fix $u,h\in\N$. We first show $\theta_u(h)\leq\theta_{h}$. Let $F\in \mathcal{G}^*_k$, $G\in \mathcal{G}_1$, $j\in\R^m$, $k\leq u$ and $\Gamma=\{i_1,\ldots,i_k\}$ with $i_1,\ldots,i_k\in V_j^h$. Now
	\begin{align*}
		&\left| Cov \left( \frac{F(X_\Gamma)}{\norm{F}_\infty}, \frac{G(X_j)}{ Lip(G)} \right) \right|
		= \left| E\left[  \frac{F(X_\Gamma)}{\norm{F}_\infty}\frac{G(X_j)}{Lip(G)} - E\left[  \frac{F(X_\Gamma)}{\norm{F}_\infty}\right] E\left[\frac{G(X_j)}{Lip(G)}\right] \right]\right|\\
		&= \left| E\left[E\left[\frac{F(X_\Gamma)}{\norm{F}_\infty}\frac{G(X_j)}{Lip(G)}\bigg| \mathcal{F}_{V_j^h} \right] - \frac{F(X_\Gamma)}{\norm{F}_\infty} E\left[\frac{G(X_j)}{Lip(G)}\right] \right]\right|\\
		&\leq E\left[ \left|\frac{F(X_\Gamma)}{\norm{F}_\infty}\right|  \left|E\left[ \frac{G(X_j)}{Lip(G)}\bigg| \mathcal{F}_{V_j^h} \right] -  E\left[\frac{G(X_j)}{Lip(G)}\right]\right| \right]\\
		&\leq \left\lVert E\left[ \frac{G(X_j)}{Lip(G)}\bigg| \mathcal{F}_{V_j^h} \right] -  E\left[\frac{G(X_j)}{Lip(G)}\right] \right\rVert_1
		= \theta(\mathcal{F}_{V_j^h},X_j)\leq \theta_{h}.
	\end{align*}
	Taking the supremum on the left hand side we obtain $\theta_u(h)\leq \theta_{h}$ and finally $\theta(h)\leq \theta_{h}$.\\
	To prove the converse inequality, we first remark that by the martingale convergence theorem
	\begin{gather}
		\theta(\mathcal{F}_{V_j^h},X_j)=\lim_{k\rightarrow\infty} \theta( \mathcal{F}_{V_j^h\backslash V_j^k},X_j).\label{eq:convergencetailalgebra}
	\end{gather}
	Now, let $G\in\mathscr{L}_1$, i.e. $G\in\mathcal{G}_1$ with $Lip(G)\leq1$ and $j\in\R^m$. We first define $X_j^h(k)=\{ X_i:i\in V_j^h\backslash V_j^k \}$ and $F(X_j^h(k))=sign(E[G(X_j) | \mathcal{F}_{V_j^h\backslash V_j^k}]-E[g(X_j)])$ for $k>h$. Then $F\in\mathcal{G}^*_{u}$ for $u=|V_j^h\backslash V_j^k|\in\N$ with $\norm{F}_\infty=1$ and it holds
	\begin{align*}
		&E\left[ \left| E[G(X_j) | \mathcal{F}_{V_j^h\backslash V_j^k}]-E[G(X_j)]\right| \right]
		=E\left[ \left(E[G(X_j) | \mathcal{F}_{V_j^h\backslash V_j^k}]-E[G(X_j)]\right)F\Big(X_j^h(k)\Big)\right]\\
		&=E\left[ E\left[F\Big(X_j^h(k)\Big)G(X_j) | \mathcal{F}_{V_j^h\backslash V_j^k}\right]-E\left[F\Big(X_j^h(k)\Big)\right] E\left[G(X_j)\right] \right]\\
		&=Cov\left(F\Big(X_j^h(k)\Big), G(X_j)\right)\leq \theta(h).
	\end{align*}
	Using (\ref{eq:convergencetailalgebra}) we can deduce the stated equality.
\end{proof}

We define $Q_X$ as the generalized inverse of the tail function $x\mapsto P(|X|>x)$ and $G_X$ as the inverse of $x\mapsto \int_0^xQ_{X}(u)du$.

\begin{Lemma}\label{lemma:finitesumclt}
	Let $X=(X_t)_{t\in\Z^m}$ be a stationary centered real-valued random field such that $\norm{X_0}_2<\infty$ and assume that 
	\begin{gather}
		\int_{0}^{\norm{X_0}_1}\tilde\theta(u)  Q_{X_0}\circ G_{X_0}(u)du<\infty,\label{eq:summability}
	\end{gather}
	with $Q_X$ and $G_X$ as defined above and $\tilde\theta(u)= \sum_{k\in V_0} \mathbb{1}_{\left\{u<\theta_{|k|}\right\}}$. Then, 
	\begin{gather}
		\sum_{k\in V_0}|E[X_kE_{|k|}[X_0]]|<\infty,\label{eq:finitesumclt}
	\end{gather}
	where $E_{|k|}[X_0]=E[X_0|F_{V_0^{|k|}}]$.
\end{Lemma}
\begin{proof}
	First, let us observe that $X_k$ is $\mathcal{F}_{V_0^{|k|}}$ measurable, since $k\in V_0^{|k|}$. Then, we define $ \varepsilon_{k}=sign(E_{|k|}[X_0])$ such that
	\begin{align*}
		&\sum_{k\in V_0}|E[X_kE_{|k|}[X_0]]|
		\leq \sum_{k\in V_0}E[|X_k||E_{|k|}[X_0]|]
		= \sum_{k\in V_0}E[|X_k| \varepsilon_{k} E_{|k|}[X_0]]\\
		&= \sum_{k\in V_0}E[ E_{|k|}[|X_k| \varepsilon_{k}X_0]]
		=\sum_{k\in V_0} Cov(|X_k| \varepsilon_{k},X_0).
	\end{align*}
	We use Equation (4.2) of \cite[Proposition 1]{DD2003} to get 
	\begin{align*}
		&\leq 2\sum_{k\in V_0}  \int_{0}^{\gamma\Big(\mathcal{F}_{V_0^{|k|}},X_0\Big)/2} Q_{ \varepsilon_{k}|X_k|}\circ G_{X_0}(u)du\\
		&=2  \int_{0}^{\norm{X_0}_1} \sum_{k\in V_0}\mathbb{1}_{\left\{u<\gamma\Big(\mathcal{F}_{V_0^{|k|}},X_0\Big)/2\right\}} Q_{X_k}\circ G_{X_0}(u)du\\
		&\leq 2  \int_{0}^{\norm{X_0}_1} \sum_{k\in V_0}\mathbb{1}_{\left\{u< \theta\Big(\mathcal{F}_{V_0^{|k|}},X_0\Big)/2\right\}} Q_{X_k}\circ G_{X_0}(u)du\\
		&\leq 2  \int_{0}^{\norm{X_0}_1} \sum_{k\in V_0}\mathbb{1}_{\left\{u< \theta_{|k|}/2 \right\}} Q_{X_k}\circ G_{X_0}(u)du.
	\end{align*}
	Now, let $\tilde\theta(u)= \sum_{k\in V_0} \mathbb{1}_{\left\{u<\theta_{|k|}\right\}}$ and note that $Q_{X_k}=Q_{X_0}$, such that
	\begin{gather*}
		\leq 2 \int_{0}^{\norm{X_0}_1}\tilde\theta(u)  Q_{X_0}\circ G_{X_0}(u)du.
	\end{gather*}
	This shows that (\ref{eq:finitesumclt}) holds if (\ref{eq:summability}) is satisfied.
\end{proof}

We now derive sufficient criteria such that (\ref{eq:summability}) holds similar to \cite[Lemma 2]{DD2003}.

\begin{Lemma}\label{lemma:summability}
	Let $X=(X_t)_{t\in\Z^m}$ be a stationary real-valued random field and $\theta_h$ defined as above. Then (\ref{eq:summability}) holds if $\norm{X}_r<\infty$ for some $r>p>1$ and $\sum_{h=0}^\infty  (h+1)^{m(p-1)\frac{(r-1)}{(r-p)}-1} \theta_h<\infty$. In particular, for $p=2$ and $r=2+\delta$ with $\delta>0$ the above condition holds if $\theta_{h}\in \mathcal{O}(h^{-\alpha})$ for $\alpha>m(1+\frac{1}{\delta})$.
\end{Lemma}
\begin{proof}
	As stated in \cite[Proof of Lemma 2]{DD2003} we note that $\int_{0}^{\norm{X}_1} Q_{X}^{r-1}\circ G_{X}(u)du=\int_0^1Q_X^r(u)d  u=E[|X|^r]$. Applying H\"older's inequality with $q=\frac{r-1}{r-p}$ and $q'=\frac{r-1}{p-1}$ gives
	\begin{align*}
		&\left(\int_{0}^{\norm{X}_1}\tilde\theta(u)^{p-1}  Q_{X}^{p-1}\circ G_{X}(u)du\right)^{(r-1)}\\
		&\leq\left( \int_{0}^{\norm{X}_1}\tilde\theta(u)^{(p-1)\frac{(r-1)}{(r-p)}} du\right)^{(r-p)} \left(\int_{0}^{\norm{X}_1}Q_{X}^{r-1}\circ G_{X}(u)du\right)^{(p-1)}\\
		&=\left( \int_{0}^{\norm{X}_1}\tilde\theta(u)^{(p-1)\frac{(r-1)}{(r-p)}} du\right)^{(r-p)} \norm{X}_r^{(rp-r)}.
	\end{align*}
	Let us note that $\theta_h$ as defined in (\ref{eq:gammathetastat}) is non-increasing. Then, for any function $f$ we have 
	\begin{align*}
		f(\tilde\theta(u))&=f\left( \sum_{k\in V_0} \mathbb{1}_{\left\{u<\theta_{|k|}\right\}} \right) =\sum_{h=0}^\infty f\left( \sum_{k\in V_0} \mathbb{1}_{\left\{u<\theta_{|k|}\right\}} \right) \mathbb{1}_{\{ \theta_{h+1}\leq u<\theta_{h} \}}\\
		&=\sum_{h=0}^\infty f\left( \sum_{k\in V_0:|k|\leq h} 1 \right) \mathbb{1}_{\{ \theta_{h+1}\leq u<\theta_{h} \}}.
	\end{align*}
	Note that $ \sum_{k\in V_0:|k|\leq h} 1=\sum_{i=0}^{m-1}h(2h+1)^i=\frac{1}{2}\left( (2h+1)^m-1\right)$ such that the above is equal to
	\begin{gather}
		\label{p1}
		\sum_{h=0}^\infty f\left(\frac{1}{2}\Big((2h+1)^m-1\Big) \right) \mathbb{1}_{\{ \theta_{h+1}\leq u<\theta_{h} \}}.
	\end{gather} 
	Let us assume that $f$ is monotonically increasing, sub-multiplicative and $f(0)=0$ such that $f\left(\frac{1}{2}\Big((2h+1)^m-1\Big) \right)\leq f(2^{m-1})f((h+1)^m)=\sum_{k=0}^h f(2^{m-1})\left(f\left((k+1)^m \right)-f\left(k^m \right)\right)$. Finally we can deduce that (\ref{p1}) is less than or equal to
	\begin{gather*}
		\sum_{h=0}^\infty f(2^{m-1}) f\left((h+1)^m \right) \mathbb{1}_{\{ \theta_{h+1}\leq u<\theta_{h} \}}
		=f(2^{m-1})\sum_{h=0}^\infty \left( f\left((h+1)^m \right)- f\left(h^m\right)\right) \mathbb{1}_{\{ u<\theta_{h} \}}.
	\end{gather*}
	Applying the above result for $f(x)=x^{v}$ with $v=(p-1)\frac{(r-1)}{(r-p)}$ and noting that $(h+1)^{vm}-h^{vm}\leq vm (h+1)^{vm-1}$ for $vm\geq1$ and $(h+1)^{vm}-h^{vm}\leq vm~h^{vm-1}$ for $vm<1$ ($h>0$ by the mean value theorem), we get that for a constant $C=2^{v(m-1)}\theta_1>0$
	\begin{align*}
		&\Bigg( \int_{0}^{\norm{X}_1}(\tilde\theta(u))^{(p-1)\frac{(r-1)}{(r-p)}} du\Bigg)^{(r-p)}\\
		&\leq \begin{cases}
			\left(C+ \int_{0}^{\norm{X}_1}2^{v(m-1)} vm \sum_{h=0}^\infty  (h+1)^{vm-1}\mathbb{1}_{\{ u<\theta_{h} \}} du\right)^{(r-p)}\!\!\!\!\!\!\!\!\!&,\text{ if $vm\geq1$}\\
			\left(C+ \int_{0}^{\norm{X}_1}2^{v(m-1)} vm \sum_{h=1}^\infty  h^{vm-1}\mathbb{1}_{\{ u<\theta_{h} \}} du\right)^{(r-p)}&, \text{ if $vm<1$}
		\end{cases}\\
		&=\begin{cases}
			\left(C+2^{v(m-1)}vm \sum_{h=0}^\infty  (h+1)^{vm-1} \theta_h\right)^{(r-p)}&,\text{ if $vm\geq1$}\\
			\left(C+2^{v(m-1)}vm \sum_{h=1}^\infty  h^{vm-1} \theta_h\right)^{(r-p)}&,\text{ if $vm<1$}
		\end{cases}\\
		&\leq \max\left(1,2^{r-p-1}\right)\left(C^{r-p}+\left((vm)2^{v(m-1)}\left(~ \sum_{h=0}^\infty  (h+1)^{vm-1} \theta_h\right)\!\right)^{(r-p)}\right),
	\end{align*}
	which concludes the proof.
\end{proof}

Before we give the proof of Theorem \ref{theorem:clt}, we use Lemma \ref{lemma:thetaiscondexp} to show Proposition \ref{proposition:thetavsalpha} and Proposition \ref{proposition:AR1notmixingbuttheta}.

\begin{proof}[Proof of Proposition \ref{proposition:thetavsalpha}]
	From Lemma \ref{lemma:thetaiscondexp}, \cite[Lemma 1]{DD2003}, \cite[Proposition 25.15 (I) (a)]{B2007}, and by applying H\"older's inequality for $q >1$ and noting that $\int_0^1Q_{X_0}^r(u)du=E[|X_0|^r]$, we obtain that
	\begin{align*}
		\theta(h)&=\theta_h=\theta\big(\mathcal{F}_{V_0^h},X_0\big)\leq2 \int_0^{2\alpha\big(\mathcal{F}_{V_0^h},\sigma(\{X_0\})\big)}Q_{X_0}(u)du \\[-10pt]
		&= 2\!\int_0^{1}\!\!\mathbb{1}_{\left\{u\leq2\alpha\big(\mathcal{F}_{V_0^h},\sigma(\{X_0\})\big)\right\}}Q_{X_0}(u)du
		\leq 2^{\frac{2q-1}{q}}\!\!\alpha\big(\mathcal{F}_{V_0^h},\sigma(\{X_0\})\big)^{\frac{q-1}{q}}\!\!\left( \int_0^{1}Q_{X_0}^q(u)du\right)^{\frac{1}{q}}.
	\end{align*}
	Thus, we have that $\theta(h)\leq 2^{\frac{2q-1}{q}}\alpha(h)^{\frac{q-1}{q}} \norm{X_0}_q$ for $m=1$, and $\theta(h)\leq 2^{\frac{2q-1}{q}}\alpha_{\infty,1}(h)^{\frac{q-1}{q}} \norm{X_0}_q$ for all $m\geq1$.
\end{proof}

\begin{proof}[Proof of Proposition \ref{proposition:AR1notmixingbuttheta}]
	Consider $\theta_{p,r}(h)$ and $\tilde\delta_{1,n}$ as defined in \cite[Definition 2.3 and Section 3.1.4]{DDLLLP2008}, respectively. Due to Lemma \ref{lemma:thetaiscondexp} it holds that $\theta(h)=\theta_h=\theta_{1,1}(h)\leq \theta_{1,\infty}(h)\leq \tilde\delta_{1,h}$, where the last inequality follows from \cite[Section 3.1.4]{DDLLLP2008}. We define $(\tilde\xi_{t})_{t\in\Z}$ as $\tilde\xi_{t}=\xi_t$ for $t>0$ and $\tilde\xi_{t}=\xi'_t$ for $t\leq0$, where $\xi'_t$ is an independent copy of $\xi_t$. For $\tilde{X}_t=\sum_{j=0}^\infty 2^{-j-1}\tilde\xi_{t-j}$, we have
	\begin{align*}
		\norm{X_h-\tilde{X}_h}_1\leq 2^{-h}=:\tilde\delta_{1,h}\underset{h\rightarrow\infty}{\longrightarrow}0.
	\end{align*}
	Thus, $X$ is $\theta$-lex-weakly dependent. To show that $X$ is neither $\alpha$- nor $\alpha_{\infty,1}$-mixing, we follow the idea given in \cite[Section 1.5]{DDLLLP2008}. The set $A=\{X_0\leq \frac{1}{2}\}$ belongs to the past $\sigma$-algebra $\sigma(\{X_s,s\leq0\})$, to the $\sigma$-algebra generated by $X_h$ and to the future $\sigma$-algebra $\sigma(\{X_s,s\geq h\})$. Hence for all $h$, $\alpha(h)\geq |P(A)P(A)-P(A)|=\frac{1}{4}$, and similarly $\alpha_{\infty,1}(h)\geq \frac{1}{4}$. 
\end{proof}

\begin{proof}[Proof of Theorem \ref{theorem:clt}]
	In order to use \cite[Theorem 1]{D1998} we need to show that
	\begin{gather*}
		\sum_{k\in V_0}|E[X_kE_{|k|}[X_k]]|<\infty.
	\end{gather*}
	By Lemma \ref{lemma:finitesumclt} and Lemma \ref{lemma:summability} the result is proven if $\theta_{h}\in \mathcal{O}(h^{-\alpha})$ with $\alpha>m(1+\frac{1}{\delta})$.
	Finally, since $X$ is stationary an application of Lemma \ref{lemma:thetaiscondexp} concludes.
\end{proof}

\subsection{Proofs of Section \ref{sec3-3}}
\label{sec5-2}

\begin{proof}[Proof of Proposition \ref{proposition:mmathetaweaklydep}]
	\begin{enumerate}[(i)]
		\item Let $t\in\R^{m}$, $\psi>0$. We restrict the MMAF $X$ to a finite support and define the truncated sequence
		\begin{gather}\label{equation:truncatedtheta}
			X_{t}^{(\psi)}=\int_S\int_{A_{t}\backslash V_{t}^\psi}f(A,t-s)\Lambda(dA,ds).		\end{gather} 
		Note that the kernel function $f$ is square integrable such that (\ref{equation:intcond1}), (\ref{equation:intcond2}) and (\ref{equation:intcond3}) hold. Therefore, $f$ is $\Lambda$-integrable. Since $E[X_{t}X_{t}']<\infty$ for all $t\in\R^m$, by Proposition \ref{proposition:MMAexistencemoments} we can derive an upper bound of the expectation
		\begin{align}
			\begin{aligned}\label{eq:L1norminequality}
				E\big[\norm{X_{t}-X_{t}^{(\psi)}}\big]&=E\bigg[ \Big\lVert\int_S\int_{A_t\cap V_{t}^\psi}f(A,t-s)\Lambda(dA,ds)\Big\rVert\bigg]\\ 
				&\leq E\bigg[ \Big\lVert\int_S\int_{A_t\cap V_{t}^\psi}f(A,t-s)\Lambda(dA,ds)\Big\rVert^2\bigg]^{\frac{1}{2}}\\
				& =\left(\sum_{\kappa=1}^nE\Bigg[ \bigg( \Big(\int_S\int_{A_t\cap V_{t}^\psi}f(A,t-s)\Lambda(dA,ds)\Big)^{(\kappa)}\bigg)^2\Bigg]\right)^{\frac{1}{2}}.
			\end{aligned}
		\end{align}
		Using Proposition \ref{proposition:MMAmoments} and the translation invariance of $A_t$ and $V_t^\psi$, the above is equal to
		\begin{gather*}
			\Big(\int_S\int_{A_0 \cap V_{0}^\psi}\textup{tr}(f(A,-s)\Sigma_\Lambda f(A,-s)')ds\pi(dA)\Big)^{\frac{1}{2}}.
		\end{gather*}
		Let $u\in\N, h\in\R^{+}, \Gamma=\{i_1,\ldots,i_u\}\in(\R^m)^u$ and $j\in\R^m$ as in Definition \ref{thetaweaklydependent} such that $i_1,\ldots,i_u\in V_j^h$. Moreover, let $G\in\mathcal{G}_1$ and $F\in\mathcal{G}_u^*$. For $a\in\{1,\ldots,u\}$ define
		\begin{gather*}
			X_{i_a}=\int_S\int_{A_{i_a}}f(A,i_a-s)\Lambda(dA,ds), \quad \textrm{and} \quad
			X_{j}^{(\psi)}=\int_S\int_{A_{j}\backslash V_{j}^\psi}f(A,j-s)\Lambda(dA,ds).
		\end{gather*}
		W.l.o.g. we assume that $i_a\leq_{lex}i_u$ for all $a\in\{1,\ldots,u\}$. If there exists a $\psi$ such that $A_{i_u}\cap A_{j}\backslash V_{j}^\psi=\emptyset$, then $A_{i_a}\cap A_{j}\backslash V_{j}^\psi=\emptyset$. \\
		Now, $A$ is translation invariant with initial sphere of influence $A_0$. Furthermore, $A_0$ satisfies (\ref{condition:scalarproduct}). Then, for $\psi(h)$ as defined in (\ref{equation:psi}) it holds that $A_{i_u}\cap A_{j}\backslash V_{j}^\psi=\emptyset$.\\
		From now on we set $\psi=\psi(h)$. We then get that $I_a=S\times A_{i_a}$ and $J=S\times A_{j}\backslash V_{j}^\psi$ are disjoint or have intersection on a set $S\times O$, where $O\subset\R^m$ and $\dim(O)<m$. Since $(\pi\times\lambda)(S\times O)=0$ and by the definition of a L\'evy basis, $X_{i_a}$ and $X_j^{(\psi)}$ are independent for all $a\in\{1,\ldots,u\}$. Finally, we get that $X_\Gamma$ and $X_j^{(\psi)}$ are independent and therefore also $F(X_\Gamma)$ and $G(X_j^{(\psi)})$. Now
		\begin{align*}
			& |Cov(F(X_\Gamma),G(X_j))|\\&\leq |Cov(F(X_\Gamma),G(X_j^{(\psi)}))|+|Cov(F(X_\Gamma),G(X_j)-G(X_j^{(\psi)}))|\\
			&=|E[(G(X_j)-G(X_j^{(\psi)}))F(X_\Gamma)]-E[G(X_j)-G(X_j^{(\psi)})]E[F(X_\Gamma)]|\\
			&\leq 2\norm{F}_\infty E\big[|G(X_j)-G(X_j^{(\psi)})|\big] \leq 2Lip(G) \norm{F}_\infty E[\norm{X_{j}-X_{j}^{(\psi)}}]\\
			&\leq 2 Lip(G)\norm{F}_\infty \Big(\int_S\int_{A_0\cap V_{0}^\psi}\textup{tr}(f(A,-s)\Sigma_\Lambda f(A,-s)')ds\pi(dA)\Big)^{\frac{1}{2}},
		\end{align*}
		using (\ref{eq:L1norminequality}). Therefore, $X$ is $\theta$-lex weakly dependent with $\theta$-lex-coefficients
		\begin{gather*}
			\theta_X(h)\leq 2 \Big(\int_S\int_{A_0\cap V_{0}^\psi}\textup{tr}(f(A,-s)\Sigma_\Lambda f(A,-s)')ds\pi(dA)\Big)^{\frac{1}{2}},
		\end{gather*}
		which converge to zero as $h$ goes to infinity by applying the dominated convergence theorem.
		\item Let $t\in\R^m$, $\psi>0$ and $X_t^{(\psi)}$ be defined as in (i). By applying Proposition \ref{proposition:MMAmoments}, we obtain
		\begin{align*}
			E\big[\norm{X_t-X_t^{(\psi)}}\big]&\leq\bigg(\int_S\int_{A_0\cap V_t^{\psi}}\textup{tr}(f(A,t-s)\Sigma_\Lambda f(A,t-s)')ds\pi(dA)\\ &\qquad+\Big\lVert\int_S\int_{A_0\cap V_t^{\psi}} f(A,t-s)\mu_\Lambda  ds\pi(dA)\Big\rVert^2\bigg)^{\frac{1}{2}}.
		\end{align*} 
		Finally, by proceeding as in the proof of (i), we obtain the stated bound for the $\theta$-lex-coefficients.
		\item Since the kernel function $f$ is in $L^1$ the Equations (\ref{equation:intcondfinvar1}) and (\ref{equation:intcondfinvar2}) hold. Moreover, by Proposition \ref{proposition:MMAexistencemoments} $f$ is $\Lambda$-integrable and $E[X_{t}]<\infty$. Let $t\in\R^m$, $\psi>0$ and $X_t^{(\psi)}$ be defined as in (i). Then, we can derive with the help of Proposition \ref{proposition:MMAmoments} that
		\begin{align*}
			& E\big[\norm{X_{t}-X_{t}^{(\psi)}}\big]\\
			&\leq \Big(\int_S\int_{A_0\cap V_{0}^\psi} \big\lVert f(A,-s)\gamma_0\big\rVert ds\pi(dA)+\int_S\int_{A_0 \cap V_{0}^\psi} \int_{\R^d} \big\lVert f(A,-s)y\big\rVert \nu(dy)ds\pi(dA) \Big),
		\end{align*}
		where we used that $E[\int_Ef(t)d\mu(t)]=\int_Ef(t)d\nu(t)$ for a Poisson random measure $\mu$ with corresponding intensity measure $\nu$ and an arbitrary set $E$.\\
		Now for $F$, $G$, $X_\Gamma$, $X_j$ and $\psi=\psi(h)$ as described in the proof of (i) we get
		\begin{align*}
			|Cov(F(X_\Gamma),G(X_j))|
			\leq2Lip(G)\norm{F}_\infty& \Big(\int_S\int_{A_0 \cap V_{0}^{\psi(h)}} \big\lVert f(A,-s)\gamma_0\big\rVert ds\pi(dA)\\&+\int_S\int_{A_0\cap V_{0}^{\psi(h)}} \int_{\R^d} \big\lVert f(A,-s)y\big\rVert \nu(dy)ds\pi(dA) \Big).
		\end{align*}
		Therefore $X$ is $\theta$-lex weakly dependent with $\theta$-lex-coefficients
		\begin{align*}
			\theta_X(h)\leq 2 \Big(&\int_S\int_{A_0 \cap V_{0}^{\psi(h)}} \big\lVert f(A,-s)\gamma_0\big\rVert ds\pi(dA)\\&+\int_S\int_{A_0\cap V_{0}^{\psi(h)}} \int_{\R^d} \big\lVert f(A,-s)y\big\rVert \nu(dy)ds\pi(dA) \Big),
		\end{align*}
		which converge to zero as $h$ goes to infinity by applying the dominated convergence theorem.
		\item We use the notations described in (i). We have that $\Lambda$ is in distribution the sum of two $\R^d$-valued independent L\'evy bases $\Lambda_1$ and $\Lambda_2$ with characteristic quadruplets $(\gamma,\Sigma,\restr{\nu}{\norm{x}\leq1},\pi)$ and $(0,0,\restr{\nu}{\norm{x}>1},\pi)$, respectively. Since $f\in L^1\cap L^2$ we know that both integrals $X_t^{(\Lambda_1)}=\int_S\int_{\R^m}f(A,t-s)\Lambda_1(dA,ds)$ and $X_t^{(\Lambda_2)}=\int_S\int_{\R^m}f(A,t-s)\Lambda_2(dA,ds)$ exist. Additionally, it holds that $\int_S\int_{\R^m}f(A,t-s)\Lambda(dA,ds)=X_t^{(\Lambda_1)}+X_t^{(\Lambda_2)}$. By noting that
		\begin{align*}
			E\big[\norm{X_t-X_t^{(\psi)}}\big]&\leq E\big[\norm{X_t^{(\Lambda_1)}-(X_t^{(\Lambda_1)})^{(\psi)}}\big] +E\big[\norm{X_t^{(\Lambda_2)}-(X_t^{(\Lambda_2)})^{(\psi)}}\big]\\
			&\leq E\bigg[ \Big\lVert X_t^{(\Lambda_1)}-(X_t^{(\Lambda_1)})^{(\psi)}\Big\rVert ^2\bigg]^{\frac{1}{2}}+E\bigg[\Big\lVert X_t^{(\Lambda_2)}-(X_t^{(\Lambda_2)})^{(\psi)}\Big\rVert\bigg]
		\end{align*}
		and following the proof of (ii) (for the first summand) and (iii) (for the second summand) we obtain the stated bound for the $\theta$-lex-coefficients.
	\end{enumerate}
\end{proof}

\begin{proof}[Proof of Proposition \ref{proposition:vectorinfluencedweaklydep}]
	In order to show that $Z$ is a well defined MMAF, we need to check that $g(A,s)$ is $\Lambda$-integrable as described in Theorem \ref{theorem:2}, i.e. $g(A,s)$ satisfies the conditions (\ref{equation:intcond1}), (\ref{equation:intcond2}) and (\ref{equation:intcond3}). For the sake of brevity we will consider in the following the norm $\norm{(x_1,\ldots,x_m)}=\norm{x_1}+\cdots+\norm{x_m}$ for $x_i\in\R$ for $i=1,\ldots,m$.\\
	Let us start by showing that $g(A,s)$ is $\Lambda$-integrable for $k=1$, then
	\begin{gather*}
		g(A,s)=\big(f\big(A,s\big),f\big(A,s-s_1,\ldots, f\big(A,s-s_{|S_1|}\big)'\text{, where }|S_1|=2\cdot 3^{m-1}.
	\end{gather*}
	Note that for $x\in\R^d$
	\begin{align*}
		\mathbb{1}_{[0,1]}(\norm{g(A,s)x})&\leq\mathbb{1}_{[0,1]}(\norm{f(A,s)x}),\\
		\mathbb{1}_{[0,1]}(\norm{g(A,s)x})&\leq\mathbb{1}_{[0,1]}(\norm{f(A,s-s_1)x}),\\
		&\ldots \\
		\mathbb{1}_{[0,1]}(\norm{g(A,s)x})&\leq\mathbb{1}_{[0,1]}(\norm{f(A,s-s_{|S_1|})x}),
	\end{align*}
	such that
	{\fontsize{10}{4}
		\begin{align*}
			&\int_S\!\int_{\R^m} \Big\lVert g(A,s)\gamma\!+\!\! \int_{\R^d}\!g(A,s)x\left(\mathbb{1}_{[0,1]}(\norm{g(A,s)x})\!-\mathbb{1}_{[0,1]}(\norm{x})\right)\nu(dx)\Big\rVert ds \pi(dA) \\
			&\leq \int_S\!\int_{\R^m} \Big\lVert f(A,s)\gamma\!+\!\! \int_{\R^d}\!f(A,s)x\left(\mathbb{1}_{[0,1]}(\norm{f(A,s)x})\!-\mathbb{1}_{[0,1]}(\norm{x})\right)\nu(dx)\Big\rVert ds \pi(dA)\\
			&+\int_S\!\int_{\R^m} \Big\lVert f(A,s-s_1)\gamma+\!\! \int_{\R^d}\!f(A,s-s_1)x\left(\mathbb{1}_{[0,1]}(\norm{f(A,s-s_1)x})\!-\mathbb{1}_{[0,1]}(\norm{x})\right)\nu(dx)\Big\rVert ds \pi(dA)\\
			&+\cdots+\!\!\int_S\!\int_{\R^m} \Big\lVert f\big(A,s-s_{|S_1|})\gamma\!+\!\! \int_{\R^d}\!f(A,s-s_{|S_1|})\\
			&\qquad\qquad\times x\left(\mathbb{1}_{[0,1]}(\norm{f(A,s-s_{|S_1|})}\!-\mathbb{1}_{[0,1]}(\norm{x})\right)\!\nu(dx)\Big\rVert ds \pi(dA).
	\end{align*}}
	Since $f$ is $\Lambda$-integrable, we can conclude that the above expression is finite and (\ref{equation:intcond1}) holds. Now
	\begin{gather*}
		\int_S\!\int_{\R^m}\!\norm{g(A,s)\Sigma g(A,s)'} ds\pi(dA)\\
		=\int_S\!\int_{\R^m}\!\!\norm{f(A,s)\Sigma f(A,s)'} ds\pi(dA) \!+\!\!\! \int_S\!\int_{\R^m}\!\!\norm{f(A,s-s_1)\Sigma f(A,s-s_1)'} ds\pi(dA)\\
		+\ldots+\int_S\!\int_{\R^m}\!\!\norm{f(A,s-s_{|S_1|})\Sigma f(A,s-s_{|S_1|})'} ds\pi(dA),
	\end{gather*}
	and it is finite since $f$ is $\Lambda$-integrable and (\ref{equation:intcond2}) holds. Since $\left(\sum_{i=1}^na_i\right)^2\leq n \sum_{i=1}^na_i^2$, we have
	\begin{gather*}
		\norm{g(A,s)}^2\leq |S_1| \Big( \norm{f(A,s)}^2+\norm{f(A,s-s_1)}^2+\ldots+\norm{f(A,s-s_{|S_1|})}^2\Big),
	\end{gather*}
	and finally
	\begin{align*}
		&\int_S\int_{\R^m}\int_{\R^d}  \Big(1\wedge \norm{g(A,s)x}^2 \Big)\nu(dx)ds\pi(dA)\\
		&\leq |S_k| \bigg(\int_S\int_{\R^m}\int_{\R^d}  \Big(1\wedge  \norm{f(A,s)}^2 \Big)\nu(dx)ds\pi(dA)\\ &\qquad+\int_S\int_{\R^m}\int_{\R^d}  \Big(1\wedge  \norm{f(A,s-s_1)}^2 \Big)\nu(dx)ds\pi(dA)\\
		&\qquad+\ldots+ \int_S\int_{\R^m}\int_{\R^d}  \Big(1\wedge \norm{f(A,s-s_{|S_1|})}^2\Big)\nu(dx)ds\pi(dA)\bigg),
	\end{align*}
	that is finite since $f$ satisfies  (\ref{equation:intcond3}). Thus, $g$ is $\Lambda$-integrable and $Z$ is an $(A,\Lambda)$-influenced MMAF. By induction, the above statement can be shown for each $k \in \N$.\\
	Assume that $X$ satisfies the assumptions of Proposition \ref{proposition:mmathetaweaklydep} (i) and consider $\psi(h)$ as defined in (\ref{equation:psi}). Then,
	\begin{align*}
		\theta_Z^{(i)}(h)\leq&2\Bigg(\int_S\int_{A_0\cap V_{0}^{\psi(h)}}\textup{tr}\Big(g(A,-s)\Sigma_\Lambda g(A,-s)'\Big)ds\pi(dA)\Bigg)^{\frac{1}{2}}\\
		=&2\Bigg(\int_S\int_{A_0\cap V_{0}^{\psi(h)}}\textup{tr}\Big(f(A,-s)\Sigma_\Lambda f(A,-s)'\Big)ds\pi(dA)\\
		&~+\int_S\int_{A_0\cap V_{0}^{\psi(h)}}\textup{tr}\Big(f(A,s_1-s)\Sigma_\Lambda f(A,s_1-s)'\Big)ds\pi(dA)\\
		&~+\ldots+\int_S\int_{A_0\cap V_{0}^{\psi(h)}}\textup{tr}\Big(f(A,s_{|S_k|}-s)\Sigma_\Lambda f(A,s_{|S_k|}-s\big)')ds\pi(dA)\Bigg)^{\frac{1}{2}}\\
		\leq&2 |S_k|^{\frac{m}{2}} \Bigg(\int_S\int_{A_0\cap V_{0}^{\psi(h)-k}}\textup{tr}\Big(f\big(A,-s\big)\Sigma_\Lambda f\big(A,-s\big)'\Big)ds\pi(dA)\Bigg)^{\frac{1}{2}}\\
		=&|S_k|^{\frac{m}{2}} \hat{\theta}_X^{(i)}(h-\psi^{-1}(k)).
	\end{align*}
	where $\psi^{-1}$ denotes the inverse of $\psi$, for all $\psi(h)>k$. Thus, $Z$ is a $(k+1)(2k+1)^{m-1}$-dimensional $\theta$-lex-weakly dependent MMAF. Similar calculations lead to the other statements in (\ref{eq:mmainfluencedvectorweakly}).
\end{proof}

\subsection{Proof of Section \ref{sec3-4}}
\label{sec5-3}
\begin{proof}[Proof of Theorem \ref{theorem:thetasamplemeanspecial}]
	Let us first consider $X$ to be univariate. In order to use \cite[Theorem 1]{D1998} we need to show 
	\begin{gather}\label{equation:L1condition}
		\sum_{k\in V_0}|X_kE_{|k|}[X_0]|\in L^1.
	\end{gather} The H\"older inequality implies
	\begin{gather*}
		\norm{X_kE_{|k|}(X_0)}_{1}\leq \norm{X_k}_{2}\norm{E[X_0|\mathcal{F}_{V_0^{|k|}}]}_{2},
	\end{gather*}
	where $ \norm{X_k}_{2}<C$ for all $k$ and a constant $C$. Furthermore, we note that $\lVert E[X|\mathcal{F}] \rVert_{2}= \lVert E[E[X|\mathcal{H}]|\mathcal{F}] \rVert_{2}\leq \lVert E[X|\mathcal{H}]\rVert_{2}$ holds for a  $\sigma$-algebra $\mathcal{H}$, a sub $\sigma$-algebra $\mathcal{F}$ and an $L^2$ random variable $X$, as the conditional expectation is the orthogonal projection in $L^2$. Now, using (\ref{equation:influencedMMAfield})
	\begin{align*}
		\norm{E[X_0|\mathcal{F}_{V_0^{|k|}}]}_{2}=& \left\lVert E\left[ \int_S\int_{V_0}\mathbb{1}_{A_0}(s) f(A,-s)\Lambda(dA,ds)\Big|\sigma(X_l:l\in V_0^{|k|} )\right]\right\rVert_{2}\\
		\leq& \left\lVert E\left[ \int_S\int_{V_0}\mathbb{1}_{A_0}(s) f(A,-s)\Lambda(dA,ds)\Big|\sigma(\Lambda(B):B\in\mathcal{B}(V_0^{|k|}))\right]\right\rVert_{2}\\
		=&\Bigg\lVert E\left[ \int_S\int_{V_0^{|k|}}\mathbb{1}_{A_0}(s) f(A,-s)\Lambda(dA,ds)\Big|\sigma(\Lambda(B):B\in\mathcal{B}(V_0^{|k|}))\right]\\
		&~+E\left[\int_S\int_{V_0\backslash V_0^{|k|}}\mathbb{1}_{A_0}(s) f(A,-s)\Lambda(dA,ds)\Big|\sigma(\Lambda(B):B\in\mathcal{B}(V_0^{|k|}))\right]\Bigg\rVert_{2}.
	\end{align*}
	We note that $\int_S\int_{V_0^{|k|}}\mathbb{1}_{A_0}(s) f(A,-s)\Lambda(dA,ds)$ is measurable with respect to $\sigma(\Lambda(B):B\in\mathcal{B}(V_0^{|k|}))$. Since $\Lambda$ is a L\'evy basis (in particular independent for disjoint sets) we get that $\int_S\int_{V_0\backslash V_0^{|k|}}\mathbb{1}_{A_0}(s) f(A,-s)$ $\Lambda(dA,ds)$ is independent of $\sigma(\Lambda(B):B\in\mathcal{B}(V_0^{|k|}))$, such that the above equation is equal to
	\begin{gather*}
		\Bigg\lVert \int_S\int_{V_0^{|k|}}\mathbb{1}_{A_0}(s) f(A,-s)\Lambda(dA,ds)+E\left[\int_S\int_{V_0\backslash V_0^{|k|}}\mathbb{1}_{A_0}(s)f(A,-s)\Lambda(dA,ds)\right]\Bigg\rVert_{2}.
	\end{gather*}
	Since $\gamma+\int_{\norm{x}>1}x\nu(dx)=0$ the second summand is equal to zero and the above is equal to
	\begin{align*}
		&\Bigg\lVert \int_S\int_{V_0^{|k|}}\mathbb{1}_{A_0}(s) f(A,-s) f(A,-s)\Lambda(dA,ds)\Bigg\rVert_{2}\\
		&=\Big(\int_S\int_{A_0\cap V_0^{|k|}}\textup{tr}(f(A,-s)\Sigma_\Lambda f(A,-s)')ds\pi(dA)\Big)^{\frac{1}{2}}
		=\theta_X(|k|),
	\end{align*}
	using Proposition \ref{proposition:MMAmoments}.\\ 
	The stated result then follows from \cite[Theorem 1]{D1998} using the dominated convergence theorem. 
	The Cram\'er-Wold device establishes the multivariate case straightforwardly.
\end{proof}

\subsection{Proofs of Section \ref{sec3-5}}
\label{sec5-4}
\begin{proof}[Proof of Proposition \ref{proposition:mmaetaweaklydep}]
	\begin{enumerate}[(i)]
		\item Let $t\in\R^m$ and $\psi>0$. We restrict the MMAF to a finite support, and define the sequence
		\begin{align}\label{equation:truncatedeta}
			\begin{aligned}
				X_t^{(\psi)}&=\int_S\int_{\R^m}f(A,t-s)\mathbb{1}_{(-\psi,\psi)^m}(t-s)\Lambda(dA,ds)\\
				&=\int_S\int_{(t-\psi,t+{\psi})^m}f(A,t-s)\Lambda(dA,ds).
			\end{aligned}
		\end{align} 
		Note that the kernel function $f$ is square integrable such that (\ref{equation:intcond1}), (\ref{equation:intcond2}) and (\ref{equation:intcond3}) hold. Therefore, $f$ is $\Lambda$-integrable. Since $E[X_tX_t']<\infty$ for all $t\in\R^m$, by Proposition \ref{proposition:MMAexistencemoments} we can derive an upper bound of the expectation
		\begin{align}
			\begin{aligned}\label{eq:L1norminequality2}
				&E\big[\norm{X_t-X_t^{(\psi)}}\big]=E\bigg[ \Big\lVert\int_S\int_{\big((t-\psi,t+\psi)^m\big)^c}f(A,t-s)\Lambda(dA,ds)\Big\rVert\bigg]\\ 
				&\qquad\leq E\bigg[ \Big\lVert\int_S\int_{\big((t-\psi,t+\psi)^m\big)^c}f(A,t-s)\Lambda(dA,ds)\Big\rVert^2\bigg]^{\frac{1}{2}}\\
				&\qquad= \left(\sum_{\kappa=1}^nE\Bigg[ \bigg( \Big(\int_S\int_{\big((t-\psi,t+\psi)^m\big)^c}f(A,t-s)\Lambda(dA,ds)\Big)^{(\kappa)}\bigg)^2\Bigg]\right)^{\frac{1}{2}},
			\end{aligned}
		\end{align}
		where $x^{(\kappa)}$ denotes the $\kappa$th coordinate of $x\in\R^n$.
		Using Proposition \ref{proposition:MMAmoments} and the stationarity of $X$ this is equal to
		\begin{gather*}
			\Big(\int_S\int_{\big((\psi,\psi)^m\big)^c}\textup{tr}(f(A,-s)\Sigma_\Lambda f(A,-s)')ds\pi(dA)\Big)^{\frac{1}{2}}.
		\end{gather*}
		For $(u,v)\in\N\times\N$ let $F\in\mathcal{G}_u$, $G\in\mathcal{G}_v$, $h\in\R^{+}, \Gamma_i=\{i_1,\ldots,i_u\}\in(\R^m)^u$ and $\Gamma_j=\{j_1,\ldots,j_v\}\in(\R^m)^v$ as in Definition \ref{etaweaklydependent} such that $dist(\Gamma_i,\Gamma_j)\geq h$. For $a\in\{1,\ldots,u\}$ and $b\in\{1,\ldots,v\}$ define
		\begin{align*}
			X_{i_a}^{(\psi)}&=\int_S\int_{(i_a-{\psi},i_a+{\psi})^m}f(A,i_a-s)\,\Lambda(dA,ds) \text{ and }\\
			X_{j_b}^{(\psi)}&=\int_S\int_{(j_b-{\psi},j_b+{\psi})^m}f(A,j_b-s)\,\Lambda(dA,ds).
		\end{align*}
		Now consider $a\in\{1,\ldots,u\}$ and $b\in\{1,\ldots,v\}$ such that $\inf_{1\leq x \leq u, 1\leq y \leq v}\norm{i_x-j_y}_\infty=\norm{i_a-j_b}_\infty$. Define the two sets $I_a=S\times (i_a-{\psi},i_a+{\psi})^m$ and $J_b=S\times (j_b-{\psi},j_b+{\psi})^m$. Let $\psi=\frac{h}{2}$ and $\norm{i_a-j_b}_\infty\geq h$. Then, it holds that $I_a$ and $J_b$ are disjoint as well as $I_{\tilde{a}}$ and $J_{\tilde{b}}$ for all $\tilde{a}=1,\ldots,u$ and $\tilde{b}=1,\ldots,v$. By the definition of a L\'evy basis $X_{i_a}^{(\psi)}$ and $X_{j_b}^{(\psi)}$ are independent for all $a\in\{1,\ldots,u\}$ and $b\in\{1,\ldots,v\}$. Finally, we get that $X_{\Gamma_i}^{(\psi)}$ and $X_{\Gamma_j}^{(\psi)}$ are independent and therefore also $F(X_{\Gamma_i}^{(\psi)})$ and $G(X_{\Gamma_j}^{(\psi)})$. Now, 
		\begin{align*}
			&|Cov(F(X_{\Gamma_i}),G(X_{\Gamma_j}))|\\&\leq |Cov(F(X_{\Gamma_i})-F(X_{\Gamma_i}^{(\psi)}),G(X_{\Gamma_j}))|+|Cov(F(X_{\Gamma_i}^{(\psi)}),G(X_{\Gamma_j})-G(X_{\Gamma_i}^{(\psi)}))|\\
			&= |E[(F(X_{\Gamma_i})-F(X_{\Gamma_i}^{(\psi)}))G(X_{\Gamma_j})]-E[F(X_{\Gamma_i})-F(X_{\Gamma_i}^{(\psi)})]E[G(X_{\Gamma_j})]|\\
			&~+|E[(G(X_{\Gamma_j})-G(X_{\Gamma_j}^{(\psi)}))F(X_{\Gamma_i}^{(\psi)})]-E[G(X_{\Gamma_j})-G(X_{\Gamma_j}^{(\psi)})]E[F(X_{\Gamma_i}^{(\psi)})]|\\
			&\leq 2\Big(\norm{G}_\infty E\big[|F(X_{\Gamma_i})-F(X_{\Gamma_i}^{(\psi)})|\big]+ \norm{F}_\infty E\big[|G(X_{\Gamma_j})-G(X_{\Gamma_i}^{(\psi)})|\big]\Big)\\
			&\leq 2 \Big(\norm{G}_\infty Lip(F) \sum_{l=1}^u E[\norm{X_{i_l}-X_{i_l}^{(\psi)}}] +\norm{F}_\infty Lip(G) \sum_{k=1}^v E[\norm{X_{j_k}-X_{j_k}^{(\psi)}}] \Big)\\
			&\leq 2 (u\norm{G}_\infty Lip(F)+v\norm{F}_\infty Lip(G))\\
			&\qquad\times \Big(\int_S\int_{\big(\big(-\frac{{h}}{2},\frac{{h}}{2}\big)^m\big)^c}\textup{tr}(f(A,-s)\Sigma_\Lambda f(A,-s)')ds\pi(dA)\Big)^{\frac{1}{2}},
		\end{align*}
		using (\ref{eq:L1norminequality2}). Therefore, $X$ is $\eta$-weakly dependent with $\eta$-coefficients
		\begin{gather*}
			\eta_X(h)\leq 2 \Big(\int_S\int_{\big(\big(-\frac{{h}}{2},\frac{{h}}{2}\big)^m\big)^c}\textup{tr}(f(A,-s)\Sigma_\Lambda f(A,-s)')ds\pi(dA)\Big)^{\frac{1}{2}},
		\end{gather*}
		which converge to zero as $h$ goes to infinity by applying the dominated convergence theorem.
		\item[(iii)] Since $f\in L^1$, (\ref{equation:intcondfinvar1}) and (\ref{equation:intcondfinvar2}) hold and $f$ is $\Lambda$-integrable. Let $t\in\R^m$, $\psi>0$ and $X_t^{(\psi)}$ be defined as in (i). Moreover, Proposition \ref{proposition:MMAexistencemoments}  implies that $E[X_t]<\infty$. Then, using Proposition \ref{proposition:MMAmoments}
		\begin{align*}
			E\big[\norm{X_t-X_t^{(\psi)}}\big]
			&\leq \Big(\int_S\int_{\big((t-\psi,t+{\psi})^m\big)^c} \big\lVert f(A,-s)\gamma_0\big\rVert ds\pi(dA)\\
			&\qquad+\int_S\int_{\big((t-\psi,t+{\psi})^m\big)^c} \int_{\R^d} \big\lVert f(A,-s)x\big\rVert \nu(dx)ds\pi(dA) \Big).
		\end{align*}
		Now, for $F$, $G$, $X_{\Gamma_i}$ and $X_{\Gamma_j}$ and $\psi$ as described in (i), we get
		\begin{align*}
			|Cov(F(X_{\Gamma_i}),G(X_{\Gamma_j}))| &\leq
			2 (u\norm{G}_\infty Lip(F)+v\norm{F}_\infty Lip(G))\\
			&\quad \times\Bigg( \int_S\int_{\big(\big(-\frac{{h}}{2},\frac{{h}}{2}\big)^m\big)^c} \big\lVert f(A,-s)\gamma_0\big\rVert ds\pi(dA)\\
			&\qquad\quad+\int_S\int_{\big(\big(-\frac{{h}}{2},\frac{{h}}{2}\big)^m\big)^c} \int_{\R^d} \big\lVert f(A,-s)x\big\rVert \nu(dx)ds\pi(dA)\Bigg).
		\end{align*}
		Therefore, $X$ is $\eta$ weakly dependent with $\eta$-coefficients
		\begin{align*}
			\eta_X(h)&\leq 2 \Bigg( \int_S\int_{\big(\big(-\frac{{h}}{2},\frac{{h}}{2}\big)^m\big)^c} \big\lVert f(A,-s)\gamma_0\big\rVert ds\pi(dA)\\
			&\qquad +\int_S\int_{\big(\big(-\frac{{h}}{2},\frac{{h}}{2}\big)^m\big)^c} \int_{\R^d} \big\lVert f(A,-s)x\big\rVert \nu(dx)ds\pi(dA)\!\Bigg),
		\end{align*}
		which converge to zero as $h$ goes to infinity by applying the dominated convergence theorem.
	\end{enumerate}
	\noindent
	Part (ii) and (iv) of the Proposition follow from the above results, analogously to Proposition \ref{proposition:mmathetaweaklydep} (ii) and (iv).
\end{proof}

\subsection{Proofs of Section \ref{sec3-8}}
\begin{proof}[Proof of Theorem \ref{theorem:levydrivenexpdecaykernel}]
	Let $\norm{A}_F=\sqrt{tr(AA')}$ for $A\in M_{n\times d}(\R)$ denote the Frobenius norm and $\norm{x}_1=\sum_{\nu=1}^m|x^{(\nu)}|$ for $x\in\R^m$. From Proposition \ref{proposition:mmaetaweaklydep} it follows that $X$ is $\eta$-weakly dependent with $\eta$-coefficients
	\begin{align*}
		\eta_X(h)&=\Bigg(\int_{\left(\left(-\frac{h}{2},\frac{h}{2}\right)^m\right)^c}\textup{tr}(g(-s)\Sigma_Lg(-s)')ds\Bigg)^{\frac{1}{2}}
		=\Bigg(\int_{\left(\left(-\frac{h}{2},\frac{h}{2}\right)^m\right)^c} \norm{g(-s)\Sigma^{\frac{1}{2}}}_F^2  ds\Bigg)^{\frac{1}{2}}\\
		&\leq \Bigg(\norm{\Sigma^{\frac{1}{2}}}_F^2 \int_{\left(\left(-\frac{h}{2},\frac{h}{2}\right)^m\right)^c} \norm{g(-s)}_F^2  ds\Bigg)^{\frac{1}{2}}
		\leq \Bigg(\norm{\Sigma^{\frac{1}{2}}}_F^2M \int_{\left(\left(-\frac{h}{2},\frac{h}{2}\right)^m\right)^c}  e^{-K\norm{s}_1}ds\Bigg)^{\frac{1}{2}}\\
		&= \norm{\Sigma^{\frac{1}{2}}}_F M^{\frac{1}{2}}\Bigg(\int_{\R^m}  e^{-K\norm{s}_1}ds-\int_{\left(-\frac{h}{2},\frac{h}{2}\right)^m}  e^{-K\norm{s}_1}ds\Bigg)^{\frac{1}{2}}\\
		&=\norm{\Sigma^{\frac{1}{2}}}_F M^{\frac{1}{2}}\Bigg(\left(\frac{1}{2K}\right)^m-\left(\frac{1}{2K}-\frac{e^{-\frac{K}{2}h}}{2K}\right)^m \Bigg)^{\frac{1}{2}}\\
		&=  \frac{\norm{\Sigma^{\frac{1}{2}}}_FM^{\frac{1}{2}}}{(2K)^\frac{m}{2}}\Bigg(1-\left(1-e^{-\frac{K}{2}h}\right)^m \Bigg)^{\frac{1}{2}}
		\leq \frac{m\norm{\Sigma^{\frac{1}{2}}}_FM^{\frac{1}{2}}}{(2K)^\frac{m}{2}} e^{-\frac{K}{4}h},
	\end{align*}
	where the last inequality follows from Bernoulli's inequality.
\end{proof}

\subsection{Proofs of Section \ref{sec4-2}}
\label{sec5-5}
\begin{proof}[Proof of Proposition \ref{proposition:ambitthetaweaklydep}]
	\begin{enumerate}[(i)]
		\item Let $(t,x)\in\R\times\R^{m}$, $\psi>0$. We define the two truncated sequences 
		\begin{align*}
			&\tilde{Y}_{t}^{(\psi)}(x)=\int_{A_t(x)\backslash V_{(t,x)}^\psi}l(x-\xi,t-s)\sigma_s(\xi)\Lambda(d\xi,ds), \,\,\, \text{and}\\
			&Y_{t}^{(\psi)}(x)=\int_{A_{t}(x)\backslash V_{(t,x)}^\psi}l(x-\xi,t-s)\sigma_s^{(\psi)}(\xi)\Lambda(d\xi,ds),\,\, \text{where}\\
			&\sigma_t^{(\psi)}(x)=\int_S\int_{A_{(t,x)}^\sigma(x)\backslash V_{(t,x)}^\psi}j(x-\xi,t-s) \Lambda^\sigma(dA,d\xi,ds).
		\end{align*}
		Since the kernel function $j$ is square integrable we have that (\ref{equation:intcond1}), (\ref{equation:intcond2}) and (\ref{equation:intcond3}) hold. Therefore, $j$ is $\Lambda^\sigma$-integrable and $\sigma$ is well-defined and stationary. Now, by Proposition \ref{proposition:MMAexistencemoments} it holds that $\sigma_t(x)\in L^2(\Omega)$. Since additionally $l\in L^2(\R^m\times\R)$ and $\sigma$ is stationary, it holds that $l\sigma\in L^2(\Omega\times\R^m\times\R)$. This implies that $l\sigma\in L^2(\R^m\times\R)$ almost surely. Then, $l\sigma$ satisfies (\ref{equation:intcond1}), (\ref{equation:intcond2}) and (\ref{equation:intcond3}) almost surely and the ambit field $Y$ is well-defined. Analogous to Proposition \ref{proposition:mmathetaweaklydep} and using Proposition \ref{proposition:ambitmoments} 
		\begin{align*}
			&E\big[|Y_{t}(x)-Y_{t}^{(\psi)}(x)|\big]\leq E\big[|Y_{t}(x)-\tilde{Y}_{t}^{(\psi)}(x)|\big] + E\big[|\tilde{Y}_{t}^{(\psi)}(x)-Y_{t}^{(\psi)}(x)|\big] \\
			&=E\bigg[ \Big|\int_{A_t(x)\cap V_{(t,x)}^\psi}l(x-\xi,t-s)\sigma_s(\xi)\Lambda(d\xi,ds)\Big|\bigg]\\
			&\quad+E\bigg[ \Big|\int_{A_t(x)\backslash V_{(t,x)}^\psi}l(x-\xi,t-s)(\sigma_s(\xi)-\sigma_s^{(\psi)}(\xi))\Lambda(d\xi,ds)\Big|\bigg] \\ 
			&\leq E\Bigg[ \bigg( \int_{A_t(x)\cap V_{(t,x)}^\psi}l(x-\xi,t-s)\sigma_s(\xi)\Lambda(d\xi,ds)\bigg)^2\Bigg]^{\frac{1}{2}}\\
			&\quad+E\Bigg[ \bigg( \int_{A_t(x)\backslash V_{(t,x)}^\psi} l(x-\xi,t-s)(\sigma_s(\xi)-\sigma_s^{(\psi)}(\xi))\Lambda(d\xi,ds))\bigg)^2\Bigg]^{\frac{1}{2}}.
		\end{align*}
		Using Proposition \ref{proposition:ambitmoments} and the translation invariance of $A_t(x)$ and $V_{(t,x)}^\psi$, the above is equal to
		\begin{align*}
			&\Big(\Sigma_\Lambda  E[\sigma_0(0)^2] \int_{A_0(0) \cap V_{(0,0)}^\psi}l(-\xi,-s)^2 d\xi ds\Big)^{\frac{1}{2}}\\
			&\quad+ \Bigg( E\left[\left(\int_S\int_ {A_0^\sigma(0)\cap V_{(0,0)}^\psi} j(A,-\xi,-s) \,\Lambda^\sigma(dA,d\xi,ds) \right)^2\right]\Sigma_\Lambda \, \int_{A_0(0) \backslash V_{(0,0)}^\psi} l(-\xi ,-s)^2 d\xi ds\,\Bigg)^{\frac{1}{2}}\\
			&=\Big(\Sigma_\Lambda  E[\sigma_0(0)^2] \int_{A_0(0) \cap V_{(0,0)}^\psi}l(-\xi,-s)^2 d\xi ds\Big)^{\frac{1}{2}}\\
			&\quad+ \Bigg(\Bigg(\Sigma_{\Lambda^\sigma} \int_S\int_ {A_0^\sigma(0)\cap V_{(0,0)}^\psi} j(A,-\xi,-s)^2 d\xi ds \pi(dA)\\
			&\quad+ \mu_{\Lambda^\sigma}^2\Bigg( \int_S\int_ {A_0^\sigma(0)\cap V_{(0,0)}^\psi} j(A,-\xi,-s) \, d\xi ds \pi(dA)\Bigg)^2 \Bigg)
			\Sigma_\Lambda \, \int_{A_0(0) \backslash V_{(0,0)}^\psi} l(-\xi ,-s)^2 d\xi ds\Bigg)^{\frac{1}{2}}.
		\end{align*}
		Now let $G\in\mathcal{G}_1$ and $F\in\mathcal{G}^*_u$, i.e. $F,G$ are bounded and $G$ additionally Lipschitz-continuous for $u\in\N, h\in\R^{+}, \Gamma=\{(t_{i_1},x_{i_1}),\ldots,(t_{i_u},x_{i_u})\}\in(\R\times\R^m)^u$ and $(t_j,x_j)\in\R\times\R^m$ such that $(t_{i_1},x_{i_1}),\ldots,(t_{i_u},x_{i_u})\in V_{(t_j,x_j)}^h$. For $a\in\{1,\ldots,u\}$ define
		\begin{align*}
			&Y_{t_{i_a}}(x_{i_a})=\int_{A_{t_{i_a}}(x_{i_a})}l(x_{i_a}-\xi,t_{i_a}-s)\sigma_s(\xi) \Lambda(d \xi ,ds),\,\, \text{ and } \\
			&Y_{t_j}^{(\psi)}(x_j)=\int_{A_{t_j}(x_j)\backslash V_{(t_j,x_j)}^{\psi(h)}}l(x_j-\xi,t_j-s)\sigma_s^{(\psi)}(\xi)\Lambda(d\xi,ds).
		\end{align*}
		W.l.o.g. we assume that $(t_{i_a},x_{i_a})\leq_{lex}(t_{i_u},x_{i_u})$ for all $a\in\{1,\ldots,u\}$. Since $A_0(0)\cup A_0^\sigma(0)$ satisfy (\ref{condition:scalarproduct}) we find analogous to (\ref{equation:psi}) a function $\psi(h)= \frac{-hb}{2\sqrt{m+1}}$, such that $A_{s_1}^\sigma(\xi_1)$ and $A_{s_2}^\sigma(\xi_2)\backslash V_{(s_2,\xi_2)}^{\psi(h)}$ are disjoint for all $(s_1,\xi_1)\in A_{(t_{i_u},x_{i_u})}$ and $(s_2,\xi_2)\in  A_{(t_{j},x_{j})}\backslash V_{(t_j,x_j)}^{\psi(h)}$ or have intersection with zero Lebesgue measure. Then, by the definition of a L\'evy basis we get that $\sigma_{s_1}(\xi_1)$ and $\sigma_{s_2}^{\psi(h)}(\xi_2)$ are independent.
		Furthermore, it holds that $A_{t_{i_u}}(x_{i_u})$ and $A_{t_{j}}(x_{j})\backslash V_{(t_j,x_j)}^{\psi(h)}$ are disjoint. We set $\psi=\psi(h)$ throughout. Finally, we get that $Y_{t_{i_a}}(x_{i_a})$ and $Y_{t_j}^{(\psi(h))}(x_j)$ are independent for all $a\in\{1,\ldots,u\}$ and therefore also $F(Y_\Gamma)$ and $G(Y_{t_j}^{(\psi(h))}(x_j))$. Now
		\begin{align*}
			&|Cov(F(Y_\Gamma),G(Y_{t_j}(x_j)))| \\
			&\leq |Cov(F(X_\Gamma),G(Y_{t_j}^{(\psi(h))}(x_j)))|+|Cov(F(Y_\Gamma),G(Y_{t_j}(x_j))-G(Y_{t_j}^{(\psi(h))}(x_j)))|\\
			&=|E[(G(Y_{t_j}(x_j))-G(Y_{t_j}^{(\psi(h))}(x_j))) F(X_\Gamma)]\\
			&\quad -E[G(Y_{t_j}(x_j))-G(Y_{t_j}^{(\psi(h))}(x_j))]E[F(X_\Gamma)]|\\
			&\leq 2\norm{F}_\infty E\big[|(G(Y_{t_j}(x_j))-G(Y_{t_j}^{(\psi(h))}(x_j)))|\big]\\
			&\leq 2Lip(G) \norm{F}_\infty E[\norm{Y_{t_j}(x_j)-Y_{t_j}^{(\psi(h))}(x_j)}]\\
			&\leq 2 Lip(G)\norm{F}_\infty\Bigg(\Big(\Sigma_\Lambda  E[\sigma_0(0)^2] \int_{A_0(0) \cap V_{(0,0)}^{\psi(h)}}l(-\xi,-s)^2 d\xi ds\Big)^{\frac{1}{2}}\\
			&+ \Bigg(\Bigg(\Sigma_{\Lambda^\sigma} \int_S\int_ {A_0^\sigma(0)\cap V_{(0,0)}^{\psi(h)}} j(A,-\xi,-s)^2 d\xi ds \pi(dA)\\
			&+ \mu_{\Lambda^\sigma}^2\!\Bigg(\!\! \int_S\!\int_ {A_0^\sigma(0)\cap V_{(0,0)}^{\psi(h)}}\!\! j(A,-\xi,-s) d\xi ds \pi(dA)\Bigg)^2 \Bigg)
			\Sigma_\Lambda \!\!  \int_{A_0(0) \backslash V_{(0,0)}^{\psi(h)}}\!\!\!\!\!l(-\xi ,-s)^2 d\xi ds\Bigg)^{\frac{1}{2}}\Bigg),
		\end{align*}
		using the above inequality for $E\big[|Y_{t}(x)-Y_{t}^{(\psi(h))}(x)|\big]$. Therefore, $Y$ is $\theta$-lex weakly dependent with $\theta$-lex-coefficients
		\begin{align*}
			\theta_Y(h)&\leq 2\Bigg(\Big(\Sigma_\Lambda  E[\sigma_0(0)^2] \int_{A_0(0) \cap V_{(0,0)}^{\psi(h)}}l(-\xi,-s)^2 d\xi ds\Big)^{\frac{1}{2}}\\
			&\qquad+ \Bigg(\Bigg(\Sigma_{\Lambda^\sigma} \int_S\int_ {A_0^\sigma(0)\cap V_{(0,0)}^{\psi(h)}} j(A,-\xi,-s)^2 d\xi ds \pi(dA)\\
			&\qquad+ \mu_{\Lambda^\sigma}^2\!\!\Bigg( \int_S\int_ {A_0^\sigma(0)\cap V_{(0,0)}^{\psi(h)}}\!\! j(A,-\xi,-s) d\xi ds \pi(dA)\Bigg)^2 \Bigg)\\
			&\qquad\times\Sigma_\Lambda   \int_{A_0(0) \backslash V_{(0,0)}^{\psi(h)}}\!\!\!\!\!l(-\xi ,-s)^2 d\xi ds\Bigg)^{\frac{1}{2}}\Bigg),
		\end{align*}
		which converges to zero as $h$ goes to infinity by applying the dominated convergence theorem.\\
		\item Let $(t,x)\in\R\times\R^m$ and $\psi>0$, $Y_t^{(\psi)}(x)$ and $\tilde{Y}_t^{(\psi)}(x)$ be defined as in (i). By Proposition \ref{proposition:ambitmoments}
		\begin{align*} 
			&E\big[|Y_{t}(x)-\tilde{Y}_{t}^{(\psi)}(x)|\big] + E\big[|\tilde{Y}_{t}^{(\psi)}(x)-Y_{t}^{(\psi)}(x)|\big]\\
			&\leq\Bigg(\Big(\Sigma_\Lambda  E[\sigma_0(0)^2]
			\int_{A_t(x) \cap V_{(t,x)}^\psi}l(x-\xi,t-s)^2 d\xi ds \\
			&\quad+\mu_\Lambda E[\sigma_0(0)^2] \Big(\int_{A_t(x) \cap V_{(t,x)}^\psi}l(x-\xi,t-s) d\xi ds\Big)^2  \Bigg)^{\frac{1}{2}}\\
			&\quad+ \Bigg( \Bigg(\Sigma_{\Lambda^\sigma} \int_S\int_ {A_0^\sigma(0)\cap V_{(0,0)}^\psi} j(A,-\xi,-s)^2 d\xi ds \pi(dA)\\
			&\quad+ \mu_{\Lambda^\sigma}^2\Bigg( \int_S\int_ {A_0^\sigma(0)\cap V_{(0,0)}^\psi} j(A,-\xi,-s) d\xi ds \pi(dA)\Bigg)^2 \Bigg)\\
			&\qquad\times\Bigg( \Sigma_\Lambda   \int_{A_t(x) \backslash V_{(t,x)}^\psi}\!\!\!\!\!\!\!l(x-\xi ,t-s)^2 d\xi ds
			+ \mu_\Lambda  \Big(\int_{A_t(x) \backslash V_{(t,x)}^\psi}\!\!\!\!\!\!\!l(x-\xi ,t-s) d\xi ds\Big)^2 \Bigg)\Bigg)^{\frac{1}{2}}.
		\end{align*} 
		Finally, we can proceed as in the proof of part (i) to obtain the stated bound for the $\theta$-lex-coefficients.\\
		\item Note that the kernel function $j$ is square integrable such that $\sigma$ is well defined and stationary. Now, by Proposition \ref{proposition:MMAexistencemoments} it holds that $\sigma\in L^1(\Omega)$. Since additionally $l\in L^1(\R^m\times\R)$ and $\sigma$ is stationary it holds that $l\sigma\in L^1(\Omega\times\R^m\times\R)$. This implies that $l\sigma\in L^1(\R^m\times\R)$ almost surely. Then $l\sigma$ satisfies (\ref{equation:intcondfinvar1}) and (\ref{equation:intcondfinvar2}) almost surely and the ambit field $Y$ is well defined. Let $(t,x)\in\R\times\R^m$ and $\psi>0$ be defined as in (i). By Proposition \ref{proposition:ambitmoments}
		\begin{align*}
			&E\big[|Y_{t}(x)-\tilde{Y}_{t}^{(\psi)}(x)|\big] + E\big[|\tilde{Y}_{t}^{(\psi)}(x)-Y_{t}^{(\psi)}(x)|\big]\\
			&\leq E[|\sigma_0(0)|]\Big(|\gamma_0|+\int_{\R^d}|y|\nu(dy)\Big) \Big(\int_{A_0(0)\cap V_{(0,0)}^\psi} |l(-\xi,-s)| d\xi ds \Big)\\
			&\quad+E[|\sigma_0(0)-\sigma_0^{(\psi)}(0)|]\Big(|\gamma_0|+\int_{\R^d}|y|\nu(dy)\Big) \Big(\int_{A_0(0)\backslash V_{(0,0)}^\psi} |l(-\xi,-s)| d\xi ds \Big) .
		\end{align*}
		Finally, we obtain a bound for the $\theta$-lex-coefficients by proceeding as in (i)
	\end{enumerate}
\end{proof}

\begin{proof}[Proof of Proposition \ref{proposition:ambitmdepthetaweaklydep}]
	Let $(t,x)\in\R\times\R^{m}$ and $\psi>0$ be defined as in Proposition \ref{proposition:ambitthetaweaklydep}. We define the truncated sequence
	\begin{gather*}\label{equation:truncatedambitthetaiid}
		Y_{t}^{(\psi)}(x)=\int_{A_t(x)\backslash V_{(t,x)}^\psi}l(x-\xi,t-s)\sigma_s(\xi)\Lambda(d\xi,ds).
	\end{gather*}
	Since $l\in L^2(\R^m\times\R)$, $\sigma\in L^2(\Omega)$ and $\sigma$ is stationary, it holds that $l\sigma\in L^2(\Omega\times\R^m\times\R)$. This implies $l\sigma\in L^2(\R^m\times\R)$ almost surely. Then $l\sigma$ satisfies (\ref{equation:intcond1}), (\ref{equation:intcond2}) and (\ref{equation:intcond3}) almost surely and the ambit field $Y$ is well defined. By Proposition \ref{proposition:ambitmoments} 
	\begin{align*}
		E\big[|Y_{t}(x)-Y_{t}^{(\psi)}(x)|\big]&=E\bigg[ \Big|\int_{A_t(x)\cap V_{(t,x)}^\psi}l(x-\xi,t-s)\sigma_s(\xi)\Lambda(d\xi,ds)\Big|\bigg]\\
		&\leq E\Bigg[ \bigg( \int_{A_t(x)\cap V_{(t,x)}^\psi}l(x-\xi,t-s)\sigma_s(\xi)\Lambda(d\xi,ds)\bigg)^2\Bigg]^{\frac{1}{2}}.
	\end{align*}
	Using the translation invariance of $A_t(x)$ and $V_{(t,x)}^\psi$ this is equal to
	\begin{gather*}
		\Big(\Sigma_\Lambda E[\sigma_0(0)] \int_{A_0(0) \cap V_{(0,0)}^\psi}l(-\xi,-s)^2 d\xi ds\Big)^{\frac{1}{2}}.
	\end{gather*}
	Define $\Gamma \in (\R\times\R^m)^u$, $(t_{j},x_{j})\in\R\times\R^m$ as in the proof of Proposition \ref{proposition:ambitthetaweaklydep}. Since $\sigma$ is $p$-dependent we get that $Y_\Gamma$ and $Y_{t_j}^{(\psi)}(x_j)$ are independent for a sufficiently big $h$. Then, for sufficiently big $h$, $Y$ is $\theta$-lex weakly dependent with $\theta$-lex-coefficients
	\begin{gather*}
		\theta_Y(h)\leq 2 \Big(\Sigma_\Lambda  E[\sigma_0(0)^2] \int_{A_0(0) \cap V_{(0,0)}^{\psi(h)}}l(-\xi,-s)^2 d\xi ds\Big)^{\frac{1}{2}},
	\end{gather*}
	which converge to zero as $h$ goes to infinity by applying the dominated convergence theorem.
\end{proof}

%
%

\section*{Acknowledgements}
The authors are grateful to two anonymous referees for helpful comments, which considerably improved this work.
The third author was supported by the scholarship program of the Hanns-Seidel Foundation, funded by the Federal Ministry of Education and Research.


\begin{thebibliography}{4}
	\bibitem{A1984}
	\textsc{Andrews, D.~W.~K.} (1984).
	Non-strong mixing autoregressive processes.
	\textit{J. Appl. Probab.}
	\textbf{21} 930--934.
	
	\bibitem{BBPV2014}
	\textsc{Barndorff-Nielsen, O.~E., Benth, F.~E., and Veraart, A.~E.~D.} (2014).
	On stochastic integration for volatility modulated L\'evy-driven {V}olterra processes. \textit{Stochastic Process. Appl.}
	\textbf{124} 812--847.
	
	\bibitem{BBV2010}
	\textsc{Barndorff-Nielsen, O.~E., Benth, F.~E., and Veraart, A.~E.~D.} (2010).
	Modelling electricity forward markets by ambit fields. \textit{Adv. in Appl. Probab.}
	\textbf{46} 719--745.
	
	\bibitem{BBV2012}
	\textsc{Barndorff-Nielsen, O.~E., Benth, F.~E., and Veraart, A.~E.~D.} (2012).
	Recent advances in ambit stochastics with a view towards tempo-spatial stochastic volatility/intermittency. \textit{Banach Center Publ.}
	\textbf{104} 25--60.		
	
	\bibitem{BBV2015}
	\textsc{Barndorff-Nielsen, O.~E., Benth, F.~E., and Veraart, A.~E.~D.} (2015). 
	Cross-commodity modelling by multivariate ambit fields. In \textit{{C}ommodities, {E}nergy and {E}nvironmental {F}inance} 109--148. Springer, New York.
	
	\bibitem{BBV2018}
	\textsc{Barndorff-Nielsen, O.~E., Benth, F.~E., and Veraart, A.~E.~D.} (2018). \textit{Ambit Stochastics.}
	Springer, Cham.
	
	\bibitem{BCP2011}
	\textsc{Barndorff-Nielsen, O.~E., Corcuera, J.~M., and Podolskij, M.} (2011).
	Multipower variation for Brownian semistationary processes. \textit{Bernoulli}
	\textbf{17} 1159--1194.
	
	\bibitem{BJJS2007}
	\textsc{Barndorff-Nielsen, O.~E., Jensen, E.~B.~V., J\'onsd\'ottir, K.~Y., and Schmiegel, J.} (2007). 
	Spatio-temporal modelling - with a view to biological growth. In \textit{Statistical Methods for Spatio-Temporal Systems} 47--76. Chapman and Hall/CRC, London.
	
	\bibitem{BPS2014}
	\textsc{Barndorff-Nielsen, O.~E., Pakkanen, M.~S., and Schmiegel, J.} (2014).
	Assessing relative volatility/intermittency/energy dissipation. \textit{Electron. J. Stat.}
	\textbf{8} 1996--2021.		
	
	\bibitem{BNS2007}
	\textsc{Barndorff-Nielsen, O.~E., and Schmiegel, J.} (2007). 
	Ambit processes: with applications to turbulence and cancer growth In \textit{{S}tochastic {A}nalysis and {A}pplications: {T}he {A}bel {S}ymposium} 93--124. Springer, Berlin.
	
	\bibitem{BS2004}
	\textsc{Barndorff-Nielsen, O.~E., and Schmiegel, J.} (2004).
	L\'evy-based tempo-spatial modelling; with applications to turbulence. \textit{Uspekhi Mat. Nauk}
	\textbf{59} 63--90.
	
	\bibitem{BS2011}
	\textsc{Barndorff-Nielsen, O.~E., and Stelzer, R.} (2011).
	Multivariate sup{OU} processes. \textit{Ann. Appl. Probab.}
	\textbf{21} 140--182.
	
	\bibitem{BGP2013}
	\textsc{Basse-O'Connor, A., Graversen, S.-E., and Pedersen, J.} (2013).
	Stochastic integration on the real line. \textit{Theory Probab. Appl.}
	\textbf{58} 193--215.
	
	\bibitem{BHP2018}
	\textsc{Basse-O'Connor, A., Heinrich, C., and Podolskij, M.} (2018).
	On limit theory for L\'evy semi-stationary processes. \textit{Bernoulli}
	\textbf{24} 3117--3146.
	
	\bibitem{B2019}
	\textsc{Berger, D.} (2019).
	Central limit theorems for moving average random fields with non-random and random sampling on lattices. arXiv:1902.01255v1.
	
	\bibitem{B2019b}
	\textsc{Berger, D.} (2019).
	{L}\'evy driven CARMA generalized processes and stochastic partial differential equations. \textit{Stochastic Process. Appl.}
	\textbf{130} 5865--5887.
	
	\bibitem{BW1971}
	\textsc{Bickel, P.~J. and Wichura M.~J.} (1971).
	Convergence Criteria for Multiparameter Stochastic Processes and Some Applications. \textit{Ann. Math. Statist.}
	\textbf{42} 1656--1670.
	
	\bibitem{B1982}
	\textsc{Bolthausen, E.} (1982).
	On the central limit theorem for stationary mixing random fields. \textit{Ann. Probab.}
	\textbf{10} 1047--1050.
	
	\bibitem{BV2004}
	\textsc{Boyd, S., and Vandenberghe, L.} (2004). \textit{Convex Optimization.}
	Cambridge Univ. Press, Cambridge.
	
	\bibitem{B1989}
	\textsc{Bradley, R.~C.} (1989).
	A caution on mixing conditions for random fields. \textit{Statist. Probab. Lett.}
	\textbf{8} 489--491.
	
	\bibitem{B2007}
	\textsc{Bradley, R.~C.} (2007). \textit{Introduction to strong mixing conditions. Vol. 3.}
	Kendrick Press, Utah.
	
	\bibitem{BCS2020}
	\textsc{Brandes, D.-P., Curato, I.~V., and Stelzer R.} (2021).
	Inheritance of strong mixing and weak dependence under renewal sampling.
	arXiv:1902.01255v1.
	
	\bibitem{B2014}
	\textsc{Brockwell, P.~J.} (2014).
	Recent results in the theory and applications of {CARMA} processes. \textit{Ann. Inst. Statist. Math.}
	\textbf{66} 647--685.
	
	\bibitem{BM2017}
	\textsc{Brockwell, P.~J., and Matsuda, Y.} (2017).
	Continuous auto-regressive moving average random fields on $\R^n$. \textit{J. R. Stat. Soc. Ser. B Stat. Methodol.}
	\textbf{79} 833--857.
	
	\bibitem{BS2007}
	\textsc{Bulinskii, A.~V., and Shashkin, A.} (2007). \textit{Limit theorems for associated random fields and related systems.}
	World Scientific, Singapore.
	
	\bibitem{C1991}
	\textsc{Chen, D.} (1991).
	A uniform central limit theorem for nonuniform $\phi$-mixing random fields. \textit{Ann. Probab.}
	\textbf{19} 636--649.
	
	\bibitem{CK2015}
	\textsc{Chong, C. and Kl{\"u}ppelberg, C.} (2015).
	Integrability conditions for space-time stochastic integrals: Theory and applications. \textit{Bernoulli}
	\textbf{21} 2190--2216.
	
	\bibitem{CS2018}
	\textsc{Curato, I.~V, and Stelzer, R.} (2019).
	Weak dependence and GMM estimation of supOU and mixed moving average processes. \textit{Electron. J. Stat.}
	\textbf{13} 310--360.
	
	\bibitem{D1998}
	\textsc{Dedecker, J.} (1998).
	A central limit theorem for stationary random fields. \textit{Probab. Theory Related Fields}
	\textbf{110} 310--426.
	
	\bibitem{D2001}
	\textsc{Dedecker, J.} (2001).
	Exponential inequalities and functional central limit theorems for random fields. \textit{ESAIM Probab. Stat.}
	\textbf{5} 77--104.
	
	\bibitem{DD2003}
	\textsc{Dedecker, J., and Doukhan, P.} (2003).
	A new covariance inequality and applications. \textit{Stochastic Process. Appl.}
	\textbf{106} 63--80.
	
	\bibitem{DDLLLP2008}
	\textsc{Dedecker, J., Doukhan, P., Lang, G., L\'eon, J.~R., Louhichi, S., and Prieur, C.} (2008). \textit{Weak dependence: with examples and applications.}
	Springer, New York.
	
	\bibitem{DM2002}
	\textsc{Dedecker, J., and Merlev\`ede, F.} (2002).
	Necessary and sufficient conditions for the conditional central limit theorem. \textit{Ann. Probab.}
	\textbf{30} 1044--1081.
	
	\bibitem{DR2000}
	\textsc{Dedecker, J., and Rio, E.} (2000).
	On the functional central limit theorem for stationary processes. \textit{Ann. Inst. Henri Poincar\'e Probab. Stat.}
	\textbf{36} 1--34.
	
	\bibitem{DFL2012}
	\textsc{Doukhan, P., Fokianos, K. and Li, X.} (2012).
	On weak dependence conditions: The case of discrete valued processes. \textit{Statist. Probab. Lett.}
	\textbf{84} 1941--1948.
	
	\bibitem{DL1999}
	\textsc{Doukhan, P., and Louhichi, S.} (1999).
	A new weak dependence condition and applications to moment inequalities. \textit{Stochastic Process. Appl.}
	\textbf{82} 313--342.
	
	\bibitem{DMT2008}
	\textsc{Doukhan, P., Mayo, N., and Truquet, L.} (2008).
	Weak dependence, models and some applications. \textit{Metrika}
	\textbf{69} 199--225.
	
	\bibitem{DT2007}
	\textsc{Doukhan, P., and Truquet, L.} (2007).
	A fixed point approach to model random fields. \textit{ALEA Lat. Am. J. Probab. Math. Stat.}
	\textbf{3} 111--137.
	
	\bibitem{DW2007}
	\textsc{Doukhan, P., and Wintenberger, O.} (2007).
	An invariance principle for weakly dependent stationary general models. \textit{Probab. Math. Statist.}
	\textbf{27} 45--73.
	
	\bibitem{G1969}
	\textsc{Gordin, M.~I.} (1969).
	The central limit theorem for stationary processes. \textit{Dokl. Akad. Nauk}
	\textbf{188} 739--741.
	
	\bibitem{H2005}
	\textsc{Hall, A.~R.} (2005). \textit{Generalized method of moments.}
	Oxford Univ. Press, Oxford.
	
	\bibitem{H2002}
	\textsc{Higdon, D.} (2002). 
	Space and space-time modeling using process convolutions. In \textit{Quantitative Methods for Current Environmental Issues} 37--56. Springer, London.
	
	\bibitem{IL89} 
	\textsc{Ivanov, A.~V. and Leonenko, N.~N.} (1989).
	\textit{Statistical analysis of random fields.}
	Kluwer Academic Publishers, Dodrecht.
	
	\bibitem{JS2003}
	\textsc{Jacod, J., and Shiryaev, A.~N.} (2003). \textit{Limit theorems for stochastic processes.}
	Springer, Berlin.
	
	\bibitem{JN2013}
	\textsc{J\`onsd\`ottir, K.~Y., R{\o}nn-Nielsen, A., Mouridsen, K. and Jensen, E.~B.~V.} (2013).
	L\'evy-based modelling in brain imaging. \textit{Scand. J. Stat.}
	\textbf{40} 511--529.
	
	\bibitem{K1985}
	\textsc{Krengel, U.} (1985). \textit{Ergodic theorems.}
	Walter de Gruyter, Berlin.
	
	\bibitem{KP2019}
	\textsc{Kl{\"u}ppelberg, C., and Pham, V.~S.} (2019).
	Estimation of causal CARMA random fields. arXiv:1902.04962v1.
	
	\bibitem{M1999}
	\textsc{Maltz, A.~L.} (1999).
	On the central limit theorem for nonuniform $\phi$-mixing random fields. \textit{J. Theoret. Probab.}
	\textbf{12} 643--660.
	
	\bibitem{N1988}
	\textsc{Nakhapetyan, B.~S.} (1988).
	An approach to proving limit theorems for dependent random variables. \textit{Theory Probab. Appl.}
	\textbf{32} 535--539.
	
	\bibitem{N1980}
	\textsc{Newman, C.~M.} (1980).
	Normal fluctuations and the FKG inequalities. \textit{Comm. Math. Phys.}
	\textbf{74} 119--128.
	
	\bibitem{NV2015}
	\textsc{Nguyen, M., and Veraart, A.~E.~D.} (2017).
	Spatio-temporal Ornstein-Uhlenbeck processes: theory, simulation and statistical inference. \textit{Scand. J. Stat.}
	\textbf{44} 46--80.
	
	\bibitem{NV2017}
	\textsc{Nguyen, M., and Veraart, A.~E.~D.} (2018).
	Bridging between short-range and long-range dependence with mixed spatio-temporal Ornstein-Uhlenbeck processes. \textit{Stochastics}
	\textbf{90} 1023--1052.
	
	\bibitem{P2014}
	\textsc{Pakkanen, M.~S.} (2014).
	Limit theorems for power variations of ambit fields driven by white noise. \textit{Stochastic Process. Appl.}
	\textbf{124} 1942--1973.
	
	\bibitem{PV2017}
	\textsc{Passeggeri, P., and Veraart, A.~E.~D.} (2019).
	Mixing properties of multivariate infinitely divisible random fields. \textit{J. Theoret. Probab.}
	\textbf{32} 1845--1879.
	
	\bibitem{P2003}
	\textsc{Pedersen, J.} (2003).
	The {L}\'evy-{I}t\^o decomposition of an independently scattered random measure. \textit{MaPhySto Research Paper 2}, available at \url{http://www.maphysto.dk}.
	
	\bibitem{PU2006}
	\textsc{Peligrad, M., and Utev, S.} (2006).
	Central limit theorem for stationary linear processes. \textit{Ann. Probab.}
	\textbf{34} 1608--1622.
	
	\bibitem{P2018}
	\textsc{Pham, V.~S.} (2020).
	L\'evy-driven causal CARMA random fields. \textit{Stochastic Process. Appl.} \textbf{130} 7547--7574.
	
	\bibitem{RR1989}
	\textsc{Rajput, B.~S., and Rosi\'nski, J.} (1989).
	Spectral representations of infinitely divisible processes. \textit{Probab. Theory Related Fields}
	\textbf{82} 451--487.
	
	\bibitem{R1956}
	\textsc{Rosenblatt, M.} (1956).
	A central limit theorem and a strong mixing condition. \textit{Proc. Natl. Acad. Sci. USA}
	\textbf{42} 43--47.
	
	\bibitem{R1985}
	\textsc{Rosenblatt, M.} (1985). \textit{Stationary sequences and random fields.}
	Birkh\"auser, Basel.
	
	\bibitem{S2013}
	\textsc{Sato, K.~I.} (2013). \textit{L\'evy processes and infinitely divisible distributions, Cambridge Studies in Advanced Mathematics 68.}
	Cambridge Univ. Press, Cambridge.
	
	\bibitem{STW2015}
	\textsc{Stelzer, R., Tosstorff, T., and Wittilinger, M.} (2015).
	Moment based estimation of sup{OU} processes and a related stochastic volatility model. \textit{Stat. Risk Model.}
	\textbf{32} 1--24.
\end{thebibliography}
\end{document}